\documentclass[11pt, a4paper]{amsart}

\usepackage{amsfonts,amsmath,amssymb,amscd}

\usepackage[all]{xy}
\usepackage{graphicx}
\usepackage{stmaryrd}

\newtheorem{theorem}{Theorem}[section]
\newtheorem{lemma}[theorem]{Lemma}
\newtheorem{prop}[theorem]{Proposition}
\newtheorem{corollary}[theorem]{Corollary}

\theoremstyle{definition}
\newtheorem{definition}[theorem]{Definition}
\newtheorem{rem}[theorem]{Remark}

\newtheorem{example}[theorem]{Example}

\usepackage{mathrsfs}
\usepackage{color}
\newcommand{\tc}[1]{\textcolor{black}{#1}}

\newcommand\pf{\begin{proof}}
\newcommand\epf{\end{proof}}
\newcommand\eq{\begin{equation}}
\newcommand\eeq{\end{equation}}
\newcommand\tu{\textup}

\newcommand{\Db}{D^{\mathrm{b}}}
\newcommand{\M}{\mathtt{SMod}}
\newcommand{\MAK}{{}_D\mathtt{Smod}_{(A|K)}}
\renewcommand{\a}{\mathfrak{a}}
\renewcommand{\b}{\mathfrak{b}}
\renewcommand{\u}{\boldsymbol{u}}
\renewcommand{\v}{\boldsymbol{v}}
\newcommand{\g}{\mathfrak{g}}
\newcommand{\I}{\mathfrak{I}}

\newcommand{\ot}{\otimes}
\newcommand{\op}{\operatorname}

\newcommand{\F}{\mathsf{F}}
\newcommand{\G}{\mathsf{G}}

\newcommand{\Quot}{\mathsf{Quot}}
\newcommand{\X}{\mathsf{X}}
\newcommand{\Y}{\mathsf{Y}}

\newcommand{\cA}{\mathcal{A}}
\newcommand{\cB}{\mathcal{B}}
\newcommand{\cI}{\mathcal{I}}
\newcommand{\cJ}{\mathcal{J}}
\newcommand{\cO}{\mathcal{O}}

\newcommand{\cS}{\mathcal{S}}
\newcommand{\cT}{\mathcal{T}}

\renewcommand{\L}{\overline{L}}
\newcommand{\K}{\overline{K}}
\newcommand{\D}{\underline{D}}
\newcommand{\otk}{\otimes_k}
\newcommand{\red}{_{\mathrm{red}}}
\newcommand{\Delb}{\Delta\! ^{\mathrm{b}}}
\newcommand{\eb}{\varepsilon^{\mathrm{b}}}
\newcommand{\Sb}{\mathcal{S}^{\mathrm{b}}}
\newcommand{\Sbi}{\mathcal{S}^{\mathrm{b}-}}

\let\oldtocsection=\tocsection

\let\oldtocsubsection=\tocsubsection

\renewcommand{\tocsection}[2]{\hspace{0em}\oldtocsection{#1}{#2}}

\renewcommand{\tocsubsection}[2]{\hspace{1em}\oldtocsubsection{#1}{#2}}

\numberwithin{equation}{section}

\title[Supersymmetric Picard-Vessiot Theory, I]
{Supersymmetric Picard-Vessiot Theory, I:\\ Basic Theory}

\author[A.~Masuoka]{Akira Masuoka}
\address{Akira Masuoka,
Department of Mathematics, 
University of Tsukuba, 
Ibaraki 305-8571, Japan}
\email{akira@math.tsukuba.ac.jp}

\begin{document}

\begin{abstract}
M.~Takeuchi (1989) proposed a Hopf-algebraic approach to Picard-Vessiot (or PV) theory, 
giving a new definition of PV extensions by which such extensions become
more smoothly connected, through Hopf-Galois
extensions, to the associated affine group schemes. 
This paper extends PV theory to the supersymmetric (or SUSY) context, following Takeuchi's approach.
The notion of SUSY fields is defined. Differential fields in the existing PV theory 
are replaced by $D$-SUSY fields, where $D$ is an acting super-cocommutative Hopf superalgebra. 
Hopf-Galis extensions here play a more and more crucial role. 
\end{abstract}

\maketitle

\tableofcontents

\medskip



\section{Introduction}\label{secINT}

\subsection{Background}\label{INTbackground}
Picard-Vessiot (or PV, for short) theory, 
having originated around 1900,
is a Galois theory for
linear differential equations (see standard
textbooks, such as Kolchin \cite{Kolchin}, Kaplansky
\cite{Kaplansky}, or van der Put and Singer \cite{PutSinger}),
which has remarkable applications, for examples, in  
showing criteria of irreducibility of generalized hypergeometric
differential equations \cite{Beukers} or
obstructions of integrability of Hamiltonian systems
\cite{MoralesRuiz}.

As a standard, modern setup, 
the theory discusses differential modules over a
differential field.
A {\em differential field} $K$ is a field, here supposed to be of
characteristic zero, which is equipped with a set, say $\varDelta$, of 
differential operators, and a {\em differential $K$-module} $V$ is a $K$-vector
space, here supposed to be of finite dimension,  on which each element
$\partial$ of $\varDelta$ acts additively so that $\partial(xv)=(\partial x)v=x(\partial v)$ for $x \in K$ and $v \in V$. 
Let $L|K$ be an extension 
(or namely, an inclusion $L\supset K$) of differential fields.
We say that $L$ is a {\em (minimal) splitting differential field}
for $V$, if the base extension $L\otimes_K V$ is isomorphic
to the direct sum $L\oplus \dots\oplus L$ of some copies of the trivial differential $L$-module $L$ 
(and if it is minimal with respect to the property). 
Roughly (though slightly) speaking, such an extension
$L|K$ satisfying the minimality condition is called a {\em PV extension}, and the automorphism group $\G(L|K)$ is seen (far from 
obviously!) to have {\em uniquely} a structure of a linear algebraic group; it is called the {\em differential Galois group} of $L|K$. A Galois correspondence 
is shown between the closed subgroups of $\G(L|K)$ and the intermediate 
differential fields of $L|K$. Moreover, as an issue of solvability of 
the differential module in question, 
it is proved that 
the identity component of $\G(L|K)$ is solvable if and only if 
$L|K$ is of Liouville-type. 

A Hopf-algebraic approach to the theory was proposed by M. Takeuchi
\cite{T}
in 1989, which has two features: (1)~Differential fields and
modules are generalized by $C$-ferential fields and modules,
with $\varDelta$ replaced by a cocommutative coalgebra $C$ given a
specific grouplike element $1_C$ which is supposed to always
act as an identity. The 
conventional, characteristic-zero assumption
is thereby removed; (2)~The notion of PV extensions is newly defined so that
they look independent of $C$-ferential modules, and naturally give rise
to linear algebraic groups, or more precisely, affine group schemes.
Separately, it is proved that an extension $L|K$ of $C$-ferrential 
fields, which we here suppose to be finitely generated as a field extension, is PV if and only if $L$ is 
a minimal splitting $C$-ferential field of some $C$-ferential
$K$-module. 
To recall from \cite[Definition 2.3]{T} Takeuchi's new definition, let $L|K$ be
an extension of $C$-ferential fields with the same field
$L^C=K^C$ of $C$-invariants (cf.\eqref{INTeq3}), which we denote by 
$k$. Then $L|K$
is called a {\em PV extension}, if there exists ({\em uniquely}, as is proved by \cite[Lemma 2.5]{T})
an intermediate $C$-ferential algebra $A$ of $L|K$ such that (i)~the algebra  
\eq\label{INTeq0}
H:=(A\otimes_K A)^C
\eeq
of the $C$-invariants in $A\otimes_KA$ generates the left (or right) $A$-module $A\otimes_KA$, and (ii)~$L$ coincides with
the quotient field of $A$. 

Then $H$ and $A$ are proved (see \cite[Lemma 2.4]{T}, where some of the words found in the following are not used) to uniquely 
turn into
a commutative Hopf algebra over $k$ and an $H$-comodule algebra, respectively, so that
$A|K$ is an {\em $H$-Galois extension}. This is translated into
the language of schemes (from functorial view-point) as follows: 
$\G:=\operatorname{Sp}_k(H)$ is an affine $k$-group scheme, and
$\operatorname{Sp}_k(A)$ is a {\em $\G$-torsor} over 
$\operatorname{Sp}_k(K)$. This $\G$ is naturally 
regarded as the automorphism $k$-group functor of the 
$C$-ferential algebra $A$ (not $L$) over $K$, and 
under some appropriate assumptions, the group $\G(k)$ of $k$-points
coincides with what should be called the {\em $C$-ferential 
Galous group}
of $L|K$ in accordance with the standard, differential situation 
overviewed before,
that is, the linear algebraic group of $C$-ferential 
automorphisms of $L|K$. 

In view of Takeuchi's new definition above, the author thinks 
that his motivation
was in smoothing a connection of PV extensions with
linear algebraic groups, which was naturally replaced with affine group
schemes in his context. Such a smooth connection was achieved through
Hopf-Galois extensions or torsors, 
though these words were not used in \cite{T}. 

Takeuchi's theory was later extended by
Amano and the author \cite{AM} (see also \cite{AMT}, \cite{A}) 
so that  
difference algebras and modules as well can be treated on the same stage.
The present paper extends his theory in another way,
namely,  
to the supersymmmetric (or SUSY) context, in which everything built 
on the obvious 
symmetry $V\otimes W\overset{\simeq}{\longrightarrow}W\otimes V$,
$v\otimes w\mapsto w\otimes v$ of vector spaces is
replaced with the more general, super analogue built
on the supersymmetry found in \eqref{INTeq1} below.

\subsection{The setup}\label{INTframework}
Throughout in what follows, we work over a fixed field $R$ of characteristic zero. This assumption on the characteristic is needed
to ensure smoothness of objects at several places. 
A \emph{super-vector space} (over $R$) is a synonym of a vector space $V=V_0\oplus V_1$
graded by the group $\mathbb{Z}/(2)=\{ 0,1\}$ of order 2. Those super-vector spaces 
$V, W,\dots$ form a symmetric monoidal category with respect to the tensor product 
$\ot\, (=\ot_R)$ over $R$, the unit object $R$ and the supersymmetry
\begin{equation}\label{INTeq1}
c_{V,W}:V\ot W \overset{\simeq}{\longrightarrow} W \ot V,\quad
c_{V,W}(v \ot w)=(-1)^{|v| |w|}w\ot v. 
\end{equation}
Here and in what follows, $|v|\, (\in \{0,1\})$ denotes the degree of $v$, and the element $v$ is supposed, by convention, 
to be homogeneous when we mention its degree. The elements $0$, $1$ 
in $\mathbb{Z}/(2)$ are referred to as \emph{even} and \emph{odd}, respectively, and a super-vector space $V\, (=V_0\oplus V_1)$ is said to be \emph{purely even}
(resp., \emph{purely odd}) if $V=V_0$ (resp., $V=V_1$). 
Basically, objects, such as algebra, Hopf algebra or Lie algebra, defined in the 
category are called with ``super" attached, so as superalgebra, Hopf superalgebra or Lie superalgebra. 
Superalgebras are supposed to be \emph{super-commutative}, $ab=(-1)^{|a||b|}ba$, unless otherwise stated. Hopf superalgebras 
are as well. An important exception is the Hopf superalgebra $D$ which plays the role of the set of differential operators
in the standard PV theory. 

Indeed, $D$ denotes throughout a fixed super-cocommutative Hopf superalgebra which may not be super-commutative.
\tc{The primitive elements of $D$ naturally forms a Lie superalgebra, which we denote by $\g_D$.}
In this Introduction we assume, as we will 
do in most parts of the paper,
that $D$ is {\em pointed}, and the odd component $(\g_D)_1$ of $\g_D$
is finite-dimensional. 
The first assumption means that $D$ is the semi-direct product
\eq\label{INTeq2}
D=U(\g_D)\rtimes \Gamma_D
\eeq
of $U(\g_D)$ by $\Gamma_D$, where $U(\g_D)$ denotes the universal envelope of \tc{the Lie superalgebra $\g_D$}, and $\Gamma_D$ denotes
the group consisting of all grouplike \tc{(necessarily, even)} elements of $D$. 
The left $D$-supermodules form a 
symmetric monoidal category, 
\[
{}_D\M=({}_D\M, \ot, R), 
\]
and algebras in this category are called
$D$-\emph{superalgebras}. Given such an algebra $K$, $K$-modules in 
${}_D\M$ correspond to 
differential or $C$-ferential modules in the ordinary, non-super
context. 

Recall that the existing PV theory discusses extensions of 
differential or $C$-ferential \emph{fields}. 
So, there arises the following.

\medskip

\noindent
{\bf Question.}\
 \emph{What is an appropriate notion in the super context which generalizes fields?}

\medskip

Our answer is the notion of \emph{SUSY fields}. Before introducing it from Definition \ref{SUFdef1},
we remark that given a superalgebra $K$, the super-ideal $I_K:=(K_1)$ generated by the odd component $K_1$ of $K$
gives the largest purely even quotient superalgebra $\overline{K}:=K/I_K$ of $K$,
and, moreover, the graded algebra 
\[
\operatorname{gr}K:=\bigoplus_{n=0}^{\infty}I_K^n/I_K^{n+1}\, (=\overline{K}\oplus I_K/I_K^2\oplus...),
\]
which is in fact a superalgebra by mod 2 grading. It is known (see Proposition \ref{SUFprop1}) that
a Noetherian superalgebra $K$ is smooth (over $R$) if and only if $\overline{K}$ is smooth, $I_K/I_K^2$ is (finitely generated)
projective as a $\overline{K}$-module and there is a {\em non-canonical} isomorphism $\wedge_{\overline{K}}(I_K/I_K^2)\simeq K$,
which can be chosen so that it turns into a {\em canonical} isomorphism with $\operatorname{gr}$ applied. 
Almost all superalgebras that will appear in this paper are Noetherian and smooth. 
Important in our discussion is that isomorphisms such as above are ``pre-canonical"; this means that
they are non-canonical, but become canonical after $\operatorname{gr}$. 

A {\em SUSY field} is defined to be a Noetherian smooth algebra $K$ such that $\overline{K}$ is a field,
so that necessarily, it is non-canonically isomorphic to a naturally associated, finite-dimensional exterior algebra over the field $\overline{K}$. 
As a typical example, given a locally Noetherian and smooth, integral superscheme $\X$, the stalk $\cO_{\X,\eta}$ of the structure sheaf 
$\cO_{\X}$ at the generic point $\eta$ is a SUSY field, which we call {\em the function SUSY field} of $\X$; see Definition \ref{SSCdef1}. 
A SUSY field is called a {\em $D$-SUSY field}, if it is a $D$-superalgebra at the same time. We will often assume that a $D$-SUSY field in question
is {\em $D$-simple}, or namely, it include no non-trivial $D$-stable super-ideal. The $D$-simplicity ensures that
the set 
\eq\label{INTeq3}
K^D:=\{\, x \in K\mid dx =\varepsilon(d)x\ \, \text{for all}\ \, d \in D\, \},
\eeq
of all {\em $D$-invariants} in $K$, 
where $\varepsilon$ denotes the counit of $D$, 
is a field. 
Suppose that $K$ is a $D$-simple $D$-SUSY field. The role of differential modules 
over a differential field in the standard PV theory is here played
by those $K$-modules in ${}_D\M$ 
which are {\em $K$-finite free} (see Definition \ref{SASdef1}), or namely, which have $K$-free bases consisting of finitely many homogeneous elements.
Giving a $K$-free finite $K$-module in ${}_D\M$ with \tc{$m$ even and $n$ odd} $K$-free basis elements is equivalent to giving an equation
\eq\label{INTeq3a}
d\boldsymbol{y}=F(d)\boldsymbol{y},\quad d\in D,
\eeq
where 
\tc{$\boldsymbol{y}={}^t\big( y_1,\dots,y_{m+n}\big)$
is the column vector with entries
even variables $y_1,\dots, y_m$ and odd variables $y_{m+1},\dots,y_{m+n}$},\  
and $F(d)$ is a matrix in $\mathsf{M}_{m|n}(K)$ (see Example \ref{SASex1}) satisfying \eqref{UEXeq2}. 
As far as the author understands, every variational equation that is linearized from any super-differential
equation discussed by \cite{MS} (see also \cite{GW}) is necessarily of the form \eqref{INTeq3a} 
in the case where $D=U(\g_D)$, and $(\g_D)_0$ and $(\g_D)_1$ are both one-dimensional.

\subsection{Main results}\label{INTmainresults}
Our main results are quite parallel to those of \cite{T},
though many of the proofs become more difficult. 
The results are proved in Sections \ref{GAL}--\ref{secUEX}, and
the preceding Sections \ref{secPRE}--\ref{secSPV} are devoted to 
preliminaries.

In this subsection we fix a $D$-simple $D$-SUSY field
$K$, and let $k:=K^D$ denote the field of $D$-invariants; this $k$, 
of course, includes the base field $R$. 

Suppose that $L|K$ is an extension of $D$-simple $D$-SUSY fields.
There is naturally associated 
an extension $\L|\K$ of $\D$-fields, where $\D$ denotes 
the cocommutative
Hopf algebra $U(\g_0)\rtimes \Gamma_D$ (cf. \eqref{INTeq2}), and
a {\em $\D$-field} is a $\D$-module algebra which is a field. 
By Definition \ref{DEFdef1}, $L|K$ is called a 
{\em SUSY PV extension},
if (1)~$L^{D}=k\, (=K^D)$, (2)~$\L|\K$ is finitely generated as
a field extension,
and (3)~there exists (necessarily, uniquely)
an intermediate $D$-superalgebra $A$ of $L|K$ such that (i)~$H:=(A\otimes_KA)^D$ (cf. \eqref{INTeq1}) generates the left (or right) $A$-supermodule $A\otimes_KA$, and (ii)~$L$ coincides with the localization 
$\Quot(A)$ of $A$ by the multiplicative subset of $A_0$ consisting of all $A$-regular elements. 
In this case, analogously as in \cite{T}, $H$ and $A$ uniquely turn into a finitely
generated Hopf superalgebra over $k$, and an $H$-supercomodule 
superalgebra, respectively, so that $A|K$ is an $H$-Galois extension,
or in other words, $\operatorname{Sp}_k(A)\to \operatorname{Sp}_k(K)$ 
is a super $\operatorname{Sp}_k(H)$-torsor; this is much more difficult
to prove than the corresponding one in \cite{T}, due to the fact that extensions
of SUSY fields are not so simple; see Example \ref{SUFex0}).
Indeed, almost the
whole of Section \ref{secSPV} is devoted to the proof, in which some results 
(Proposition \ref{HGEprop1}, Theorem \ref{HGEthm1}) reproduced from the recent article \cite{MOT} 
play a crucial role. 
We let
\[
\G(L|K)=\operatorname{Sp}_k(H)
\]
denote the affine algebraic $k$-supergroup scheme represented by $H$,
and call it {\em the Galois supergroup} of $L|K$.

The first main result is the following.
\medskip

\noindent
{\bf Galois correspondence}~(Theorems \ref{GALthm1} and \ref{GALthm2}).\  \emph{Given a SUSY PV extension $L|K$, there is a one-to-one correspondence
between the closed sub-supergroup schemes of $\G(L|K)$ and
the intermediate $D$-SUSY fields of $L|K$, which restricts to 
a one-to-one correspondence between the closed normal sub-supergroup
schemes of $\G(L|K)$ and the intermediate $D$-SUSY fields $M$ of $L|K$
such that $M|K$ is a SUSY PV extension.}

\medskip

In the situation above, notice that $\G(L|K)$ acts on the affine superscheme $\X:=\operatorname{Sp}_k(A)$, where $A$ is such as in the definition of SUSY PV extensions.
Given a closed sub-supergroup scheme $\F$ of $\G(L|K)$, the quotient fppf sheaf $\X\tilde{/}\F$ of $\X$
by the restricted action by $\F$ is proved to be a Noetherian and smooth, integral superscheme. The intermediate $D$-SUSY field 
of $L|K$ corresponding to $\F$ is the function SUSY field of $\X\tilde{/}\F$. 

Let $V$ be a $K$-module in ${}_D\M$, which is $K$-finite free so that
$V \simeq K^{m|n}$ as $K$-supermodules for uniquely determined 
non-negative integers $m$, $n$. 
Here, $K^{m|n}$ denotes the direct sum
of $m$-copies of $K$ and of $n$-copies of the degree shift $K[1]$ of 
$K$.
An extension $L|K$ of $D$-simple $D$-SUSY
fields is called a {\em (minimal) splitting $D$-SUSY field} for $V$ (see Definition
\ref{SPDdef1}), if 
$L\ot_KV\simeq L^{m|n}$ as $L$-modules in ${}_D\M$ (and if $L$ is
minimal with respect to the property, roughly speaking).

The following is the second main result, which 
would justify our definition of SUSY PV extensions, showing
that they are precisely the objects which behave as expected from the traditional definition 
 (see Section \ref{INTbackground}) of PV extensions.

\medskip

\noindent
{\bf Characterization}~(Theorem \ref{CHAthm1}).\ 
\emph{Let $L|K$ be an extension of $D$-simple $D$-SUSY fields such that
$L^D=K^D$. Then $L|K$ is a SUSY PV extension if and only if there exists
a $K$-finite free $K$-module $V$ in ${}_D\M$ such that $L$ is a
minimal splitting $D$-SUSY field for $V$.}

\medskip

Suppose that $L|K$ is a SUSY PV extension, whence we have
$L^D=k\, (=K^D)$. 
The finite-dimensional
$\G(L|K)$-supermodules form a $k$-linear abelian, rigid symmetric
monoidal category, which we denote by ${}_{\G(L|K)}\mathtt{Smod}$.
The so-called Hopf-module theorem (Theorem \ref{HGEthm0}), 
a fundamental theorem in Hopf algebra theory, 
formulated in the super context proves
a Tannaka-type theorem (Theorem \ref{TNKthm1}) which states 
that ${}_{\G(L|K)}\mathtt{Smod}$ is equivalent, as 
a monoidal category of the described sort, to
a certain monoidal full subcategory 
of the category ${}_K({}_D\M)$ of $K$-modules in ${}_D\M$. 
By the Characterization above, $L$ is a minimal splitting $D$-SUSY field for some $V$. We prove that the last monoidal full subcategory 
of ${}_K({}_D\M)$ contains $V$, and is generated by $V$, whence
it may be denoted by $\langle\langle V\rangle\rangle_{\otimes}$. 
Thus we have the familiar presentation (cf. \cite{Del})
\[
\langle\langle V\rangle\rangle_{\otimes}\approx {}_{\G(L|K)}\mathtt{Smod}
\]
of the Tannaka-type theorem in PV theory; see Corollary \ref{TNKcor1}. 

The following is the last main result, which is,
among the results presented here, 
probably most important with regard to
actual applications.
\medskip

\noindent
{\bf Unique existence}~(Theorem \ref{UEXthm1}).\ 
\emph{Assume that $D$ is irreducible, or explicitly, $D=U(\g_D)$; see
\eqref{INTeq2}. Let $K$ be a $D$-simple $D$-SUSY field such that the field $K^D$ of $D$-invariants in $K$ is algebraically
closed. Then, given a $K$-finite free $K$-module $V$ in ${}_D\M$, there exists a $D$-simple minimal-splitting
$D$-SUSY field $L$ for $V$ such that $L^D=K^D$. Such an $L$ is unique up to $K$-algebra isomorphism in ${}_D\M$. }

\medskip

Going back to the situation of the Characterization, suppose that $L$ is a minimal splitting $D$-SUSY field for $V$. 
One sees that $\widetilde{V}:=V/K_1V$ naturally turns into a 
(non-super) object 
discussed by the existing PV theory, or to be more precise, 
a $\K$-module in the category ${}_{\D}\mathtt{Mod}$ of $\D$-modules.
It is proved
(see Proposition \ref{ASSprop1}) that the $\D$-field $\L$ associated with $L$ is a {\em minimal splitting
$\D$-field} for $\widetilde{V}$ in the sense of the existing theory. 
The continuation \cite{MasuokaII} of the present paper will discuss,
among others, 
the solvability of $V$ in relation with 
that of $\widetilde{V}$. 


\section{Basics on super algebra/algebraic geometry}\label{secPRE}
Recall that we work over a fixed field $R$ of characteristic zero, and 
superalgebras are supposed to be
super-commutative unless otherwise sated. 
The unadorned $\ot$ presents the tensor product over $R$. 

\subsection{Superalgebras and supermodules}\label{SAS}
Let $V$, $W$ be super-vector spaces (over $R$). 
We let $\op{Hom}(V,W)$ denote the vector space of all $R$-linear morphisms
$V \to W$  that may not preserve the parity.
In fact, this is a super-vector space with respect to the parity
\eq\label{SASeq0}
(\op{Hom}(V,W))_i=\op{Hom}(V_0,W_{0+i})\oplus \op{Hom}(V_1,W_{1+i}),\quad i=0,1.
\eeq

We let $V[1]$ denote the degree shift of $V$, while we let $V[0]=V$. Thus we have
\[
V[i]_j=V_{i+j},\quad i,j \in \{0,1\}.
\]

The \emph{dimension} $\dim V$ of $V$ is the pair of the dimensions of the homogeneous components
$V_i$, where $i=0,1$, and is presented so as
\eq\label{SASeq00-}
\dim V=\dim V_0 \mid \dim V_1.
\eeq
By $\dim V \le \dim W$, we mean $\dim V_i\le \dim W_i$ for $i=0,1$, both. 

Let $K$ be a superalgebra. Left and right $K$-supermodules are identified. 
To be more precise, the left and the right $K$-supermodule
structures on a fixed super-vector space $V$ are in one-to-one correspondence through the relation
\[
xv=(-1)^{|x||v|} vx,\quad x \in K,\ v \in V.
\]
Given a pair of such corresponding structures, $V$ turns into a $(K,K)$-super-bimodule, or namely, we have
$(xv)y=x(vy)$, where $x,y\in K$, $v \in V$. 
Let 
${}_K\M$ (resp., $\M_K)$ denote the abelian category of left (resp., right) $K$-supermodules. The categories are identified so that
\begin{equation}\label{SASeq00}
{}_K\M=\M_K,
\end{equation}
as was just seen. We will not specify the side, just saying $K$-supermodules, except when we need to make it clear. 

Notice that ${}_K\M$ forms a monoidal category
\begin{equation}\label{SASeq0a}
({}_K\M,\otimes_K,K).
\end{equation}
Here, the tensor product $V\ot_KW$ of two objects is defined as the co-equalizer 
\eq\label{SASeq0b}
V\ot K \ot W
\rightrightarrows
V\ot W \to V\ot_K W 
\eeq
of the right $K$-action on $V$ and the left $K$-action on $W$; we suppose that $K$ acts on $V\ot_KW$ 
on the left through the tensor factor $V$, or equally, on the right through the tensor factor $W$. 
Moreover, the monoidal category is symmetric with respect to the supersymmetry $V\ot_KW \overset{\simeq}{\longrightarrow}
W\ot_KV$, or to be more precise, the isomorphism induced from the supersymmetry $c_{V,W}$; see \eqref{INTeq1}.

\begin{lemma}[\tu{\cite[Lemma 5.1 (1)]{M}}]\label{SASlem1}
A $K$-supermodule $V$, regarded as an ordinary left or right module over the algebra $K$, is (faithfully) flat if and only if the 
tensor-product functor $V\ot_K:{}_K\M \to{}_K\M$ is (faithfully) exact. 
\end{lemma}

If the equivalent conditions above are satisfied, we say simply that $V$ is (\emph{faithfully}) \emph{flat}. 

We say that $V$ is \emph{free}, if it has a $K$-free basis consisting of homogeneous elements. 
As a remark there can exist a $K$-supermodule which is free, regarded as an ordinary module, 
but which does not have any free basis consisting 
of homogeneous elements. For example, in case $K$ is the direct product $A\times B$ of two non-zero
superalgebras, $V=A[0]\oplus B[1]$
is such an example. 

Let $m, n$ be non-negative integers. 
We let
\[ K^{m|n}=K[0]^m\oplus K[1]^n \]
denote the $K$-supermodule which is the direct sum of the $m$-copies of $K[0]\, (=K)$ and $n$-copies of
the degree shift $K[1]$.

\begin{definition}\label{SASdef1}
A $K$-supermodule $V$ is said to be \emph{$K$-finite free of rank} $m|n$, if it is isomorphic to $K^{m|n}$. 
If $K$ is a (purely even) field, the rank coincides with the dimension \eqref{SASeq00-}. 
\end{definition}

Assume $m>0$ or $n>0$. The ring
$\op{Hom}_K(K^{m|n},K^{m|n})$
consisting of all $K$-linear endomorphisms of $K^{m|n}$ that may not preserve the parity 
turns into a non-supercommutative superalgebra with respect to the parity inherited from $\op{Hom}(K^{m|n},K^{m|n})$; 
see \tc{\eqref{SASeq0}}.
Suppose that $K^{m|n}$ consists of column vectors on which $K$ acts on the right.
Then $\op{Hom}_K(K^{m|n},K^{m|n})$ is identified with $\mathsf{M}_{m|n}(K)$ given by the following.

\begin{example}\label{SASex1}
$\mathsf{M}_{m|n}(K)$ denotes the \tc{vector space} of \tc{those} $(m+n)\times (m+n)$-matrices which 
\tc{are in the direct sum}
\eq\label{SASeq1}
\begin{pmatrix} (K_0)_{m,m}&(K_1)_{m,n}\\ (K_1)_{n,m}&(K_0)_{n,n}\end{pmatrix}\tc{\oplus}
\begin{pmatrix} (K_1)_{m,m}&(K_0)_{m,n}\\ (K_0)_{n,m}&(K_1)_{n,n}\end{pmatrix}. 
\eeq
\tc{Here,} $(K_i)_{r,s}$ denotes the \tc{vector space} of the $r\times s$-matrices with entries in $K_i$, $i=0, 1$, and
the $2\times 2$-matrices \tc{in \eqref{SASeq1}} denote the \tc{vector spaces} of those matrices which \tc{consist of four} sub-matrices
belonging to \tc{the} $(K_i)_{r,s}$ \tc{indicated} at their positions. One sees that $\mathsf{M}_{m|n}(K)$ is a non-supercommutative
superalgebra with respect to the ordinary operations on matrices. The even (resp., odd) component
is the first (resp., second) \tc{vector space in \eqref{SASeq1}}. 
\tc{Those homogeneous} elements will be called \emph{even} \tc{or} \emph{odd matrices}. 
\end{example}

\subsection{SUSY fields}\label{SUF}

Let $K$ be a superalgebra. 

We define
\begin{equation}\label{SUFeq0}
I_K:=(K_1),\quad \K:=K/I_K\, (=K_0/K_1^2).
\end{equation}
Thus, $I_K$ is the super-ideal of $K$ generated by the odd component $K_1$ of $K$, and
$\K$ is the largest purely even quotient of $K$. 
The $\mathbb{N}$-graded algebra 
\begin{equation}\label{SUFeq0a}
\op{gr}K=\bigoplus_{n=0}^\infty I_K^n/I_K^{n+1}
\end{equation}
associated with the $I_K$-adic filtration has homogeneous components such that
\[
(\op{gr}K)(n)=I_K^n/I_K^{n+1}=\begin{cases} K_0/K_1^2, & n=0; \\ K_1^n/K_1^{n+2}, &n>0\end{cases}
\]
(see \cite[p.360]{MZ1}), which are purely even (resp., purely odd) if $n$ is even (resp., odd). 
Hence $\op{gr} K$ turns into a superalgebra over $\K\, (=(\op{gr}K)(0))$ 
by the $\op{mod}2$ grading, which is indeed super-commutative. 

The subset $\sqrt{0}$ of $K$ consisting of all nilpotent elements forms a super-ideal, called the \emph{nil radical} of $K$. We let
\eq\label{SUFeq0b}
K_{\op{red}}=K/\sqrt{0}
\eeq
denote the quotient superalgebra, which is purely even. In fact, we see $\sqrt{0}\supset I_A$,
so that
\[
K_{\op{red}}=\K_{\op{red}}.
\]

We say that $K$ is \emph{Noetherian}, if the super-ideals of $K$ satisfy
the ACC, or equivalently, if the even component $K_0$ is a Noetherian algebra, and 
the odd component $K_1$, regarded as a $K_0$-module,
is finitely generated; see \cite[Section 2.4]{MT} for other equivalent conditions. 
Note that if $K$ is Noetherian, then the super-ideal $I_K$, being finitely generated, is nilpotent. 

We say that $K$ \emph{smooth} (over $R$), if given a surjective
superalgebra morphism $A \to B$ with nilpotent kernel, every superalgebra morphism 
$K\to B$ factors through $A$.
The generalized notion of \emph{smooth algebras over a superalgebra} (not necessarily, a field)
is defined in the obvious manner.

The following is part of \cite[Theorem A.2]{MZ2}.

\begin{prop}\label{SUFprop1}
Assume that $K$ is Noetherian. Then $K$ is smooth if and only if
\begin{itemize}
\item[(i)] $\overline{K}$ is smooth,
\item[(ii)] $I_K/I_K^2$, naturally regarded as a purely odd (finitely generated) $\overline{K}$-super-module, 
is projective as an $\overline{K}$-module, and 
\item[(iii)] there is an isomorphism 
\eq\label{SUFeq1}
\wedge_{\overline{K}}(I_K/I_K^2)\overset{\simeq}{\longrightarrow} K
\eeq
of superalgebras.
\end{itemize}
\end{prop}

\begin{rem}\label{SUFrem1}
Here are two remarks.
\begin{itemize}
\item[(1)] Recall the assumption $\op{char} R=0$. 
Then one sees that in case $\K$ is a field, Condition (i), as well as (ii), is necessarily satisfied. 
On the other hand it follows from \cite[Theorem A.2]{MZ2} that a Notherian superalgebra 
$K$ is smooth if and only if 
it is \emph{regular} \cite[Definition A.1]{MZ2} in the sense of T. Schmitt
that (i)~$\K$ is regular, (ii)~the same as in the preceding proposition, and (iii)~the canonical $\K$-superalgebra
morphism 
\begin{equation}\label{SUFeq1a}
\wedge_{\K}(I_K/I_K^2) \to \op{gr}K
\end{equation}
which arises from the embedding 
$I_K/I_K^2=(\op{gr}K)(1)\hookrightarrow
\op{gr}K$ (see \eqref{SUFeq0a}) is an isomorphism. 
\item[(2)]
Suppose that $K$ is Noetherian and smooth. By Condition (i) of the preceding proposition, the projection $K\to \K$ splits as a superalgebra morphism. Choose arbitrarily a section,
and regard $K$ as a $\K$-superalgebra through it. 
As is seen from the proof of \cite[Theorem A.2]{MZ2}
(or from \cite[Remark 2.2 (2)]{MOT}),  
we can choose an isomorphism such as in \eqref{SUFeq1}, so that it is an isomorphism over $\K$, 
and the associated graded-algebra isomorphism 
$\wedge_{\K}(I_K/I_K^2) \overset{\simeq}{\longrightarrow} \op{gr}K$ is the canonical one given in \eqref{SUFeq1a}. 
This holds without \tc{the assumption} $\op{char}R=0$. 
\end{itemize}
\end{rem}

\begin{definition}\label{SUFdef1}
A \emph{SUSY field} is a Noetherian superalgebra $K$ such that $\K$ is a field.
\end{definition}

\begin{prop}\label{SUFprop2}
Every SUSY field $K$ has the following properties. 
\begin{itemize}
\item[(1)]
$K$ is self-injective in the sense that $K$ is injective as a $K$-supermodule.
\item[(2)]
A $K$-supermodule $V$ is free if and only if it is flat.
\end{itemize}
\end{prop}
\pf
(1)
Set $\tc{W}:=I_K/I_K^2$; this is a finite-dimensional, purely odd super-vector space over 
the field $\K$. By Proposition \ref{SUFprop1} one may suppose
$K=\wedge_{\K}(\tc{W})$, whence $K$ has a Hopf-superalgebra structure over $\K$, such that every element of \tc{$W$} is odd primitive. 
Since the $K$-supermodule $K$ is projective (see \cite[Lemma 5.1 (2)]{M}), the argument of duality shows that
the $\K$-linear dual $K^*$ of $K$, naturally regarded as a $K$-supermodule, is injective.
An analogous argument of proving that every finite-dimensional Hopf algebra is Frobenius shows that
$K^*$ is isomorphic to $K$ (resp., to $K[1]$), in case the unique (up to scalar multiplication)
integral of $K^*$ is even (resp., odd). This shows the desired self-injectivity. 

(2)
To prove the non-trivial ``if'', assume that $V$ is flat. 
By the assumption, $\K$ is a field. 
By choosing homogeneous elements of $V$ which form a $\K$-basis modulo $I_K$, 
one obtains a $K$-supermodule morphism
$f : U \to V$ from a free $K$-supermodule $U$, such that it turns into an isomorphim with $\K\ot_K$ applied.
We wish to show that $f$ is an isomorphism.
Since $\K\ot_K\op{Coker}(f)=0$ and $I_K$ is nilpotent, it follows that $\op{Coker}(f)=0$. The 
exact sequence $\op{Tor}_1^K(\K,V)\to \K\ot_K\op{Ker}(f) \to \K\ot_KU\to \K\ot_KV\to 0$, combined
with the fact $\op{Tor}_1^K(\K,V)=0$ ensured by the flatness assumption, 
shows that $\K\ot_K \op{Ker}(f)=0$, whence $\op{Ker}(f)=0$ again by the nilpotency of $I_K$.
\epf 

\begin{rem}\label{SUFrem2}
We see from the proof above the following.
\begin{itemize}
\item[(1)]
In the situation of Part 1 above, $K$ is injective as an ordinary $K$-module on either side. 
\tc{Now, one} sees by the same way of proving \cite[Lemma 5.1 (2)]{M} (which uses the fact $K\rtimes \mathbb{Z}/(2)\supset K$
is a separable ring extension) that a $K$-supermodule \tc{$V$ is necessarily injective in ${}_K\M$, 
provided it is injective 
as an ordinary left or right $K$-module; indeed, $V$ is then seen to be a direct summand of its injective hull in ${}_K\M$}.
As a simple application it follows that 
a direct sum of (possibly, infinitely many) copies of $K$
is injective in ${}_K\M$, \tc{since it is known to be injective as an ordinary $K$-module}. 
\item[(2)]
Part 2 holds, only assuming that $\K$ is a field and $I_K$ is nilpotent. 
\end{itemize}
\end{rem}

An \emph{extension} $L|K$ of objects (e.g., fields, SUSY fields) means an inclusion
$L\supset K$ of those objects. 

Let $L|K$ be an extension of SUSY fields. Then we have an extension $\L|\K$ of fields. 

\begin{lemma}\label{SUFlem1}
The following are equivalent:
\begin{itemize}
\item[(a)]
The $\L$-linear extension $\L\ot_{\K} I_K/I_K^2\to I_L/I_L^2$ of the canonical $\K$-linear morphism
$I_K/I_K^2\to I_L/I_L^2$ is injective;
\item[(b)] The graded-algebra morphism $\op{gr}K \to \op{gr} L$ associated with the inclusion $K \to L$
is injective; 
\item[(c)] An isomorphism $\wedge_{\K}(I_K/I_K^2)\overset{\simeq}{\longrightarrow}K$ such as in \eqref{SUFeq1}
and an analogous one $\wedge_{\L}(I_L/I_L^2)\overset{\simeq}{\longrightarrow}L$ for $L$ can be chosen so that
the diagram 
\[
\begin{xy}
(0,7)   *++{\wedge_{\K}(I_K/I_K^2)}  ="1",
(26,7)  *++{K} ="2",
(0,-8)  *++{\wedge_{\L}(I_L/I_L^2)} ="3",
(26,-8)  *++{L} ="4",
{"1" \SelectTips{cm}{} \ar @{->}^{\hspace{6mm}\simeq} "2"},
{"1" \SelectTips{cm}{} \ar @{->} "3"},
{"2" \SelectTips{cm}{} \ar @{_(->} "4"},
{"3" \SelectTips{cm}{} \ar @{->}^{\hspace{6mm}\simeq} "4"}
\end{xy}
\]
commutes, where the vertical arrow on the LHS indicates the graded-algebra morphism over $\K$ which
is induced from the canonical $\K$-linear morphism $I_K/I_K^2\to I_L/I_L^2$. 
\end{itemize}
\end{lemma}

\begin{definition}\label{SUFdef1a}
An extension $L|K$ of SUSY fields is said to be \emph{admissible}, if it satisfies the
equivalent conditions above.
\end{definition}

\begin{proof}[Proof of Lemma \ref{SUFlem1}]
(a) $\Rightarrow$ (c).\ 
Choose arbitrarily a section $\K \to K$ of 
the projection $K\to \K$, and regard $K$ as a $\K$-superalgebra through it.
Since $\L$ is smooth over $\K$ by the assumption $\op{char} R=0$, a section $\L\to L$ of $L \to \L$
can be chosen so that it extends the above-chosen section $\K\to K$. 
Hence $L$, regarded as an $\L$-superalgebra through the section, includes $K$ as a $\K$-sub-superalgebra. 

Choose odd elements $x_1,\dots,x_r$ in $I_K$, such that they modulo $I_K^2$ form a $\K$-basis
of $I_K/I_K^2$. Assume (a). Then the elements can extend to odd elements $x_1,\dots,x_r,\dots,x_s$ in $I_L$,
so that they modulo $I_L^2$ form an $\L$-basis of $I_L/I_L^2$. We see that the $\K$-superalgebra morphism
\[
\wedge_{\K}(I_K/I_K^2) \to K\ \, \text{defined by}\ \, x_i~\tu{mod}~I_K^2\mapsto x_i,\ 1\le i\le r
\]
and the $\L$-superalgebra morphism 
\[
\wedge_{\L}(I_L/I_L^2) \to L\ \, \text{defined by}\ \, x_i~\tu{mod}~I_L^2\mapsto x_i,\ 1\le i\le s
\]
are isomorphisms required by (c). Here one should notice that these morphisms are isomorphic since
the associated graded-algebra morphisms are so; \tc{see Remark \ref{SUFrem1} (1)}.

(c) $\Rightarrow$ (b).\ This follows from the commutative diagram in (c) with $\op{gr}$ applied. Notice that
the vertical morphism on the LHS is injective, and remains unchanged with $\op{gr}$ applied.

(b) $\Rightarrow$ (a).\ Contrary to (a), let us have $\K$-linearly independent elements $u_1,\dots,u_r$ in $I_K/I_K^2$,
whose images in $I_L/I_L^2$, say $u_1',\dots,u_r'$, are $\L$-linearly dependent. Then the element
$u_1\wedge \dots \wedge u_r$ in $\op{gr}K$ is non-zero, while $u_1'\wedge \dots\wedge u_r'=0$ in $\op{gr}L$.
Thus, (b) does not hold. 
\end{proof}

\begin{rem}\label{SUFrem3}
Suppose that $L|K$ is admissible extension of SUSY fields. 
Then $L$ is free or equivalently, flat over $K$ (see Proposition \ref{SUFprop2} (2)), as is seen from the last proof.
Indeed, 
we can regard $K$ and $L$ as superalgebras over $\K$ and over $\L$, respectively, such that $K$ is a $\K$-sub-superalgebra
of $L$. Moreover, the $\L$-linear extension $\L\ot_{\K}K\to L$ of the inclusion $K\hookrightarrow L$ can be identified 
with the graded-algebra morphism
\[
\L\ot_{\K}\wedge_{\K}(I_K/I_K^2)=\wedge_{\L}(\L\ot_{\K}(I_K/I_K^2))\to \wedge_{\L}(I_L/I_L^2)
\]
which arises from the canonical injection $\L\ot_{\K}(I_K/I_K^2)\to I_L/I_L^2$ given in Condition (a).
This shows the freeness. 
\end{rem}

\begin{example}\label{SUFex0}
Supoose that \tc{$W$} is a purely odd super-vector space (over $R$) of dimension $d\ge 3$, and let 
\[ L:=\wedge(\tc{W})\, \big( \! \! =\bigoplus_{n=0}^d\wedge^n(\tc{W})\big) ,\quad K:=\wedge^0(\tc{W})\oplus \wedge^d(\tc{W}). \]
Then $L|K$ is an extension of SUSY fields. This is not admissible since the associated graded-algebra morphism
$\op{gr} K\to \op{gr}L$ is the natural projection $K \to R=\wedge^0(\tc{W})\, (\subset L)$, which is not injective. 
Moreover, 
$L$ is not free/flat over $K$, since the quotient $K$-supermodule $L/K\, (=\bigoplus_{n=1}^{d-1}\wedge^n(\tc{W}))$,
\tc{which is isomorphic to a direct summand of $L$ in ${}_K\M$ by Proposition \ref{SUFprop2} (1)}, 
is annihilated by $I_K\, (=\wedge^d(\tc{W}))$. 
\end{example}

\begin{definition}\label{SUFdef2}
Given a superalgebra $A$, let 
\eq\label{SUFeq2}
\Quot(A)=S^{-1}A
\eeq
denote the localization of $A$ by the multiplicative subset
$S$ consisting of those elements in $A_0$ which are $A$-regular, or namely,
which annihilate no non-zero elements in $A$. 
\end{definition}

One sees that in case $A$ is purely even, $\Quot(A)$ is the total quotient ring of $A$. 

The following is easy to see.

\begin{lemma}\label{SUFlem2}
$\Quot(A)$ is a superalgrbra including $A$ as a sub-superalgebra. Moreover, it is the smallest
among those superalgebras $T$ including $A$ as a sub-superalgebra, in which every $T$-regular even element is invertible.
\end{lemma}

\begin{example}\label{SUFex1}
Suppose that $A$ is a Noetherian smooth superalgebra, and regard it as an $\overline{A}$-superalgebra
through an arbitrarily chosen section of the projection $A \to \overline{A}$. Then we have an isomorphism 
$A\simeq \wedge_{\overline{A}}(P)$ of $\overline{A}$-superalgebras, where we have set $P:=I_A/I_A^2$; 
see Proposition \ref{SUFprop1} and Remark \ref{SUFrem1} (2). 
Suppose that $\overline{A}$ is an integral domain, and 
let $F=\Quot(\overline{A})$ denote the quotient field of $\overline{A}$. 
In view of the fact that $P$ is finitely generated projective over $\overline{A}$, 
we see that the isomorphism above extends to an isomorphism 
\[
\Quot(A)=A\ot_{\overline{A}}F\simeq\wedge_{F}(P\ot_{\overline{A}}F)
\]
of $F$-superalgebras. Therefore, $\Quot(A)$ is a SUSY field. 
\end{example}

\subsection{Hopf-Galois extensions}\label{HGE}
With the applications in following sections in mind, 
we let $k$ be a field including $R$, and work over $k$
through to the end of Section \ref{SSC}. 

Let $H=(H,\Delta,\varepsilon)$ be a Hopf superalgebra over $k$; such a superalgebra as well is assumed to be 
super-commutative unless otherwise stated. 
An $H$-supercomodule, without a specific side referred to, 
will mean a \emph{right} $H$-supercomodule. The category of those supercomodules is denoted by 
\[
\M^H=(\M^H,\ot_k, k).
\]
which is, in fact, monoidal, as here presented,
and is symmetric with respect to the supersymmetry. 

One sees that
\begin{equation*}\label{HGEeq1}
\overline{H}=H/(H_1)
\end{equation*}
is a purely even quotient Hopf superalgebra of $H$. 
Let $H^+=\op{Ker}(\varepsilon : H \to k)$ denote the augmentation super-ideal of $H$, and define
\begin{equation}\label{HGEeq1a}
\mathfrak{w}_H:=(H^+/(H^+)^2)_1\, (=H_1/(H^+\cap H_0)H_1),
\end{equation}
the odd component of the super-vector space $H^+/(H^+)^2$ over $k$; this is alternatively denoted by $W^H$ in \cite{M}.
We let $\wedge_k(\mathfrak{w}_H)$ denote the exterior algebra on $\mathfrak{w}_H$, which is naturally a superalgebra over $k$. 
Theorem 4.5 of \cite{M} tells that there is an algebra isomorphism 
\begin{equation}\label{HGEeq2}
H \overset{\simeq}{\longrightarrow} \overline{H}\otk \wedge_k(\mathfrak{w}_H)
\end{equation}
in the category ${}^{\overline{H}}\M$ of \emph{left} $\overline{H}$-supercomodules, which, composed with
the natural projection $\overline{H}\otk \wedge_k(\mathfrak{w}_H)\to \overline{H}\otk k$, coincides with
the projection $H \to \overline{H}$. 

Let $A$ be a non-zero $H$-supercomodule superalgebra, that is, a non-zero algebra in $\M^H$. 
By convention we will write the structure morphism 
so as
\begin{equation}\label{HGEeq1b}
\theta_{\! A} : A \to A\otk H.
\end{equation}
We let
\begin{equation}\label{HGEeq1c}
A^{\op{co}H}=\{\, a \in A: \theta_{\! A}(a)=a\ot 1\, \}
\end{equation}
denote the sub-superalgebra of $A$ which consists of the \emph{$H$-coinvariants}, that is, 
those elements $a$ which satisfy the equation above. 

We say that $A$ is \emph{$H$-smooth} over a subalgebra $K$ in $\M^H$, 
if it is smooth as a $K$-algebra in the category, or more explicitly, if 
given a surjective $K$-algebra morphism $S \to T$ in $\M^H$ with nilpotent kernel,
every $K$-algebra morphism $A \to T$ in the category factors through $S$. 
In what follows the notion will appear in the restricted situation when $K=A^{\op{co}H}$. 

Let $H$, $A$ be as above.

\begin{definition}\label{HGEdef1}
We say that $A$ is an \emph{$H$-Galois extension} over a sub-superalgebra $K$, or more simply, 
that $A|K$ is an \emph{Hopf-Galois extension}, if 
\begin{itemize}
\item[(HG1)] $K=A^{\op{co}H}$, 
\item[(HG2)] the morphism
\begin{equation}\label{HGEeq3}
\Theta_A:A\ot_KA\to A\ot_k H,\quad \Theta_A(a\ot a')=a\, \theta_{\! A}(a')
\end{equation}
is bijective, and
\item[(HG3)] $A$ is faithfully flat over $K$; see the paragraph following Lemma \ref{SASlem1}.
\end{itemize}
\end{definition}

\begin{rem}\label{HGErem1}
Here are two remarks concerning the definition above.
\begin{itemize}
\item[(1)]
Assume (HG1). Then $\theta_{\! A} : A \to A\otk H$ is then left $K$-superlinear. Hence, as its left $A$-superlinear extension,
$\Theta_A$ is well defined. Note that it is an algebra morphism
in $\M^H$, where $H$ is supposed to coact on \tc{the source $A\ot_KA$ through the right tensor factor $A\, (=K\ot_KA)$}.
In what follows we will use the symbol $\Theta_A$ to denote the morphism associated with $\theta_{\! A}$ as above. 
\item[(2)]
Assume in turn that 
\[
\Xi_A : A \otk A\to A\otk H,\quad \Xi_A(a\ot a')=
a\, \theta_{\! A}(a')
\]
happens to  
induce an isomorphism $A\ot_K A \overset{\simeq}{\longrightarrow} A\otk H$, where $K$ is a sub-superalgebra of $A$. 
Then we see through the isomorphism that
\[
A^{\op{co} H}=\{\, a \in A : 1\ot_K a=a\ot_K 1\ \, \text{in}\ \, A\ot_K A\, \}.
\]
Therefore, either if (HG3) is satisfied or if $K$ is self-injective (and so, $K$ is a direct summand of $A$ in ${}_K\M$),
then (HG1), as well as (HG2), is satisfied. 
\end{itemize}
\end{rem}

Let $H$ be a Hopf algebra over $k$, as above. 
Given an algebra $A$ in $\M^H$, let
\[
{}_A(\M^H)=({}_A(\M^H),\ot_A,A)
\]
denote the category of $A$-modules in $\M^H$. This is, in fact, a monoidal category, as here presented, which is, moreover,
symmetric 
with respect to the supersymmetry; the circumstance is the same as for ${}_K\M$ in \eqref{SASeq0a}.

The following is the so-called Hopf-module theorem (see \cite[Theorem 8.5.6]{Mon}, for example) formulated in the super context. 

\begin{theorem}[\tu{Hopf-module theorem}]\label{HGEthm0}
Suppose that $A|K$ is an $H$-Galois extension. 
\begin{itemize}
\item[(1)]
Given an object ${}_K\M$, the base extension $V\ot_K A$ turns into an object in ${}_A(\M^H)$, for which $H$ is supposed to
coact on the tensor factor $A$ in $V\ot_KA$. 
\item[(2)]
Given an object $M=(M,\theta_M)$ in ${}_A(\M^H)$, 
\[
M^{\op{co}H}:=\{\, m\in M : \theta_M(m)=m\ot 1\ \, \tu{in}\ \, M\ot_k H\, \}
\]
is naturally an object in ${}_K\M$.
\item[(3)] 
The assignments $V\mapsto V\ot_KA$ and $M \mapsto M^{\op{co}H}$ obtained above give rise to
symmetric monoidal equivalences 
\eq\label{HGEeq3a}
({}_K\M,\ot_K,K) \approx ({}_A(\M^H),\ot_A,A)
\eeq
which are mutual quasi-inverses, with respect to the obvious monoidal structures.
\end{itemize}
\end{theorem}

We reproduce two results from \cite{MOT} in modified form suited to our applications. They originally assume 
that $H$ is smooth over $k$, which is now satisfied since
$\op{char} k=0$ by assumption. 

\begin{prop}[\tu{\cite[Proposition 5.1, Theorem 3.12]{MOT}}]\label{HGEprop1}
Assume that $H$ is finitely generated. 
If $A|K$ is an $H$-Galois extension, then $A$ is smooth and 
$H$-smooth over $K$.
\end{prop}

\begin{theorem}[\tu{\cite[Theorem 3.12]{MOT}}]\label{HGEthm1}
Assume that $H$ is finitely generated. 
Let $A$ be an $H$-Galois extension over a Noetherian smooth
$k$-superalgebra $K$. 
\begin{itemize}
\item[(1)] 
$\overline{A}$ is naturally regarded as a purely even 
algebra in $\M^{\overline{H}}$, and includes $\overline{K}$ 
as a subalgebra. Moreover, $\overline{A}|\overline{K}$ is an
$\overline{H}$-Galois extension.
\item[(2)]
We have a $K$-algebra isomorphism
\begin{equation}\label{HGEeq4}
A \simeq (K\ot_{\overline{K}}\overline{A})\square_{\overline{H}}H
\end{equation}
in $\M^H$, where $K$ is regarded as a 
$\overline{K}$-superalgebra 
through an arbitrarily chosen section of the projection 
$K \to \overline{K}$. Moreover, the isomorphism above, combined with \eqref{HGEeq2}, gives an isomorphism
\begin{equation}\label{HGEeq5}
A \simeq K\ot_{\overline{K}}\overline{A}\ot_k \wedge_k(\mathfrak{w}_H)
\end{equation}
of $K$-superalgebras. 
\end{itemize}
\end{theorem}

\begin{rem}\label{HGErem2}
The symbol $\square_{\overline{H}}$ which appears in \eqref{HGEeq4} is the co-tensor product; see \cite[Definition 2.2.9]{HS}, for example. 
In the present situation, 
given an algebra $T=(T,\vartheta_T)$ in $\M^{\overline{H}}$, $T\, \square_{\overline{H}}H$
denotes the subalgebra of $T\ot_k H$ in $\M^H$ which is defined by the equalizer diagram
\[
\begin{xy}
(0,0.9)   *++{T\, \square_{\overline{H}}H}  ="1",
(20,1)  *++{T\ot_k H} ="2",
(50,1)  *++{T\ot_k \overline{H}\ot_k H,} ="3",
(27,0)  *++{} ="4",
(27,2)  *++{} ="5",
(38,0)  *++{} ="6",
(38,2)  *++{} ="7",
{"1" \SelectTips{cm}{} \ar @{->} "2"},
{"4" \SelectTips{cm}{} \ar @{->}_{\op{id}\ot \delta_H} "6"},
{"5" \SelectTips{cm}{} \ar @{->}^{\vartheta_T\ot\op{id}} "7"}
\end{xy}
\]
where $\delta_H$ denotes the composite $H\overset{\Delta}{\longrightarrow}H\ot_kH\to \overline{H}\ot_kH$
of the coproduct of $H$ with the natural projection; cf. \eqref{SASeq0b}. 
Since $\vartheta_T :T\to T\ot_k \overline{H}$ restricts to
an isomorphism $T\simeq T\, \square_{\overline{H}}\overline{H}$, we obtain 
$T\ot_k\wedge_k(\mathfrak{w}_H)\simeq T\, \square_{\overline{H}}H$ from \eqref{HGEeq2}. Thus, \eqref{HGEeq5} follows
from \eqref{HGEeq4}.

Here is a remark on \eqref{HGEeq4}: since $K$ is $\overline{K}$-flat (or $H$ is left $\overline{H}$-coflat), we have the canonical isomorphism
\[
(K\ot_{\overline{K}}\overline{A})\square_{\overline{H}}H\simeq K\ot_{\overline{K}}(\overline{A}\, \square_{\overline{H}}H).
\]
Therefore, the both sides may be presented simply by $K\ot_{\overline{K}}\overline{A}\, \square_{\overline{H}}H$. 
\end{rem}


\subsection{Affine supergroup schemes}\label{ASG} 
An \emph{affine supergroup scheme} $\G$ (over $k$) is by definition a representable group-valued functor defined
on the category $\mathtt{SAlg}_k$ of superalgebras over $k$. It is, therefore, uniquely represented by a Hopf superalgebra,
say $H$; the relation between $\G$ and $H$ is presented so as 
\[
\G=\op{Sp}_k(H),\quad H=\mathcal{O}_k(\G).
\]
Note that $H^+/(H^+)^2$ is regarded as the cotangent superspace of $\G$ at the identity, and 
the $\mathfrak{w}_H$ defined by \eqref{HGEeq1a} is its odd component. 
A \emph{$\G$-supermodule} may be understood to be the same as an $H$-supercomodule, so that
the category ${}_{\G}\M$ of $\G$-supermodules is defined by
\eq\label{ASGeq1}
{}_{\G}\M =\M^H.
\eeq

We call $\G$ an \emph{affine algebraic supergroup scheme} if $H$ is finitely generated.
Here is a typical example of such supergroup schemes, which we will use in Sections 
\ref{CHA} and \ref{secUEX}.

\begin{example}\label{ASGex1}
Let $m, n$ be non-negative integers, at least one of which is positive.
For every superalgebra $K$ over $k$, $\mathsf{GL}_{m|n}(K)$ denotes the set of all even matrices
in $\mathsf{M}_{m|n}(K)$ (see Example \ref{SASex1})
that are invertible with respect to the product of matrices. Obviously,
this forms a group, and gives rise, with $K$ varying, to a group-valued functor, $\mathsf{GL}_{m|n}$. 
This is, in fact, an affine algebraic supergroup scheme with
\begin{equation}\label{ASGeq2}
\mathcal{O}_k(\mathsf{GL}_{m|n})=k[t_{ij}: 1\le i,j\le m+n]_{\det_0}.  
\end{equation}
This is the polynomial superalgebra localized at $\det_0$, where the indeterminates $t_{ij}$ are such that $\big(t_{ij}\big)_{i,j}$ is 
the generic even matrix, and $\det_0$ denotes the even polynomial defined as the product 
\begin{equation}\label{ASGeq3}
\op{det}_0=\op{det}_0\big(t_{ij}\big)_{i,j}:=\det\big(t_{ij}\big)_{1\le i,j\le m}\, \det\big(t_{ij}\big)_{m< i,j\le m+n}
\end{equation}
of the determinants of the two matrices with even entries in diagonal position. 
\end{example}


\subsection{Superschemes}\label{SSC}
\tc{As a notion more primitive than} affine supergroup schemes,  
an \emph{affine superscheme}  (over $k$) is defined to be a representable set-valued functor on 
the category $\mathtt{SAlg}_k$ of $k$-superalgebras. 
From such functorial view-point, \emph{superschemes} and \emph{fppf sheaves} are defined as natural generalizations
of the ordinary notions of schemes and of fppf sheaves, respectively, defined in the non-super context; they are functors on 
$\mathtt{SAlg}_k$ possessing some appropriate properties, and their categories are in the relation
\[
\begin{pmatrix} \text{affine superschemes} \end{pmatrix} \subset \begin{pmatrix} \text{superschemes}
\end{pmatrix} \subset  \begin{pmatrix}\text{fppf sheaves}\end{pmatrix}.
\]

Given a functor $\mathsf{X}$ on $\mathtt{SAlg}_k$, we obtain by restriction a functor,  
denoted by $\mathsf{X}_{\op{ev}}$, which is defined on the category $\mathtt{Alg}_k$
of (commutative) $k$-algebras, where $\mathtt{Alg}_k$ 
is identified with the full subcategory of $\mathtt{SAlg}_k$
consisting of all purely even superalgebras. 
If $\mathsf{X}$ is an (affine) superscheme (resp., fppf sheaf), then $\mathsf{X}_{\op{ev}}$ is an
affine scheme (resp., fppf sheaf). In particular, if $\mathsf{X}$ is represented by $A$, then 
$\mathsf{X}_{\op{ev}}$ is represented $\overline{A}$. Note that if $\mathsf{G}$ is an affine supergroupscheme,
then $\mathsf{G}_{\op{ev}}$ is an affine group scheme. 

From a geometric view-point, a \emph{superscheme} (over $k$)
is alternatively defined to be a super-ringed space,
that is, a topological space equipped with a sheaf of $k$-superalgebras, which is locally isomorphic to a
\emph{geometric affine superscheme} such as defined below. 
For this geometric view-point see 
\cite[Section 3.3]{CCF}, \cite[Chapter 4, Section 1]{Manin} or \cite[Section 4]{MZ1}.

Given a superalgebra $A$, a \emph{prime} of $A$ is a super-ideal $\mathfrak{p}$ of $A$ such that
$A/\mathfrak{p}$ is a (purely even) integral domain. The condition is equivalent to saying 
(i)~$\mathfrak{p}\supset A_1$, and (ii)~$\mathfrak{p}/I_A$ is a prime
of $\overline{A}\, (=A/I_A)$, or equivalently, (i)~$\mathfrak{p}\supset A_1$, 
and (ii)$'$~$\mathfrak{p}_0$ is a prime of $A_0$. 
Thus, the set $\op{Spec}A$ of 
all primes of $A$ is naturally identified with the prime spectrum of $\op{Spec}\overline{A}$ of $\overline{A}$,
and also with that $\op{Spec}A_0$ of $A_0$. These \tc{two prime spectrums} are identified with each other 
as topological spaces.
Transfer their topologies into
$\op{Spec}A$ through the natural identifications.
Then $\op{Spec}A$ 
turns into a super-ringed space, equipped with
a unique sheaf $\cO_{\op{Spec}\! A}$
of $k$-superalgebras
such that the stalk $\cO_{\op{Spec}\! A,\mathfrak{p}}$ at $\mathfrak{p}$
is the localization $A\ot_{A_0}(A_0)_{\mathfrak{p}_0}$ of $A$ by the multiplicative subset 
$A_0\setminus \mathfrak{p}_0$ of $A_0$. 
A \emph{geometric affine superscheme} is a super-ringed
space of the form $(\op{Spec}A, \cO_{\op{Spec}A})$ such as above. 

In Section \ref{GAL} we will use the word ``superschemes'' in the geometric sense as just defined, 
as well as in the functorial sense; 
see Proposition \ref{GALprop2} and its subsequence. This is justified since the notions 
defined in two ways are equivalent, or more precisely, since there is a natural category-equivalence
between the functorial superschemes and the geometric ones, in which affine objects $\op{Sp}_k(A)$ and
$\op{Spec}A$ correspond to each other; see \cite[Theorem 5.14]{MZ1}.

In the rest of this subsection we let $\mathsf{X}$ be a non-empty superscheme in the geometric sense.

The associated superscheme $\mathsf{X}_{\op{ev}}$, viewed geometrically,
is characterized by the properties: \tc{(1)}~$\mathsf{X}_{\op{ev}}$ has the same underlying topological space
as $\mathsf{X}$, and 
\tc{(2)}~$\mathsf{X}_{\op{ev}}|_U=\op{Spec}\overline{\cO_{\mathsf{X}}(U)}$ for every affine open set $U$
of $\mathsf{X}$. 

We say that $\mathsf{X}$ is 
\begin{itemize}
\item[--]\emph{integral} or \emph{irreducible}, if the associated scheme $\mathsf{X}_{\op{ev}}$ is such;
\item[--]\emph{locally Noetherian} (resp., \emph{smooth} over $k$), if the stalks $\cO_{\mathsf{X},x}$, 
where $x\in \mathsf{X}$, 
are all Noetherian (resp., smooth over $k$);
\item[--]\emph{Noetherian}, if $\mathsf{X}$ is locally Noetherian and quasi-compact. 
\end{itemize}

Assume that $\mathsf{X}$ is integral. Then it contains a
unique generic point, say, $\eta$. We define
\[
k(\mathsf{X}):=\cO_{\mathsf{X},\eta}.
\]
We will discuss $k(\mathsf{X})$ only under the assumptions (i)--(iii) below. Notice that $\mathsf{X}$
is integral under the assumptions.

\begin{lemma}\label{SSClem1}
Assume that $\mathsf{X}$ is (i)~irreducible, (ii)~locally Noetherian, and (iii)~smooth over $k$. 
\begin{itemize}
\item[(1)] $k(\mathsf{X})$ is a SUSY field such that $\overline{k(\mathsf{X})}$ coincides with
the function field of the integral scheme $\mathsf{X}_{\op{ev}}$. 
\item[(2)] For every non-empty affine open set $U$
of $\mathsf{X}$, 
we have $\Quot(\cO_{\mathsf{X}}(U))=k(\mathsf{X})$.
\end{itemize}
\end{lemma}

\begin{definition}\label{SSCdef1}
Under the assumptions (i)-(iii) above, 
we call $k(\mathsf{X})$ \emph{the function SUSY field} of $\mathsf{X}$. 
\end{definition}

\pf[Proof of Lemma \ref{SSClem1}]
Let $U$ be a non-empty affine open set of $\mathsf{X}$, and set $A:=\cO_{\mathsf{X}}(U)$. 
Then by (ii) and (iii), $A$ is Noetherian smooth superalgebra over $k$, so that $P:=I_A/I_A^2$
is a finitely generated projective $\overline{A}$-module, and we have a $k$-superalgebra isomorphism
\eq\label{SSCeq1}
\wedge_{\overline{A}}(P)\simeq A
\eeq
such as in \eqref{SUFeq1}, which may be supposed to be
an isomorphism of superalgebras over $\overline{A}$; see Remark \ref{SUFrem1} (2). 
With (i) added, $U$ contains the generic point $\eta$, and it corresponds to the prime $I_A$ of $A$,
which is indeed a prime since $\overline{A}\, (=\cO_{\mathsf{X}_{\op{ev}}}(U))$ is an integral domain.
Let $F=\Quot(\overline{A})$ be the quotient field of $\overline{A}$, so that $F$ equals the function field
of the integral scheme $\mathsf{X}_{\op{ev}}$, or in notation, 
\eq\label{SSCeq2}
F=k(\mathsf{X}_{\op{ev}}).
\eeq
The isomorphism \eqref{SSCeq1}
extends, by localization at $I_A$, to 
\eq\label{SSCeq3}
\wedge_{F}(P\ot_{\overline{A}}F)\simeq k(\mathsf{X}),
\eeq
since every element of $A_0\setminus I_A$ coincides with some non-zero element of $\overline{A}$
modulo nilpotents. Since $\wedge_{F}(P\ot_{\overline{A}}F)=\Quot(A)$, as was seen in Example \ref{SUFex1}, 
Part 2 of the lemma follows. Part 1 follows from \eqref{SSCeq2} and \eqref{SSCeq3}.
\epf

\section{$D$-superalgebras}\label{secDSA}

Throughout in what follows we let $D=(D, \Delta,\varepsilon)$ denote a super-cocommutative Hopf superalgebra 
(over the base field $R$),
which may not be super-commutative. We present the coproduct $\Delta : D\to D\ot D$ so as
\[
\Delta(d)=d_{(1)}\ot d_{(2)},\quad d\in D.
\]
A $D$-supermodule will mean a left $D$-supermodule. The category $({}_D\M, \ot, R)$ of $D$-supermodules
forms a monoidal category, which is symmetric with respect to the supersymmetry. 

We let
\eq\label{DSAeq0a}
\g_D=\{\, x\in D : \Delta(x)=x\ot 1+ 1\ot x \, \} 
\eeq
denote the set of all primitive elements in $D$, which forms a Lie superalgebra (over $R$) with respect to the
super-commutator $[x,y]:=xy-(-1)^{|x||y|} yx$, where $x,y\in \g_D$. 
The even component $(\g_D)_0$ of $\g_D$ is a Lie algebra. 

We let
\eq\label{DSAeq0b}
\Gamma_D=\{\, \gamma \in D : \Delta(\gamma)=\gamma\ot \gamma,\ \varepsilon(\gamma)=1\, \}
\eeq
denote the set of all grouplike elements in $D$. 
Those elements are necessarily even, and $\Gamma_D$ forms a group with respect to the product of $D$.
This group acts by conjugation on the Lie superalgebra $\g_D$; the action extends onto
the universal envelope $U(\g_D)$ of $\g_D$. 
By the assumption $\op{char}R=0$, the associated semi-direct (or smash) product 
$U(\g_D)\rtimes \Gamma_D$, as well as $U(\g_D)$, 
is included in $D$ as a Hopf sub-superalgebra. 

According to the situation we will pose to $D$ some of the following assumptions with explicit citation. 
\begin{itemize}
\item[(D1)] $\dim (\g_D)_1<\infty$, that is, the odd component $(\g_D)_1$ of $\g_D$ is finite-dimensional.
\item[(D2)] $D$ is pointed, or explicitly, $D=U(\g_D)\rtimes \Gamma_D$.
\item[(D3)] $D$ is irreducible, or explicitly, $D=U(\g_D)$.
\end{itemize}
Obviously, (D3) implies (D2). 

Define
\begin{equation}\label{DASeq0c}
\D:=\Delta^{-1}(D_0\ot D_0),
\end{equation}
the pull-back of $D_0\ot D_0$ (in $D\ot D$) along the coproduct of $D$; see \cite[p.291]{M}. 
This is the largest purely even sub-supercoalgebra of $D$, and is indeed a Hopf sub-superalgebra.
By \cite[Theorem 3.6]{M} there exists a non-canonical coalgebra-isomorphism 
\eq\label{DASeq0c0}
\D\ot \wedge((\g_D)_1) \overset{\simeq}{\longrightarrow} D
\eeq
in ${}_{\D}\M$ which extends the inclusions $(\D\ot R=)\, \D \hookrightarrow D\hookleftarrow
(\g_D)_1\, (=R\ot (\g_D)_1)$. The condition
(D2) (resp., (D3)) is equivalent to saying that $\D$ is pointed (resp., irreducible), 
or explicitly, $\D =U((\g_D)_0)\rtimes \Gamma_D$\ (resp., $\D =U((\g_D)_0)$).

\subsection{$D$-simplicity}\label{DSA}
Let $D$ be in general. Given an object $V$ in ${}_D\M$, we let 
\[
V^D=\{\, a \in A : dv=\varepsilon(d)v\ \text{for all}\ d \in D \, \}
\]
denote the sub-super-vector space of $V$ which consists of all 
$D$-invariants. If $A$ is an algebra in ${}_D\M$, then $A^D$
is a sub-superalgebra of $A$. 

\begin{definition}\label{DSAdef1}
An algebra in ${}_D\M$ is called \emph{$D$-superalgebra}, which is assumed to be super-commutative by convention.
It is often called a $D$-supermodule superalgebra in the literature. A $D$-superalgebra is called a \emph{$D$-SUSY field}, if it
is a SUSY field, forgetting the $D$-action. 
\end{definition}

Let $A$ be a $D$-superalgebra.
Notice that the requirement that the identity $R \to A$, $1\mapsto 1$ and the product $A\ot A\to A$,
$a\ot a'\mapsto aa'$ should preserve the action by $d\in D$ is explicitly presented by
\eq\label{DSAeq0c1}
d1=\varepsilon(d)1,\quad d(aa')=(-1)^{|a||d_{(2)}|}(d_{(1)}a)(d_{(2)}a').
\eeq
The left $A$-modules in ${}_D\M$ form a symmetric monoidal category
\eq\label{DSAeq0d}
{}_A({}_D\M)= ({}_A({}_D\M),\ot_A,A)
\eeq
\tc{with respect to the supersymmetry; cf.~\eqref{SASeq0a}.}
This category is naturally identified with 
the one $({}_D\M)_A$ of right $A$-modules in ${}_D\M$. 
An $A$-module will be supposed to be left unless otherwise
specified. Given an $A$-module $V$ in ${}_D\M$, we have the
natural $A$-module morphism 
\begin{equation}\label{DSAeq1}
\mu=\mu_V : A\ot_{A^D} V^D\to V, \quad \mu(a \ot v)=av
\end{equation}
in ${}_D\M$; this is an $A$-algebra morphism in ${}_D\M$ if
$V$ is an $A$-algebra in the category. The subscript $V$ will be
omitted from the notation $\mu_V$, in case the relevant $A$-module 
is complicated. 

\begin{definition}\label{DSAdef2}
A $D$-superalgebra $A$ is said to be \emph{$D$-simple},
if it is non-zero and includes no non-trivial $D$-stable 
super-ideal. 
\end{definition}

\begin{prop}\label{DSAprop1}
Suppose that $A$ is a $D$-simple $D$-superalgebra. Then we have 
\begin{itemize}
\item[(1)] $A^D$ is a field, and is included in $A_0$, in particular.
\item[(2)] For every $A$-module $V$ in ${}_D\M$, $\mu_V : A\ot_{A^D} V^D\to V$ is injective. 
\end{itemize}
\end{prop}
\pf[Proof \tu{(cf. the proof of \cite[Proposition~12.5]{AMT})}]
Notice that ${}_A({}_D\M)$ is an abelian category, in which 
a morphism is monic if and only if it is (set-theoretically) injective. 
The $D$-simplicity assumption is equivalent to saying that
$A$ is a simple object in that category. 

(1)\ First, we see $A^D\subset A_0$. Indeed, if $A^D$ contained a non-zero odd element $x$, then $a \mapsto ax$
would give an isomorphism $A\overset{\simeq}{\longrightarrow} A[1]$ in ${}_A({}_D\M)$. Then the $1$ in $A$ \tc{must} be annihilated by
$x$ since the corresponding $x$ in $A[1]$ is; this is absurd.
Therefore, the endomorphism ring ${}_A({}_D\M)(A,A)$ of $A$ is naturally identified with 
$A^D\cap A_0=A^D$, which must be a field by Schur's Lemma. 

Let us denote the field by $k=A^D$. 

(2)\ We prove by induction on $n >0$ that $k$-linearly independent, 
homogeneous elements $v_1,\dots, v_n$ of $V^D$, whose degree we suppose are $\delta_i=|v_i|\, (\in \{0,1\})$, are
$A$-linearly independent in $V$.  
The conclusion is equivalent to saying that the morphism 
\[
F_n:=(f_1,\dots,f_n): \bigoplus_{i=1}^n A[\delta_i] \to V
\]
in ${}_A({}_D\M)$ is injective, where $f_i : A[\delta_i] \to V$ is defined by $f_i(a)=av_i$. 
When $n=1$, this is true since $A[\delta_1]$ is a simple object. 
Suppose $n>1$. We have the pull-back diagram
\[
\begin{xy}
(0,0)   *++{\op{Ker}(F_n)}  ="1",
(30,0)  *++{\bigoplus_{i=1}^{n-1} A[\delta_i]} ="2",
(0,-16)  *++{A[\delta_n]} ="3",
(30,-16) *++{V.}="4",
{"1" \SelectTips{cm}{} \ar @{->}^{q} "2"},
{"3" \SelectTips{cm}{} \ar @{->}^{-f_n} "4"},
{"1" \SelectTips{cm}{} \ar @{->}_{p} "3"},
{"2" \SelectTips{cm}{} \ar @{->}^{F_{n-1}} "4"}
\end{xy}
\]
By the induction hypothesis we may suppose that $F_{n-1}$ is injective. This implies that $p:\op{Ker}(F_n)\to A[\delta_n]$ is injective.
Contrary to the desired result that $F_n$ is injective, suppose $\op{Ker}(F_n)\ne 0$. Then $p$ must be isomorphic since
$A[\delta_n]$ is a simple object.
The composite $q\circ p^{-1}: A[\delta_n]\to \bigoplus_{i=1}^{n-1}A[\delta_i]$ is of the form ${}^t(c_1,\dots,c_{n-1})$, where
$c_i$ is a scalar in $k$ if $\delta_i=\delta_n$, and is zero if $\delta_i\ne \delta_n$, as is seen from the argument in the preceding paragraph. 
Since $-f_n=F_{n-1}\circ q\circ p^{-1}$, it follows that $-f_n$ is a $k$-linear combination, 
\tc{$f_n=\sum_{i=1}^{n-1} c_if_i$,} 
of $f_1,\dots,f_{n-1}$ in the vector
space $\op{Hom}_{A,D}(A,V)$ of all $A$- and $D$-linear morphisms $A\to V$. 
This contradicts the $k$-linear independency of $v_1,\dots,v_n$.
\epf

Let $A$ be a superalgebra. Recall from \eqref{SASeq0} that $\op{Hom}(D,A)$ is a super-vector space (over $R$). 
This is seen to be a $D$-superalgebra with respect to the convolution
product $*$ defined by 
\begin{equation}\label{DSAeq00}
f*g(d)=(-1)^{|f||d_{(2)}|}f(d_{(1)})g(d_{(2)}),
\end{equation}
and the $D$-action defined by 
\[
df(d')=f(d'd),
\]
where $d,d'\in D$ and $f,g\in\op{Hom}(D,A)$. 
If $A$ is a $D$-superalgebra, then 
there is associated the superalgebra morphism
\eq\label{DSAeq1a}
\rho=\rho_A : A \to \op{Hom}(D, A),\ \rho(a)(d)=da.
\eeq
This is, in fact, an algebra morphism in ${}_D\M$, whence  $\op{Hom}(D, A)$ turns into an $A$-algebra
in ${}_D\M$ through $\rho_A$. 

\begin{rem}\label{DSArem1}
There is another choice of defining the convolution product so as
\eq\label{DSAeq1b}
f*'g(d)=(-1)^{|d_{(1)}||g|}f(d_{(1)})g(d_{(2)}).
\eeq
But, $\rho$ turns to be a superalgebra morphism 
due to our choice of definition. 
This is based on the fact that $D$ acts on $A$ on the left, satisfying \eqref{DSAeq0c1}. 
If $D$ acted on the right, we should have defined so as \eqref{DSAeq1b}. 
\end{rem}

\begin{prop}\label{DSAprop1a}
Let $A$ be a $D$-superalgebra. Then
\eq\label{DSAeq2}
\mu : A\ot_{A^D} A \to \op{Hom}(D,A),\quad \mu(a\ot a')(d)=(da)a'
\eeq
defines an algebra morphism in ${}_D\M$, \tc{where $D$ is supposed to act on the source $A \ot_{A^D}A$ through the
left tensor factor}. If $A$ is $D$-simple, then $A^D$ is a field, and
this $\mu$ is injective. 
\end{prop}
\pf
Regard $\op{Hom}(D,A)$ as an $A$-algebra in ${}_D\M$, as above.
Since one sees that 
\[
\op{Hom}(D,A)^D=\op{Hom}(D/D^+,A)=A,
\]
where $D^+=\op{Ker}(\varepsilon)$ denotes the augmentation super-ideal of $D$, it follows
that the $\mu_V$ for $V=\op{Hom}(D,A)$ is precisely the one given above. This justifies
the notation, and derives the present proposition from Proposition \ref{DSAprop1}. 
\epf

Let $A$ be a $D$-superalgebra. Recall from \eqref{DASeq0c} the purely even Hopf sub-superalgebra $\D$ of $D$. 
With the restricted $\D$-action $A$ is a $\D$-superalgebra, in which $I_A$ is $\D$-stable. 
Therefore, $\overline{A}$ turns into a $\D$-algebra, that is, a purely even $\D$-superalgebra.
The projection $A \to \overline{A}$ restricts to the superalgebra morphism
\begin{equation*}
A^D \to \overline{A}^{\D}.
\end{equation*}

\begin{prop}\label{DSAprop2}
Assume
\begin{itemize}
\item[(D1)] $\dim (\g_D)_1<\infty$.
\end{itemize}
Let $A$ be a $D$-superalgebra such that
\begin{itemize}
\item[(i)] 
$I_A$ is nilpotent (this is the case if $A$ is Noetherian), 
\item[(ii)]
$A$ is $D$-simple, and 
\item[(iii)]
$\overline{A}$
is a field. 
\end{itemize}
Then the morphism $A^D \to \overline{A}^{\D}$ above is an isomorphism of fields. 
\end{prop}
\pf
By (ii) and (iii), $A^D$ and $\overline{A}^{\D}$ are fields. We have the commutative diagram
\[
\begin{xy}
(0,0)   *++{A\ot_{A^D}A}  ="1",
(30,0)  *++{\op{Hom}(D,A)} ="2",
(0,-16)  *++{\overline{A}\ot_{\overline{A}^{\D}}\overline{A}} ="3",
(30,-16) *++{\op{Hom}(\D,\overline{A}),}="4",
{"1" \SelectTips{cm}{} \ar @{>->}^{\hspace{-2mm} \mu}"2"},
{"3" \SelectTips{cm}{} \ar @{>->}^{\hspace{-2mm} \mu}"4"},
{"1" \SelectTips{cm}{} \ar @{->>}"3"},
{"2" \SelectTips{cm}{} \ar @{->>}"4"}
\end{xy}
\]
in which the horizontal morphisms $\mu$ are injective by (ii) and (iii), as is ensured by 
Proposition \ref{DSAprop1a}.
The vertical morphism on the RHS, which arises from the inclusion $\D \to D$ and the projection $A \to \overline{A}$,
is surjective. Let us write simply $\g$ for $\g_D$, and identify the super-coalgebra
$D$ with $\D \ot \wedge(\g_1)$ through an isomorphism such as in \eqref{DASeq0c0}. 
Then the kernel of the vertical morphism on the RHS is
\begin{equation}\label{DASeq20}
\op{Hom}(\D\ot (\wedge(\g_1)/R1), A)+ \op{Hom}(D,I_A). 
\end{equation}
We claim that this is a nilpotent super-ideal, whence in particular, it consists of nilpotent elements. 
To see this, it suffices to show that the two Hom sets in \eqref{DASeq20} are nilpotent super-ideals. 
Indeed, $\op{Hom}(D,I_A)$ is such by (i). 
In view of (D1), suppose $d=\dim \g_1\, (<\infty)$. Since $\wedge(\g_1)$ then vanishes
through the composite  
\[
\wedge(\g_1)\overset{\Delta_d}{\longrightarrow}(\wedge(\g_1))^{\ot(d+1)}\to (\wedge(\g_1)/R1)^{\ot(d+1)}
\]
of the $d$-th iterated coproduct $\Delta_d$ with the natural projection, it follows 
that 
\eq\label{DSAeq2a}
(\op{Hom}(\D\ot(\wedge(\g_1)/R1), A))^{d+1}=0. 
\eeq

By the assumption $\op{char}R=0$, $\D\ot \overline{A}$ is a smooth coalgebra over the field $\overline{A}$, so that
its dual $\overline{A}$-algebra $\op{Hom}(\D,\overline{A})$ is reduced. This, combined with the claim above, shows 
\[
(\op{Hom}(D,A))_{\op{red}}=\op{Hom}(\D,\overline{A}).
\]
In view of the commutative diagram above we see
\eq\label{DSAeq3}
(A\ot_{A^D}A)_{\op{red}}=\overline{A}\ot_{\overline{A}^{\D}}\overline{A}.
\eeq
On the other hand, the projection $A\to \overline{A}$ induces the superalgebra morphism
\[
A\ot_{A^D}A\to \overline{A}\ot_{A^D}\overline{A},
\]
whose kernel consists of nilpotents, as is easily seen. Since $\overline{A}\ot_{A^D}\overline{A}$
is reduced by the assumption $\op{char}R=0$, again, it follows that
\eq\label{DSAeq4}
(A\ot_{A^D}A)_{\op{red}}=\overline{A}\ot_{A^D}\overline{A}.
\eeq
We see from \eqref{DSAeq3} and \eqref{DSAeq4} that $\overline{A}\ot_{A^D}\overline{A}=\overline{A}\ot_{\overline{A}^{\D}}\overline{A}$,
which shows the desired result. 
\epf

An analogous argument using a commutative diagram shows the following.

\begin{prop}\label{DSAprop3}
Assume (D1), as in the preceding proposition. 
Let $A$ be a $D$-simple $D$-superalgebra.
Then we have the following.
\begin{itemize}
\item[(1)]
The nil radical $\sqrt{0}$ of $A$ is the largest $\D$-stable super-ideal of $A$, whence in particular, $A_{\op{red}}$
is a quotient $\D$-algebra of $A$. 
\item[(2)]
Assume in addition, 
\begin{itemize}
\item[(D3)] \tc{$D$ is irreducible, or explicitly, $D=U(\g_D)$.}
\end{itemize}
Then $A_{\op{red}}$ is an integral domain. 
\end{itemize}
\end{prop}
\pf
(1)
Choose arbitrarily a maximal \tc{$\D$-stable} super-ideal $\mathfrak{n}$ of $A$, and a maximal super-ideal $\mathfrak{m}$
of $A$ which includes $\mathfrak{n}$. \tc{Our aim} is to show $\mathfrak{n}=\sqrt{0}$. We have the commutative diagram 
\[
\begin{xy}
(0,0)   *++{A}  ="1",
(30,0)  *++{\op{Hom}(D,A/\mathfrak{m})} ="2",
(0,-16)  *++{A/\mathfrak{n}} ="3",
(30,-16) *++{\op{Hom}(\D,A/\mathfrak{m}),}="4",
{"1" \SelectTips{cm}{} \ar @{>->}"2"},
{"3" \SelectTips{cm}{} \ar @{>->}"4"},
{"1" \SelectTips{cm}{} \ar @{->>}"3"},
{"2" \SelectTips{cm}{} \ar @{->>}"4"}
\end{xy}
\]
in which the horizontal arrows indicate the composites of $\rho_A$ and $\rho_{A/\mathfrak{n}}$
(see \eqref{DSAeq1a}) with the morphisms arising from the projection $A\to A/\mathfrak{m}$; they are both injective
since $A$ is $D$-simple, and $A/\mathfrak{n}$ is $\D$-simple. Note that $A/\mathfrak{m}$ is
a field. The vertical morphisms on the RHS, which arises from the inclusion $\D \to D$,
is a surjection with nilpotent kernel by (D1), and the target $\op{Hom}(\D,A/\mathfrak{m})$ is reduced,
as is seen just as in the last proof. Therefore, we have
\[
(\op{Hom}(D,A/\mathfrak{m}))_{\op{red}}=\op{Hom}(\D,A/\mathfrak{m}),
\]
which, combined with the diagram above, implies $\mathfrak{n}=\sqrt{0}$.

(2) The additional assumption implies $\D=U(\g_0)$, so that $\op{Hom}(\D,A/\mathfrak{m})$ 
is an integral domain. As its subalgebra $A_{\op{red}}\, (=A/\mathfrak{n})$ is an integral domain. 
\epf


\subsection{Extension of $D$-actions}\label{EXD}
Let $A$ be a non-zero $D$-superalgebra.
We discuss the possibility of extending the $D$-action onto a localization $S^{-1}A$ of $A$, where 
$S$ is a multiplicative subset of $A_0$. 
Regard $S^{-1}A$ as an $A$-superalgebra with respect to the natural morphism $\iota : A \to S^{-1}A$, $\iota(a)=a/1$.

\begin{prop}\label{EXDprop1}
Assume
\begin{itemize}
\item[(D2)]
$D$ is pointed, or explicitly, $D=U(\g_D)\rtimes \Gamma_D$.
\end{itemize}Then we have the following. 
\begin{itemize}
\item[(1)] 
$S^{-1}A$ uniquely turns into an $A$-algebra in ${}_D\M$.
\item[(2)]
Let $B$ be an $A$-algebra in ${}_D\M$ through $f : A \to B$. Assume that for every $s\in S$, $f(s)$ is invertible
in $B_0$, so that $f$ uniquely extends to an $A$-superalgebra morphism, say, $\widehat{f} :  S^{-1}A\to B$. Then 
$\widehat{f}$ is $D$-superlinear. 
\end{itemize}
\end{prop}
\pf
(1)\
Define a superalgebra morphism, which is close to the $\rho_A$ in \eqref{DSAeq1a}, by
\[
\varrho : A \to \op{Hom}(D,S^{-1}A),\quad \varrho(a)(d)=\iota(da).
\]

Suppose $h \in \op{Hom}(D,S^{-1}A)$.
By (D2), $D$ is, as a coalgebra, \emph{co-generated} by the grouplike subcoalgebra $R\Gamma_D$,
that is, $D=\bigcup_{n>0} \wedge^n(R\Gamma_D)$, where $\wedge^n$ denotes the $n$-th wedge of 
Sweedler \cite[p.179]{Sw}.
It follows that if the restriction $h|_{R\Gamma_D}$ vanishes, then $h$ is locally nilpotent in the sense that it, restricted
to any finite-dimensional subcoalgebra, is nilpotent (with respect to the convolution product); see the argument
proving \eqref{DSAeq2a}. 
This fact essentially proves Takeuchi's Lemma (see \cite[Corollary 5.3.10]{HS}, for example) in the super context, which is now applied to
conclude: $h$ is invertible if and only if 
$h|_{R\Gamma_D}$ is invertible. The latter condition is equivalent to saying that for every $\gamma \in \Gamma_D$,
$h(\gamma)$ is invertible in $S^{-1}A$.

Let $s \in S$. For every $\gamma \in \Gamma_D$, $\varrho(s)(\gamma)\, (=\gamma s)$ is invertible in $S^{-1}A$,
since $\Gamma_D$ acts on $A$ as superalgebra automorphisms. By the result of the preceding \tc{paragraph,} 
$\varrho(s)$ in invertible, so that $\varrho$ uniquely extends to a superalgebra morphism, say,
\[
\widehat{\varrho} : S^{-1}A \to \op{Hom}(D,S^{-1}A).
\]
Obviously, $\widehat{\varrho}(\alpha)(1)=\alpha$ for all $\alpha
\in S^{-1}A$. 
To see that the $D$-action 
$d\alpha =\widehat{\varrho}(\alpha)(d)$ 
associated with $\widehat{\varrho}$
is the desired unique one, it remains to prove 
\[
(dd')\alpha=d(d'\alpha)\ \, \text{for all}\ \, d, d'\in D,\ \alpha \in S^{-1}A. 
\]
For this we need to prove that the composite of superalgebra morphisms
\[
S^{-1}A \overset{\widehat{\varrho}}{\longrightarrow} \op{Hom}(D,S^{-1}A)\longrightarrow
\op{Hom}(D\ot D,S^{-1}A),
\]
where the latter is the morphism induced from the product
$D\ot D\to D$ of $D$, is naturally identified with the one
\[
S^{-1}A \overset{\widehat{\varrho}}{\longrightarrow} \op{Hom}(D\ot D,S^{-1}A) \overset{\op{Hom}(D,\widehat{\varrho})}{\longrightarrow}
\op{Hom}(D, \op{Hom}(D,S^{-1}A)).
\]
Indeed, this holds since the two composites, further composed with
$\iota$, 
coincide with
the morphism 
\[
A \to \op{Hom}(D, \op{Hom}(D,S^{-1}A))=\op{Hom}(D\ot D, S^{-1}A)
\]
which associates to $a \in A$, 
the morphism $D\ot D \to S^{-1}A$ given by
\[
d\ot d'\mapsto \iota((dd')a)=
\iota(d(d'a)))\, (=d(\iota(d'a))).
\]

(2)\ We need to prove that the composite 
\[
S^{-1}A \overset{\widehat{\varrho}}{\longrightarrow} \op{Hom}(D, S^{-1}A)
\overset{\op{Hom}(D,\widehat{f})}{\longrightarrow}\op{Hom}(D,B).
\]
of superalgebra morphisms coincides with the one
\[ 
S^{-1}A \overset{\widehat{f}}{\longrightarrow} B
\overset{\rho_B}{\longrightarrow}\op{Hom}(D,B).
\]
Indeed, this holds since they, composed with $\iota$, coincide
with the morphism $A \to  \op{Hom}(D,B)$ which associates to
$a\in A$, the morphism
\[
D \to B,\quad d\mapsto f(da)=d f(a). 
\]
\epf

In what follows
the proposition above will be always applied when $S^{-1}A$ is $\Quot(A)$, in particular.


\section{SUSY PV extensions}\label{secSPV}

Throughout in this section we assume
\begin{itemize}
\item[(D1)] $\dim (\g_D)_1<\infty$, and
\item[(D2)] $D$ is pointed, or explicitly, $D=U(\g_D)\rtimes \Gamma_D$.
\end{itemize}


\subsection{The intermediate $D$-superalgebra with desirable properties}\label{IND}

Suppose that $L|K$ is an extension of objects (e.g., $D$-superalgebras, $D$-SUSY fields), that is,
an inclusion $L\supset K$ of those objects. 
By saying that $A$ is an \emph{intermediate object} of an extension $L|K$, we mean that 
$L|A$ and $A|K$ are extensions. 

Now, let $L|K$ be a fixed extension of $D$-superalgebras. We assume
\begin{itemize}
\item[(LK1)] $K$ is a $D$-SUSY field, 
\item[(LK2)] $L$ is Noetherian (whence $I_L$ is nilpotent), 
\item[(LK3)] $\overline{L}$ is a field,
\item[(LK4)] $\L|\K$ is finitely generated as a field extension, 
\item[(LK5)] $L$ and $K$ are both $D$-simple, and
\item[(LK6)] $L^D=K^D$. 
\end{itemize}

By (LK1) we may and we do regard $K$ as a $\K$-superalgebra through an 
arbitrarily chosen section of the projection $K \to \K$. By \tc{(LK3)}, $\L$ is a $\D$-\emph{field},
which means a $\D$-algebra that is a field, and so, $\L|\K$ is an extension of $\D$-fields. 

We set
\begin{equation}\label{INDeq1}
k:=L^D\, (=K^D).
\end{equation}
This is a field by the $D$-simplicity (LK5); the assumption together with (LK2--3) ensure
\begin{equation}\label{INDeq2}
k=\L^{\D}=\K^{\D}
\end{equation}
by Proposition \ref{DSAprop2}. 

Suppose that we are given an intermediate $D$-superalgebra $A$ of $L|K$ which satisfies
\begin{itemize}
\item[(A1)] $A\cdot (A\otimes_K A)^D=A\otimes_KA$, and
\item[(A2)] $L=\Quot(A)$. 
\end{itemize}
Notice that the $D$-invariants $(A\otimes_K A)^D$ in $A\otimes_KA$ form a $k$-sub-superalgebra, which maps onto itself 
through the supersymmetry $A\ot_K A\overset{\simeq}{\longrightarrow}A\ot_KA$.
The condition (A1) means that it generates the left $A$-supermodule $A\ot_K A$; this is equivalent 
to saying that $(A\otimes_K A)^D$
generates the right $A$-supermodule $A\ot_KA$, or in notation, 
\begin{itemize}
\item[(A1$'$)] $(A\otimes_K A)^D\cdot A=A\otimes_KA$,
\end{itemize}
as is seen by applying the supersymmetry above. 

\begin{theorem}\label{INDthm1}
$A$ is $D$-simple.
\end{theorem}

The proof of the theorem will complete
after proving Lemma \ref{INDlem4}. 

Let $P$ denote the image of $A$ under the projection
$L \to \L$; it is thus an intermediate $\D$-algebra of the extension $\L|\K$ of $\D$-fields. 

\begin{lemma}\label{INDlem1}
$\L|\K$ is a PV extension with $P$ its principal $\D$-algebra
in the sense of Takeuchi \cite[Definition 2.3]{T}; the term ``principal $\D$-algebra'' here is slightly modified
\tc{from the original definition}, see Remark \ref{INDrem1} below. 
\end{lemma}
\begin{rem}\label{INDrem1}
As was touched in Section \ref{INTbackground}, 
Takeuchi \cite{T} defined the notion of $C$-\emph{ferential algebras}, which generalizes differential algebras in the traditional PV theory; those are algebras
on which a cocomutative coalgebra $C=(C,\Delta,\varepsilon)$ equipped with a specific grouplike $1_C$ 
acts in an appropriate way, or in other words, they are
$T(C^+)$-\emph{algebras}. Here, $T(C^+)$ denotes the tensor algebra on
the augmented coideal $C^+:=\op{Ker}(\varepsilon)$  of $C$; it is alternatively expressed so as 
$T(C)/(1_{T(C)}-1_C)$, or namely, as the free bialgebra
$T(C)$ on the coalgebra $C$, \tc{modulo the relation which identifies} the identity element $1_{T(C)}$ with
the specific grouplike $1_C$. 
We should now understand Takeuchi's theory, replacing $T(C^+)$ with our $\D$; so we have modified 
what was originally called ``principal $C$-ferential algebra'' so as ``principal $\D$-algebra''. 
To prove some of the results,
Takeuchi \cite{T} poses the assumption that $C$ is pointed irreducible, which ensures that $T(C^+)$ is a 
pointed cocommutative (and non-commutative, unless $C=R1_C$) Hopf algebra. 
But those results except the ones in \cite[Section 4]{T} remain true, even if $T(C^+)$ is replaced by a 
pointed cocommutative Hopf algebra just as the present $\D$ of ours. 
\end{rem}

According to the cited definition the lemma above means that
\begin{itemize}
\item[(P1)] $P\cdot (P\ot_{\K}P)^{\D}=P\ot_{\K} P$, and
\item[(P2)] $\L$ is the quotient field of $P$.
\end{itemize}
We see that these indeed follow easily from the analogous conditions (A1--2) which we are assuming. 
 It is proved by \cite[Lemma 2.5]{T} that those intermediate 
$\D$-algebras of $\L|\K$ with the properties (P1-2) are unique;
it is called the \emph{principal} $\D$-\emph{algebra} of the PV extension \cite[p.492]{T}.
Set
\[
J:= (P\ot_{\K}P)^{\D} .
\]
Obviously, this is a $k$-subalgebra of $P\ot_{\K}P$.

\begin{prop}\label{INDprop1}
We have the following.
\begin{itemize}
\item[(1)] \cite[Theorem 2.11]{T}\ $P$ is $\D$-simple. 
\item[(2)] \cite[Corollary 3.4]{T}\ $P$ is finitely generated over $\K$, 
and $J$ is finitely generated over $k$.
\item[(3)] $J$ and $P$ uniquely turn into a Hopf algebra over $k$ and
a right $J$-comodule algebra, respectively, so that $P|\K$ is a $J$-Galois extension.
\item[(4)] $P$ is smooth and $J$-smooth over $\K$.
\end{itemize}
\end{prop}
\pf  
(3)
Takeuchi \cite[pp.490--493]{T} gives the way how they turn into such. When proving Lemma \ref{INDlem3} below,
we will reproduce it in our super context,
and verify the uniqueness in the assertion above. We remark that the uniqueness will not be used in our argument. 

(4)
This follows by Proposition \ref{HGEprop1} since $P|\K$ is $J$-Galois by Part 3.
\epf 

The $J$ above is called \emph{the Hopf algebra} of the PV extension \cite[p.493]{T}.

\begin{rem}\label{INDrem2}
Here are two remarks.
\begin{itemize}
\item[(1)]
The principal $C$-ferential (or $\D$-)algebra, defined intrinsically as above, corresponds to what is called a
\emph{PV ring}
in traditional PV theory; it is defined as a splitting ring for a given differential equation
with a certain desirable property. 
In Takeuchi's theory, the notion of those algebras above plays a crucial (or \emph{principal}) role, and in fact, they act
as Galois extensions, or as \emph{principal} homogeneous spaces (or torsors), in a geometric 
term; see Section \ref{INTbackground}. 
That is why Takeuchi gave the name to those algebras, the author guesses.
\item[(2)]
Recall that by (LK4), $\L|\K$ is assumed to be finitely generated. 
To define the notion of PV extensions Takeuchi \cite[Definition 2.3]{T} does not assume that the field extension is finitely generated. 
The present $\L|\K$ is what is called in \cite{T} a \emph{finitely generated} PV extension,  
see \cite[Theorem 3.3, Condition (a), p.501]{T}. 
\end{itemize}
\end{rem}

Choose arbitrarily a maximal $D$-stable super-ideal of $A$, and divide
$A$ by it. Let $B$ denote the resulting $D$-superalgebra, which is $D$-simple.  
To prove Theorem \ref{INDthm1}, 
we will finally show $A=B$. 

\begin{lemma}\label{INDlem2}
Let $r : A \to P$ denote the restriction of the natural projection $L \to \L$. 
\begin{itemize}
\item[(1)] We have
\[
P=A\red =B\red.
\]
\item[(2)] $r$ is a surjective $\K$-superalgebra morphism with nilpotent kernel, which splits. 
\item[(3)] Choose arbitrarily a $\K$-superalgebra section $P \to A$ of $r$, and regard $A$ as a $P$-superalgebra
through it. 
Then $L$ is the localization of $A$ by $P \setminus \{0\}$, or
in notation,
\begin{equation}\label{INDeq3}
L= A\otimes_P \overline{L}. 
\end{equation}
\end{itemize}
\end{lemma}
\pf
(1) 
The kernel $\op{Ker}(r)$ of $r$, being included in $\mathrm{Ker}(L \to \L)$, is nilpotent. 
Since $P$ is an integral domain, we have $P=A\red$. 
By Proposition \ref{DSAprop3} (1),
$B\red$ is a quotient $\D$-algebra of $B$. Therefore, the superalgebra morphism $P=A\red \to B\red$ induced from
the projection $A\to B$ is a $\D$-algebra morphism which is surjective.
Since $P$ is $\D$-simple by Proposition \ref{INDprop1} (1),
we have $P=B\red$. 

(2) 
It now suffices to see that $r$ splits. This is indeed the case since $P$ is smooth over $\K$ by Proposition \ref{INDprop1} (4). 

(3) 
Regard $P$ as a $\K$-sub-superalgebra of $A$ as prescribed. 
Then one sees that every non-zero element of $P$
is invertible in $L$ (since it is
so modulo the nilpotent $I_L$ by (LK2--3)), and is,
therefore, $A$-regular. Conversely, every $A$-regular element
of $A_0$ coincides with an non-zero element of $P$
modulo the nilpotent $\op{Ker}(r)$, since 
$A=P\oplus 
\op{Ker}(r)$. These prove the desired result. 
\epf

We will prove $\overline{A}=P$ in Corollary \ref{INDcor1} (1). 

Now, we set
\[
H:=(B\ot_{K}B)^D,
\]
which is seen to be a $k$-sub-superalgebra of $B\ot_K B$.

\begin{lemma}\label{INDlem3}
$H$ and $B$ uniquely turn into a Hopf superalgebra over $k$ and
an algebra in $\M^H$, respectively, so that Conditions (HG1--2)
in Definition \ref{HGEdef1}, with $A|K$ replaced by the present $B|K$, are
satisfied. 
\end{lemma}

To prove the lemma, and also for later use, 
note that since $B$ is an algebra in the monoidal category ${}_D\M$, we have the monoidal category 
\[
{}_B({}_D\M)_B=
({}_B({}_D\M)_B,\ot_B, B)
\]
of $(B,B)$-bimodules in ${}_D\M$.  
By a $B$-\emph{coring} we mean a coalgebra in ${}_B({}_D\M)_B$. 
We regard the $(B,B)$-bimodule $B\ot_KB$ in ${}_D\M$, as a $B$-coring with 
respect to the familiar structure morphisms 
\[
\mathbf{\Delta} : B\ot_KB\to B\ot_KB\ot_KB=(B\ot_KB)\ot_B(B\ot_KB),\ \, 
\mathbf{e} : B\ot_KB\to B
\]
defined by
\eq\label{INDeq3a}
\mathbf{\Delta}(b\ot_K c)=b\ot_K1\ot_Kc,\quad \mathbf{e}(b\ot_Kc)=bc.
\eeq
We remark that the definition makes sense with $K\hookrightarrow B$ replaced by any algebra morphism in ${}_D\M$. 

\begin{proof}[Proof of Lemma \ref{INDlem3}]
First, \tc{notice from Proposition \ref{DSAprop1} (1)} that $B^D$ is a field since $B$ is $D$-simple. We claim
\[ B^D=k. \]
Indeed, the natural morphisms $A \to B \to B_{\op{red}}=P$ induce $A^D\to B^D\to P^{\D}$. The claim follows since
$A^D=P^{\D}=k$ by (LK6) and \eqref{INDeq2}. 

By the claim above, it follows from Proposition \ref{DSAprop1} \tc{(2)} and its opposite-sided analogue that
the left and the right $B$-multiplications 
\begin{equation}\label{INDeq4}
\mu : B\otk H\to B\ot_KB,\quad \mu': H\otk B\to B\ot_KB
\end{equation}
are both injective, which are indeed isomorphic as is easily seen from 
(A1) and (A1$'$). 
An iterated use of the isomorphism $\mu$ 
\begin{equation}\label{INDeq4a}
B\otk H \otk H \overset{\mu \otk \op{id}}{\longrightarrow} 
B\ot_K B\otk H \overset{\op{id}\ot_{\! B}\, \mu}{\longrightarrow}B\ot_KB\ot_KB
\end{equation}
shows
\begin{equation}\label{INDeq5}
H\otk H =(B\ot_KB\ot_K B)^D.
\end{equation}

Recall from \eqref{INDeq3a} the structure morphisms of the $B$-coring $B\ot_KB$. 
Restricted onto $D$-invariants, the morphisms induce superalgebra morphisms
\[
\Delta : H \to H\ot_kH, \quad \varepsilon : H \to k.
\]
We can verify that these satisfy the coalgebra-axioms, whence $(H,\Delta,\varepsilon)$ is a super-bialgebra, by using the result
\[
H\otk H\otk H =(B\ot_KB\ot_K B\ot_K B)^D,
\]
which is analogous to \eqref{INDeq5} and is shown similarly. 
The supersymmetry $B\ot_KB\to B\ot_KB$, restricted onto $D$-invariants, induces through the isomorphisms
in \eqref{INDeq4} a superalgebra morphism 
\[
\cS : H\to H.
\]
We can verify that this is an antipode of $H$, whence $H$ is, in fact, 
a Hopf superalgebra. The verifications above are done just as was 
done in \cite{T}. (One might try, completely imitating the argument of \cite{T},
to make $(A\ot_KA)^D$ into a Hopf superalgebra. But this is
impossible, and we had to take $B$ and to do with it, as above. For 
we cannot obtain, as Takeuchi \cite{T} did, the isomorphism $A\otk H \to A\ot_KA$ which is analogous to the one in \eqref{INDeq4}, by using the $D$-simplicity of $L$, since $K$ is not a field, now, whence
the argument at \cite[Page 493, lines 5--6]{T} which regards $(A\ot_KA)^D$ as to be included in $(L\ot_KA)^D$ cannot apply.)

Notice that 
\begin{equation}\label{INDeq50}
b \mapsto 1\ot_Kb,\quad B \to B\ot_KB=B\ot_B(B\ot_KB)
\end{equation}
is the unique superalgebra-morphism with which $B$ is a right comodule over the $B$-coring $B\ot_KB$.
Let $\theta_B : B\to B\otk H$ denote the composite of this morphism 
with the inverse $\mu^{-1}:B\ot_KB\overset{\simeq}{\longrightarrow}B\otk H$
of the isomorphism $\mu$ in \eqref{INDeq4}. It is easy to see that 
$(B,\theta_B)$ an algebra in $\M^H$, and
$B|K$ satisfies Conditions (HG1--2); to verify (HG1), see Remark \ref{HGErem1} (2). 

As for the uniqueness of construction, we first see that the super-coalgebra structure on $H$ must be as above,
since it is to give rise, through the isomorphism $\mu$ in \eqref{INDeq4} and the one in \eqref{INDeq4a},
to the $B$-coring structure on $B\ot_KB$ defined by \eqref{INDeq3a}. 
Next, we see that the $H$-supercomodule structure on $B$
must be the $\theta_B$ above, since it is to give rise, through $\mu$, to the $B\ot_KB$-comodule structure in 
\eqref{INDeq50}. 
\end{proof}

\begin{lemma}\label{INDlem4}
Recall from Lemma \ref{INDlem2} (1) that $P=B\red$.
\begin{itemize}
\item[(1)] 
The natural projection $B\ot_k B \to P\ot_k P$ induces 
a surjective algebra morphism in ${}_{\D}\M$,
\[
B\ot_K B \to P\ot_{\K}P,
\]
and it restricts to a Hopf superalgebra morphism over $k$
\[
H=(B\ot_K B)^D\to 
(\overline{B}\ot_{\K}\overline{B})^{\D}=J.
\]
These induce a $k$-algebra isomorphism 
\eq\label{INDeq5a}
(B\ot_K B)\red \overset{\simeq}{\longrightarrow} P\ot_{\K}P
\eeq
and a Hopf algebra isomorphism
\begin{equation}\label{INDeq6a}
\overline{H}\overset{\simeq}{\longrightarrow} J,
\end{equation}
respectively. 
\item[(2)] 
We identify so as $\overline{H}=J$ through the last isomorphism. 
The projection $B \to P$, which is seen to be a $\K$-algebra 
morphism in $\M^J$, splits. 
If we choose arbitrarily a section $s : P \to B$ of the last projection, then 
\begin{equation}\label{INDeq7}
B^{\op{co} J}\otimes_{\K}P
\overset{\simeq}{\longrightarrow} B,\quad b\ot x \mapsto b s(x),
\end{equation}
is a $B^{\op{co}J}$-algebra isomorphism in $\M^J$. 
\end{itemize}
\end{lemma}
\pf
(1)\
It is easy to see that we have the first two morphisms. The former induces the isomorphism \eqref{INDeq5a},
since (i) its kernel consists of nilpotents, and (ii)~$P\ot_{\K}P$ is reduced; here, (ii) holds, since we see
from Proposition \ref{INDprop1} (2) and (4) that $P\ot_{\K}P$ is finitely generated and smooth over $\K$. 
The latter is, indeed, a Hopf superalgebra morphism, as is seen from 
the constructions of $H$ and $J$. This is surjective, since it, tensored with the projection $B\to P$,
constitutes the commutative diagram
\[
\begin{xy}
(0,0)   *++{B\otk H}  ="1",
(30,0)  *++{B\ot_K B} ="2",
(0,-16)  *++{P\otk J} ="3",
(30,-16) *++{P\ot_{\K} P,}="4",
{"1" \SelectTips{cm}{} \ar @{->}_{\mu}^{\simeq} "2"},
{"3" \SelectTips{cm}{} \ar @{->}_{\mu}^{\simeq} "4"},
{"1" \SelectTips{cm}{} \ar@{->} "3"},
{"2" \SelectTips{cm}{} \ar @{->>} "4"}
\end{xy}
\]
where the vertical arrow on the RHS indicates a natural surjection. 
We see just as for
\eqref{INDeq5a} that another natural surjection $B\otk H\to P\otk \overline{H}$ induces an isomorphism
\eq\label{INDeq7a}
(B\otk H)\red \overset{\simeq}{\longrightarrow} P\otk \overline{H},
\eeq
in view of the following: (i)~$P\otk \overline{H}=P\ot_{\K}(\K\otk \overline{H})$,
(ii)~$P$ is finitely generated and smooth over $\K$, by Proposition \ref{INDprop1} (2) and (4),
and (iii)~$\K\otk \overline{H}$ is a filtered union of finitely generated Hopf subalgebras over $\K$ (of characteristic zero),
which are necessarily smooth, and hence remain to be reduced, being tensored with $P$ over $\K$. 
The isomorphisms \eqref{INDeq7a} and \eqref{INDeq5a} show that 
$(B\otk H)\red$ and $(B\ot_K B)\red$ and are
$\D$-algebra quotients of $B\otk H$ and of $B\ot_K B$, respectively, and they are mutually isomorphic. 
The isomorphism restricts to the desired isomorphism $\overline{H}=(P\ot_k \overline{H})^{\D}\simeq (P\ot_{\K} P)^{\D}=J$. 

(2)\
Recall from Proposition \ref{INDprop1} (4) that $P$ is $J$-smooth over $\K$. 
The projection $B \to P$ is a surjective algebra-morphism in $\M^{J}$, whose kernel is, 
as an epimorphic image of $\op{Ker}(r)$ (see Lemma \ref{INDlem2} (2)), nilpotent. Hence it splits. 
Through an arbitrarily chosen section, $B$ is regarded as a $P$-algebra in $\M^J$, or an algebra
in the monoidal category $({}_P(\M^J),\ot_P,P)$. 
Since $P|\K$ is $J$-Galois by Proposition \ref{INDprop1} (3),
we have the monoidal equivalence
\[
({}_{\K}\M,\ot_{\K},\K)\approx ({}_P(\M^J),\ot_P,P),
\]
such as in \eqref{HGEeq3a}; \tc{the associated natural isomorphism}
\[
M^{\op{co}J}\ot_{\K}P\simeq M,\quad M \in {}_P(\M^J)
\]
is precisely \eqref{INDeq7} when $M=B$, in particular.
\epf
 
\begin{proof}[Proof of Theorem \ref{INDthm1}]
Let $q:A \to B$ denote the projection. We wish to show that $q$ is injective. 
Let $x$ be a non-zero-element of $P$. 
Regard $P \subset A$ (whence $x \in A)$ through a section of $A \to P$, as in Lemma \ref{INDlem2} (3),
and recall \eqref{INDeq3}. \tc{In the situation of Lemma \ref{INDlem4} (2), $s(x)$ is $B$-regular,
as is seen from \eqref{INDeq7}. Since $q(x)\equiv s(x)$ modulo nilpotents, it follows} 
that $q$ extends to a superalgebra morphism $L \to \Quot(B)$. We claim that
this is injective; this will prove the desired injectivity of $q$. 
Let us apply the two parts of Proposition \ref{EXDprop1}. 
Part 1 (applied to $B$) shows that $\Quot(B)$ uniquely turns into a $B$-algebra 
in ${}_D\M$. 
Part 2 then shows that the extended $L\to \Quot(B)$ is an algebra morphism 
in ${}_D\M$. 
Since $L$ is $D$-simple by (LK5),  the claim is proven. 
\end{proof}

Since we have proved $A=B$,  the 
results on $B$ obtained above hold with $B$ replaced by $A$. 

\begin{prop}\label{INDprop2}
We have the following.
\begin{itemize}
\item[(1)]
$A|K$ satisfies Conditions (HG1--2) in Definition \ref{HGEdef1}.
\item[(2)] 
$A$ is Noetherian.
\item[(3)] 
The Hopf superalgebra $H$ over $k$ is finitely generated.
\end{itemize}
\end{prop}
\pf
(1)\
This holds by Lemma \ref{INDlem3}.

(2)\
We identify so as $\overline{H}=J$ through \eqref{INDeq6a}, as before. 
By \eqref{INDeq7} we have the isomorphism $A^{\op{co}\overline{H}} \ot_{\K}P\simeq A$, which extends to 
\[ A^{\op{co}\overline{H}} \ot_{\K}\L\simeq L \]
by \eqref{INDeq3}.
The latter isomorphism, combined with (LK2), proves that 
$A^{\op{co}\overline{H}}$ is Noetherian. 
Since $P$ is finitely generated over $\K$ by Proposition \ref{INDprop1} (2), 
the former isomorphism shows that $A$ is Noetherian.  

(3)\ 
Since $\overline{H}\, (=J)$ is finitely generated by Proposition \ref{INDprop1} (2), it suffices
to prove $\dim_k(\mathfrak{w}_H)< \infty$, in view of \eqref{HGEeq2}.

To the isomorphism $\Theta_A:A \ot_KA\overset{\simeq}{\longrightarrow} A\otk H$ (see \eqref{HGEeq3}),
apply the base extension $\L\ot_A$ along $A\to A\red=P \hookrightarrow \L$. 
Then we have
\[
\L\ot_{\K}\widetilde{A} \simeq \L\ot_k H, \]
where we have set
\begin{equation}\label{INDeq8}
\widetilde{A}:=A/K_1A.
\end{equation}
Since $\widetilde{A}$ is Noetherian by Part 2, it follows by (LK4) that 
$\L\ot_k H\, (\simeq
\L \ot_{\K}\widetilde{A})$ is Noetherian. 
This implies $\dim_k(\mathfrak{w}_H)< \infty$. 
\epf

We have the functor
\begin{equation}\label{INDeq9}
\K\ot_{K} : ({}_{K}\M, \ot_K, K) \to ({}_{\K}\M, \ot_{\K},\K),
\end{equation}
which is easily seen to be monoidal. 
One sees that the $\widetilde{A}$ in \eqref{INDeq8} is, in fact, $\K\ot_KA$, 
confirming that it is a $\K$-superalgebra. Moreover, we see that $\widetilde{A}$ is a $\K$-algebra in $\M^H$. 

\begin{lemma}\label{INDlem5}
$\widetilde{A}|\K$ is an $H$-Galois extension. 
\end{lemma}
\pf
Apply $\K\ot_K$ to the isomorphism $\Theta_A:A \ot_KA\overset{\simeq}{\longrightarrow} A\otk H$.
Then we have
\[
\widetilde{A}\ot_{\K}\widetilde{A}\overset{\simeq}{\longrightarrow}\widetilde{A}\ot_k H.
\]
This shows that $\widetilde{A}|\K$ is $H$-Galois, due to the fact $\K$ is a field. 
We emphasize that the fact ensures (HG1) and (HG3); see Remark \ref{HGErem1} (2). 
\epf

\begin{prop}\label{INDprop3}
$A|K$ is an $H$-Galois extension, and has, therefore, the properties shown in 
Proposition \ref{HGEprop1} and Theorem \ref{HGEthm1}. 
\end{prop}
\pf
By Proposition \ref{INDprop2} (1) it remains to prove that $A|K$ satisfies (HG3). 
By the preceding lemma it follows from Proposition \ref{HGEprop1} (applied to $\widetilde{A}|\K$) that
$\widetilde{A}$ is $H$-smooth over $\K$. Therefore, the natural, surjective $\K$-algebra-morphism $A \to \widetilde{A}$
in $\M^H$ splits. Through an arbitrarily chosen section, regard $A$ as an $\widetilde{A}$-algebra in $\M^H$. 
The monoidal equivalence \eqref{HGEeq3a}, now presented as 
\[ ({}_K\M,\ot_K,K)\approx ({}_{\widetilde{A}}(\M^H),\ot_{\widetilde{A}},\widetilde{A}),\]
gives a $K$-algebra isomorphism
$K\ot_{\K}\widetilde{A}\simeq A$ in $\M^H$; see the proof of Lemma \ref{INDlem4} (2). This verifies (HG3), that is, 
the $K$-faithful flatness of $A$. 
\epf

\begin{corollary}\label{INDcor1}
The following result from Proposition \ref{INDprop3}.
\begin{itemize}
\item[(1)] We have $\overline{A}=P$.
\item[(2)] $A$ is finitely generated over $K$. 
\end{itemize}
\end{corollary}
\pf
(1) 
By Proposition \ref{INDprop2} (2), $A$ is Noetherian. By Proposition \ref{INDprop3} it is smooth over
the $\K$-smooth superalgebra $K$, and is so over $\K$. 
\tc{Therefore, $\overline{A}$ is Noetherian, and is smooth over $\K$ by Proposition \ref{SUFprop1}. 
We then have $\overline{A}=A_{\op{red}}$, 
whence $\overline{A}=P$ by Lemma \ref{INDlem2} (1).} 

(2)
This follows from the isomorphism $A \simeq K\ot_{\overline{K}}\overline{A}\ot_k \wedge_k(\mathfrak{w}_H)$
given in \eqref{HGEeq5}, since $\overline{A}\, (=P)$ is finitely generated over $\K$,
by Proposition \ref{INDprop1} (2), and we have $\dim_k(\mathfrak{w}_H)<\infty$, as was shown in the proof of 
Proposition \ref{INDprop2} (3). 
\epf


\subsection{Definition of SUSY PV extensions}\label{DEF}

Let $L|K$ be as in the preceding subsection. Thus it is an extension of algebras in ${}_D\M$
which satisfies (LK1--6).

\begin{prop}\label{DEFprop1}
We have the following.
\begin{itemize}
\item[(1)] 
An intermediate $D$-superalgebra $A$ of $L|K$ which satisfies (A1--2), if it exists, is unique. 
\item[(2)] 
If there exists an intermediate $D$-superalgebra such as above, then $L$ is
a $D$-SUSY fields, so that $L|K$ is an extension of $D$-SUSY
fields. Moreover, 
$L|K$ is admissible; see Definition \ref{SUFdef1a}.
\end{itemize}
\end{prop}
\begin{proof}
(1)\
(cf. the proof of \cite[Lemma 2.5]{T}.) 
Let $B$ be another one. Since $AB$ ($\subset L$) is seen to
satisfy (A1--2), we may suppose $A \subset B$ by replacing $B$ with $AB$. 

Set $H_A:=(A\ot_KA)^D$,\ $H_B:=(B\ot_KB)^D$. Then, by using the $K$-flatness of $A$ and of $B$, we see $H_A\subset H_B$
and that $H_A$ is a Hopf sub-superalgebra of $H_B$. Moreover, we have the natural commutative diagram
\[
\begin{xy}
(0,0)   *++{B\ot_KA}  ="1",
(30,0)  *++{B\otk H_A} ="2",
(0,-16)  *++{B\ot_KB} ="3",
(30,-16) *++{B\otk H_B}="4",
{"1" \SelectTips{cm}{} \ar @{->}^{\simeq} "2"},
{"3" \SelectTips{cm}{} \ar @{->}_{\Theta_B}^{\simeq} "4"},
{"1" \SelectTips{cm}{} \ar@{->} "3"},
{"2" \SelectTips{cm}{} \ar @{_(->} "4"}
\end{xy}
\]
where top isomorphism is the base extension $B\ot_A$ of $\Theta_A$. 
Since $H_A\to H_B$ is faithfully flat by \cite[Theorem 3.10 (1)]{M}, and $K\to B$ is faithfully flat, it follows that $A\to B$ is faithfully flat.
To prove $A\supset B$, 
let $b \in B$. Then by (A2) for $A$, it is of the form $b=a/s$, where $a \in A$, $s\in A_0$, and $s$ is $A$-regular. \tc{Notice $B\ot_AB \subset L\ot_AL$. Computing in $L\ot_AL$, we have} 
\[
b\ot 1=\frac{a}{s}\ot \frac{s}{s}=\frac{s}{s}\ot \frac{a}{s}=1\ot b
\]
\tc{in $B\ot_AB$.}
The faithful flatness of $A\to B$ shows $b \in A$, as desired. 

(2)\
By \eqref{HGEeq5}, the $K$-superalgebra $A$ is identified with $K\ot_{\K}\overline{A}\otk \wedge_k(\mathfrak{w}_H)$. 
Since $L=A\ot_{\overline{A}}\L$ by \eqref{INDeq3} and
Corollary \ref{INDcor1} (1), 
we have $L=K\ot_{\K}\L\otk \wedge_k(\mathfrak{w}_H)$, which proves
that $L$ is a $D$-SUSY field. 
The argument also shows $\op{gr} L\supset \op{gr} A\supset \op{gr} K$, which proves
the admissibility. 
\end{proof}

\begin{definition}\label{DEFdef1}
Let $L|K$ be an extension of $D$-SUSY fields.
\begin{itemize}
\item[(1)]\
We say that $L|K$ is \emph{of finite type} if
\begin{itemize}
\item[(LK4)] $\L|\K$ is a finitely generated field extension.
\end{itemize}
If this is the case, Conditions (LK1--4) of those given at the beginning of Section \ref{IND} are satisfied. 
\item[(2)]\
Suppose that the remaining 
\begin{itemize}
\item[(LK5)] $L$ and $K$ are both $D$-simple, and
\item[(LK6)] $L^D=K^D$. 
\end{itemize}
of the conditions just mentioned are satisfied. 
We say that $L|K$ is a \emph{SUSY PV extension}, if $L|K$ is of finite type, and if there exists
necessarily uniquely (see Proposition \ref{DEFprop1} (1)), 
an intermediate $D$-superalgebra $A$ of $L|K$ such that 
\begin{itemize}
\item[(A1)] $A\cdot (A\otimes_K A)^D=A\otimes_KA$, and
\item[(A2)] $L=\Quot(A)$. 
\end{itemize}
This $A$ is called the \emph{principal $D$-superalgebra} of the SUSY PV extension. 
\end{itemize}
\end{definition}

Let us summarize \tc{the so far obtained results on the extensions just defined.} 

\begin{theorem}\label{DEFthm1}
Let $L|K$ be a SUSY PV extension with the principal $D$-superalgebra
$A$. Set
\begin{equation}\label{DEFeq1}
k:=L^D\, (=K^D),\quad H:=(A\ot_KA)^D. 
\end{equation}
\begin{itemize}
\item[(1)]
$L|K$ is admissible as an extension of SUSY fields.
\item[(2)]
$H$ is a $k$-sub-superalgebra of $A\ot_KA$, and is finitely generated over $k$.
\item[(3)]
$A$ is finitely generated over $K$. 
\item[(4)]
$H$ and $A$ uniquely turn into a Hopf superalgebra over $k$ and an algebra in $\M^H$, respectively, so that $A|K$ is an $H$-Galois extension.
Moreover, $A$ is  
presented by the isomorphisms
\eqref{HGEeq4}--\eqref{HGEeq5}.
\item[(5)]
With respect to the induced $\D$-actions, $\L|\K$ is a PV extension with $k=\L^{\D}\, (=\K^{\D})$
in the sense of Takeuchi \cite{T} (see also Remark \ref{INDrem1}). It has $\overline{A}$ and $\overline{H}$
as the principal $\D$-algebra and as the Hopf algebra, respectively, and $\overline{A}|\K$ is a purely even $\overline{H}$-Galois 
extension. 
\end{itemize}
\end{theorem}

\begin{definition}\label{DEFdef2}
In the situation of the theorem above, we call $H$ \emph{the Hopf superalgebra} of $L|K$, and let
\begin{equation}\label{DEFeq2}
\G(L|K)=\op{Sp}_k(H)
\end{equation}
denote the affine algebraic supergroup scheme over $k$ which is
represented by $H$, calling it \emph{the Galois supergroup} of the SUSY PV extension. 
We will also say so that $(L|K, A, H)$ \emph{is a SUSY PV extension}, just as we say (see
\cite{AM}, \cite{AMT})
that $(\L|\K, \overline{A}, \overline{H})$ is a PV extension.
\end{definition}

Suppose that $(L|K,A,H)$ is a SUSY PV extension with $k:=L^D\, (=K^D)$. 
Let $\mathsf{Aut}_{D}(A|K)$ denote the group-valued
functor defined on the category $\mathtt{SAlg}_k$ of superalgebras over $k$, which assigns to each $T\in \mathtt{SAlg}_k$,
the group $\op{Aut}_{D}(A\otk T|K\otk T)$ of all $D$-superalgebra automorphisms of $A\otk T$ over $K\otk T$. 
The definition makes sense with $A$ replaced by $L$.

Notice from Definition \ref{DEFdef2} 
that $\G(L|K)$ assigns to each $T$ as above, the group of all superalgebra morphisms 
$H\to T$ over $k$. 

\begin{theorem}[\tu{see \cite[Theorem A.2]{T}}]\label{DEFthm2}
We have the following.
\begin{itemize}
\item[(1)] 
For any $T\in \mathtt{SAlg}_k$, an endomorphism of the $D$-superalgebra $A\otk T$ over $K\otk T$ is necessarily 
an automorphism. 
\item[(2)]
Let $T \in \mathtt{SAlg}_k$ and $g \in \G(L|K)(T)$, and define $\phi_T(g) : A \otk T\to A \otk T$ to be the $T$-superlinear
extension of the composite
\[
A \overset{\theta_{\! A}}{\longrightarrow} A \otk H \overset{\op{id} \ot g}{\longrightarrow} A \otk T.
\]
Then $\phi_T(g)\in \op{Aut}_{D}(A\otk T|K\otk T)$, and the assignment $g \mapsto \phi_T(g)$ gives rise to an isomorphism
of group-valued functors
\[
\phi : \G(L|K)\overset{\simeq}{\longrightarrow}\mathsf{Aut}_{D}(A|K).
\]
When $T=k$, in particular, we have
\eq\label{DEFeq3}
\G(L|K)(k) \simeq \op{Aut}_{D}(A|K)=\op{Aut}_{D}(L|K).
\eeq
\end{itemize}
\end{theorem}
\pf 
One can prove this just as proving \cite[Theorem A.2]{T} referred to above. 
We only prove the equality in \eqref{DEFeq3}, which is analogous to 
the one \eqref{DEFeq4} below, reproduced from \cite{T}. By the uniqueness of the principal
$D$-superalgebra (see Proposition \ref{DEFprop1} (1)) 
one has the restriction $\op{Aut}_{D}(L|K)\to \op{Aut}_{D}(A|K)$, which is isomorphic,
since the $D$-action on $A$ uniquely extends onto $L$; see 
Proposition \ref{EXDprop1} (1). 

Here is a remark on \cite[Theorem A.2]{T}: 
in order to use an analogous unique extension property (see \cite[Proposition 1.9]{T}), one should add the assumption that $C$ is pointed irreducible to obtain the equation
\eq\label{DEFeq4}
\op{Aut}_{\text{$C$-fer.$K$-alg}}(A)=\op{Aut}_{\text{$C$-fer.$K$-alg}}(L)
\eeq
on Page 508, line --12 of \cite{T}, as far as the author sees.
\epf


\section{Galois correspondence}\label{GAL}

Throughout in this section we assume that $D$ satisfies (D1--2), as in the
preceding section, and we let
$(L|K,A,H)$ be a SUSY PV extension with $k=L^D\, (=K^D)$. 

Notice from the paragraph following Lemma \ref{INDlem3} that we have the $A$-coring $A\ot_KA$ 
and the $L$-coring $L\ot_KL$; \tc{in what follows,} the $L$-coring $L\ot_M L$ will also appear, where $M$ is an 
intermediate $D$-SUSY field of $L|K$.
These are coalgebras in ${}_A({}_D\M)_A$ or in ${}_L({}_D\M)_L$. 

We call a sub-object $\cI$ of $A\ot_KA$ in ${}_A({}_D\M)_A$, a $D$-\emph{coideal}, 
if it satisfies
\[
\mathbf{\Delta}(\cI)\subset \cI\ot_
A(A\ot_KA)+(A\ot_KA)\ot_A\cI,\quad 
\mathbf{e}(\cI)=0.
\]
(We thus add ``$D$-'', to emphasize it is $D$-stable.)  This is precisely a sub-object $\cI$ such that
the quotient $(A \ot_K A)/\cI$ naturally turns into an $A$-coring. Similarly, a $D$-\emph{coideal} of $L\ot_KL$
is defined. 

\begin{prop}\label{GALprop1}
There are natural one-to-one correspondences among the following three sets:
\begin{itemize}
\item the set of all Hopf super-ideals of $H$. 
\item the set of all $D$-coideals of $A\ot_KA$.
\item the set of all $D$-coideals of $L\ot_KL$.
\end{itemize}
\end{prop}

We denote these three sets, \tc{from top to bottom}, by
\eq\label{GALeq-3}
\I^{\tu{Hopf}}(H),\quad \I^{\tu{co}}_D(A\ot_K A),\quad \I^{\tu{co}}_D(L\ot_K L).
\eeq

\pf
When $X$ is a superalgebra (resp., an algebra in ${}_D\M$), 
we let $\I(X)$ (resp., $\I_D(X)$) denote the set of all super-ideals (resp., all $D$-stable super-ideals) of $X$. 
Recall $H=(A\ot_K A)^D$. 
We claim that the map
\eq\label{GALeq-2}
\I(H)\to \I_D(A\ot_KA),\ \a \mapsto A\cdot \a
\eeq
is a bijection.
Indeed, given $\cI \in \I_D(A\ot_KA)$, we have the natural isomorphism 
\eq\label{GALeq-1}
\begin{xy}
(0,8)   *++{A\otk H}  ="1",
(40,8)  *++{A\ot_K A} ="2",
(0,-8)  *++{A\ot_k((A\ot_KA)/\cI)^D} ="3",
(40,-8) *++{(A\ot_KA)/\cI,}="4",
{"1" \SelectTips{cm}{} \ar @{->}_{\mu}^{\simeq} "2"},
{"3" \SelectTips{cm}{} \ar@{>->}_{\qquad\mu} "4"},
{"1" \SelectTips{cm}{} \ar@{->} "3"},
{"2" \SelectTips{cm}{} \ar @{->>} "4"}
\end{xy}
\eeq
where $H$ maps into $((A\ot_KA)/\cI)^D$ through the vertical morphism on the LHS. 
It follows that the bottom morphism, being injective by the $D$-simplicity of $A$, is isomorphic, and 
there uniquely exists $\a \in \I(H)$ such that
\eq\label{GALeq0}
H/\a=((A\ot_KA)/\cI)^D.
\eeq
Therefore, $\cI$ uniquely arises from this $\a$, so that $\cI=A\cdot \a$. This proves the claim.

Suppose $Y=L\ot_K A$, $A\ot_KL$ or $L\ot_KL$. Since $Y$ is a localization of $A\ot_KA$ by a central multiplicative subset, we
can regard $\I(A\ot_K A)\supset \I(Y)$. But we have
\eq\label{GALeq0a}
\I_D(A\ot_KA)=\I_D(L\ot_KA)=\I_D(A\ot_KL)=\I_D(L\ot_KL).
\eeq
Indeed, since we have the bijection $\I(H)\to \I_D(L\ot_KA),\ \a\mapsto L\cdot \a$ analogous to \eqref{GALeq-2}, it follows that
$\I_D(A\ot_KA)=\I_D(L\ot_KA)$. The opposite-sided argument shows $\I_D(A\ot_KA)=\I_D(A\ot_KL)$.
The two equalities conclude the remaining. A similar argument, which uses some bijections including the analogue $\I(H\otk H)\overset{\simeq}{\longrightarrow}
\I_D(A\ot_KA\ot_KA)$ of \eqref{GALeq-2}, proves some equations that are analogous to \eqref{GALeq0a}; they include
\eq\label{GALeq0b}
\I_D(A\ot_KA\ot_KA)=\I_D(L\ot_KA\ot_KL)=\I_D(L\ot_KL\ot_KL).
\eeq

Let $\a\in\I(H)$, and suppose $\cI=A\cdot \a$. 
An iterated use of the bottom isomorphism in \eqref{GALeq-1} combined with \eqref{GALeq0} shows
\begin{align*}
A\otk(H/\a)\otk (H/\a)&\overset{\simeq}{\longrightarrow}((A\ot_KA)/\cI)\ot_A((A\ot_KA)/\cI),\\
(H/\a)\otk (H/\a)&=[((A\ot_KA)/\cI)\ot_A((A\ot_KA)/\cI)]^D.
\end{align*}
Therefore, the structure morphisms of $H$ induce
\[
k\leftarrow H/\a \to (H/\a)\otk(H/\a)
\]
if and only if the structure morphisms of $A\ot_KA$ induce
\[
A\leftarrow (A\ot_KA)/\cI \to ((A\ot_KA)/\cI)\ot_A((A\ot_KA)/\cI). 
\]
This, combined with \tc{Lemma \ref{GALlem1} below which proves the general fact} that
a super-bi-ideal of $H$ is necessarily a Hopf super-ideal,
shows that the bijection \eqref{GALeq-2} restricts to a bijection
\[
\I^{\tu{Hopf}}(H)\overset{\simeq}{\longrightarrow}\I^{\tu{co}}_D(A\ot_K A).
\]

Under $\I_D(A\ot_KA)=\I_D(L\ot_KL)$ in \eqref{GALeq0a}, identified with $\cI$ in $\I_D(A\ot_KA)$ is
$\cJ:=L\cdot \cI\cdot L$ in $\I_D(L\ot_KL)$. Under $\I_D(A\ot_KA\ot_KA)=\I_D(L\ot_KL\ot_KL)$ in \eqref{GALeq0b}, 
\[
\cI\ot_KA\quad\text{and}\quad A\ot_K\cI\quad \text{in}\quad \I_D(A\ot_KA\ot_KA)
\]
are identified with
\[
\cJ\ot_KL\quad\text{and}\quad L\ot_K\cJ\quad \text{in}\quad \I_D(L\ot_KL\ot_KL),
\]
respectively, since the two on the same side are identified with the same
\[
(L\cdot \cI)\ot_KL\quad\text{and}\quad L\ot_K(\cI\cdot L) \quad \text{in}\quad \I_D(L\ot_KA\ot_KL),
\]
respectively, under the identifications in \eqref{GALeq0b}. 
From this we see that $\cI$ is a $D$-coideal of $A\ot_KA$ if and only if $\cJ$ is a $D$-coideal of $L\ot_KL$.
This shows
\[
\I_D^{\op{co}}(A\ot_KA)=\I_D^{\op{co}}(L\ot_KL),
\]
completing the proof.
\epf

\begin{rem}\label{GALrem1}
Here are two remarks.
\begin{itemize}
\item[(1)]
Suppose that 
$\a$ in $\I^{\tu{Hopf}}(H)$,
$\cI$ in $\I^{\tu{co}}_D(A\ot_K A)$ 
and $\cJ$ in $\I^{\tu{co}}_D(L\ot_K L)$ 
correspond with each other. As is seen from
from the preceding proof the correspondences are characterized by the isomorphisms \eqref{GALeq0c}--\eqref{GALeq0d} below.
\begin{itemize}
\item[(i)]
The composite $A\ot_KA\overset{\simeq}{\longrightarrow} A\otk H\to A\otk(H/\a)$
of the isomorphism $\Theta_A$ with the natural projection $A\otk H \to A\otk(H/\a)$ induces an algebra
isomorphism
\eq\label{GALeq0c}
(A\ot_KA)/\cI\overset{\simeq}{\longrightarrow} A\otk(H/\a)
\eeq
in ${}_D\M$.
\item[(ii)]
The isomorphism above induces by localization an algebra isomorphism
\eq\label{GALeq0d}
(L\ot_KL)/\cJ\overset{\simeq}{\longrightarrow} T^{-1}(L\otk(H/\a))
\eeq
in ${}_D\M$. 
Here the target is the localization with respect to the multiplicative subset
\eq\label{GALeq1}
T:=\{\,  \theta'(a) : a\ \text{is an}\ A\text{-regular element of}\ A_0\, \}
\eeq
of $A\otk (H/\a)$, where $\theta' : A \to A\otk(H/\a)$ denotes the composite of 
$\theta_{\! A} : A \to A\otk H$
with the natural projection $A \otk H\to A\otk(H/\a)$. 
\end{itemize}
\item[(2)]
In what follows we will use only the correspondence between $\I^{\tu{Hopf}}(H)$ and $\I^{\tu{co}}_D(L\ot_K L)$. 
Indeed, Takeuchi \cite[Corollary 2.1]{T} does not refer to the set corresponding 
to our $\I^{\tu{co}}_D(A\ot_K A)$; 
one reason should be
that the necessary $D$-simplicity of the principal algebra 
(see \cite[Theorem 2.11]{T})
is not yet proven at the stage. 
\tc{We involved 
$\I^{\tu{co}}_D(A\ot_K A)$, in order to thus make}
it clear how the isomorphism \eqref{GALeq0d}, which will play a role in the proofs of Proposition \ref{GALprop2} and Theorem \ref{GALthm1} below, arises. 
\end{itemize}
\end{rem}

The following is the general fact which was used in the second last paragraph of the last proof.  

\begin{lemma}\label{GALlem1}
A super-bi-ideal of a Hopf superalgebra is necessarily stable under the antipode,
or in other words, it is a Hopf super-ideal. 
\end{lemma}
\pf
We suppose that $J$ is a (super-commutative) Hopf superalgebra, in general, and
$\b$ is a super-bi-ideal of $J$. We may suppose that $J$ is finitely generated,
since $J=\bigcup_{\alpha}J_{\alpha}$ is a filtered union of finitely generated Hopf sub-superalgebras $J_{\alpha}$, and we have
only to show that in each $J_{\alpha}$, the super-bi-ideal $\b \cap J_{\alpha}$ is stable under the antipode. 
Let $Q=J/\b$  be the quotient super-bialgebra of $J$ by $\b$. To see that $Q$ is a Hopf superalgebra we have to prove that 
the left $Q$-superlinear morphism 
\[
Q\ot_k Q \to Q\ot_k Q, \quad p\ot q
\mapsto pq_{(1)}\ot q_{(2)},
\]
where $\Delta(q)=q_{(1)}\ot q_{(2)}$ denotes the coproduct of $Q$, 
is bijective; notice that the condition is translated in terms of the affine supermonoid scheme $\mathsf{M}=\op{Sp}_k(Q)$, 
so that the morphism $\mathsf{M}\times \mathsf{M}\to \mathsf{M}\times \mathsf{M}$ given by $(x,y)\mapsto (x,xy)$ is
isomorphic. Since $Q$ is supposed to be finitely generated, $I_Q\, (=(Q_1))$ is nilpotent. Therefore, 
we have only to prove that the morphism 
$\overline{Q}\otk Q\to \overline{Q}\otk Q$ considered modulo $I_Q$ is bijective. In fact, it has 
\[
\overline{p}\ot q \mapsto \overline{p}\cS(\overline{q}_{(1)})\ot q_{(2)}
\]
as its inverse, where $\overline{p}$ denotes $p\, \op{mod}I_Q$. Here we have let
$\cS$ denote the antipode of $\overline{Q}$, which exists 
since as was proved by Nichols \cite[Theorem 1 (iv)]{N}, any commutative 
quotient bialgebra of a (not necessarily commutative) Hopf algebra over a field is necessarily a Hopf algebra. 
\epf

Set $\G:=\G(L|K)\, (=\op{Sp}_k(H))$, the Galois supergroup of $L|K$. Let
\[
\X:=\op{Sp}_k(A)
\]
denote the affine $k$-superscheme represented by the principal $D$-superalgebra $A$. Note that $\G$ 
acts on $\X$ on the right with the structure morphism $\theta_{\! A} : A \to A\otk H$. The fact that $A|K$ is $H$-Galois
(see Theorem \ref{DEFthm1} (4)) is translated in geometric terms so that the $\G$-action on $\X$ is free, and 
the fppf $k$-sheaf $\X\tilde{/}\G$
of $\G$-orbits in $\X$ is an affine $k$-superscheme, and is represented by $K$. 
Let $\mathsf{F}$ be a closed sub-supergroup scheme of $\G$. Then $\F$ acts freely on $\X$ by the restricted action, so that
the fppf $k$-sheaf $\X\tilde{/}\F$ is constructed. This is investigated by the following. 

\begin{prop}\label{GALprop2}
Retain the notation as above.
\begin{itemize}
\item[(1)] 
$\X\tilde{/}\mathsf{F}$ is a Noetherian and smooth, integral superscheme over $k$.
\item[(2)]
The function SUSY field $k(\X\tilde{/}\mathsf{F})$ of $\X\tilde{/}\mathsf{F}$ (see Definition \ref{SSCdef1}) 
is an intermediate $D$-SUSY field of $L|K$,
such that $L|k(\X\tilde{/}\mathsf{F})$ is admissible as an extension of SUSY fields. 
\end{itemize}
\end{prop}
\pf
Let us set
$\Y:=\X\tilde{/}\mathsf{F}$.

\medskip

\emph{Step 1}.\
We aim to conclude from \cite[Theorem 1.8]{MOT} (for $\G$ and $\X$ applied to the present $\F$ and $\X$) that 
$\Y$ is a Noetherian smooth superscheme over $k$. 

We have to verify the two assumptions (a)--(b) 
of the cited theorem; the Noetherianity, in particular, then follows from the proof of the theorem, in view of the
present situation that  $\X$ is
Noetherian, more strongly than being locally Noetherian as assumed by the theorem. 
It is trivial to verify the assumption (a): 
$\F$ and $\X$ are smooth over $k$. As for $\X$, it is smooth 
over $K$ by Proposition \ref{HGEprop1}, and
is so over $k$ since $K$ is smooth over $k$. 
It remains to verify the assumption (b): 
the fppf $k$-sheaf $\X_{\op{ev}}\tilde{/}\F_{\op{ev}}$ with respect the induced (necessarily, free)
$\F_{\op{ev}}$-action on $\X_{\op{ev}}$ is a $k$-scheme. 
Recall $\X_{\op{ev}}=\op{Sp}_k(\overline{A})$, and that 
$\overline{A}|\K$ is 
$\overline{H}$-Galois; see Theorem \ref{DEFthm1} (5). 
One then sees that $\X_{\op{ev}}$ is a $\K$-torsor under the base change $(\G_{\op{ev}})_{\K}$ of the
affine algebraic $k$-group scheme $\G_{\op{ev}}=\op{Sp}_k(\overline{H})$, or in other words,
the $(\G_{\op{ev}})_{\K}$-equivariant $\K$-schemes $(\G_{\op{ev}})_{\K}$ and $\X_{\op{ev}}$ turn isomorphic after the base change
to some extension field of $\K$ of finite degree. 
Since $\G_{\op{ev}}\tilde{/}\F_{\op{ev}}$ is known to be a $k$-scheme, 
the fppf $\K$-sheaf $\X_{\op{ev}}\tilde{/}\F_{\op{ev}}$ turns into a scheme after base change 
to some finite-degree field extension of $\K$. 
An argument of descent, which is found, for example, in Step 3 (Conclusion) of the proof of \cite[Theorem B.37, pp.605--606]{Milne},
proves that $\X_{\op{ev}}\tilde{/}\F_{\op{ev}}$ is a scheme over $\K$, and hence over $k$. 
We thus have obtained the conclusion we aim at. 

\medskip

\emph{Step 2}.\
We aim to prove: \emph{the Noetherian smooth superscheme $\Y$ is integral, and the function SUSY field $k(\Y)$
is included in $L$, so that $L|k(\Y)$ is an admissible extension of SUSY fields.}

In this paragraph we adapt to the present situation, 
some of the argument of  proving the result \cite[Theorem 1.8]{MOT} which was applied in the last step.
Let $\pi : \X\to \Y$
be the canonical morphism of $k$-superschemes; it is 
affine, faithfully flat and finitely presented; see \cite[Proposition 1.6 (2)]{MOT}.
Recall that the underlying topological spaces of $\Y$ and of 
$\Y_{\op{ev}}$ ($=\X_{\op{ev}}\tilde{/}\F_{\op{ev}}$, see Remark (ii) above) are
the same, and notice from Zubkov's Theorem \cite[Theorem 3.1]{Z} that
an open set is affine, regarded in $\Y$, if and only 
if it is so, regarded in $\Y_{\op{ev}}$. 
As an important consequence, 
$\Y$ is covered by finitely many affine open sets, since $\Y_{\op{ev}}$ is seen to be so. 
Choose arbitrarily a non-empty affine open set $U$ of $\Y$. Then $\pi^{-1}(U)$ 
is an affine open set of $\X$, which is stable under the action by $\mathsf{F}$ so that $\pi^{-1}(U)\tilde{/}\F=U$. 
Set
\[
\cA :=
\pi_*
\cO_{\X}(U)\, 
(=\cO_{\X}(\pi^{-1}(U))),
\quad \cB:=\cO_{\Y}(U). 
\]
Then these are Noetherian smooth $k$-superalgebras. 
Suppose $\mathsf{F}=\op{Sp}_k(H/\mathfrak{a})$, where $\mathfrak{a}$
is a Hopf super-ideal of $H$. 
Let $\theta_{\! \cA} : \cA \to \cA \, \otk (H/\mathfrak{a})$ denote the structure morphism of the
action $\pi^{-1}(U)\times \mathsf{F}\to \pi^{-1}(U)$. Then $(\cA, \theta_{\! \cA})$ is an algebra in $\M^{H/\mathfrak{a}}$, 
and in fact, $\cA|\cB$ is an $H/\mathfrak{a}$-Galois extension. We thus have the isomorphism
\begin{equation*}\label{GALeq2}
\Theta_{\! \cA}:\cA \ot_{\cB}\cA \overset{\simeq}{\longrightarrow}\cA \otk (H/\mathfrak{a}),\ 
\Theta_{\cA}(a \ot b)=a\, \theta_{\! \cA}(b).
\end{equation*}
The isomorphism \eqref{HGEeq5}, applied to the $H/\a$-Galois extension $\cA|\cB$,
gives a $\cB$-superalgebra isomorphiism
\eq\label{GALeq2a}
\cA\simeq \cB\ot_{\overline{\cB}} \overline{\cA}\otk \wedge_k(\mathfrak{w}_{H/\a}),
\eeq
in which $\cB$ is regarded as a $\overline{\cB}$-superalgebra through an arbitrarily chosen section $\overline{\cB}\to \cB$
of the projection 
$\cB\to \overline{\cB}$. 
In addition, $\overline{\cB}$ is naturally included in $\overline{\cA}$, so as
$\overline{\cB}=\overline{\cA}^{\op{co}(\overline{H/\a})}$; see Eq.\! (3.14) 
in \cite[Theorem 3.15]{MOT}. 
Since $\overline{\cA}$ is an integral domain,
$\overline{\cB}$ is as well. This shows that $\Y$ is integral. 

Let us set
$M:=k(\Y)$. 
As in Example \ref{SUFex1} we see
\begin{equation}\label{GALeq2b}
\begin{split}
&L=\Quot(A)=\Quot(\cA),\quad \L= \Quot(\overline{\cA}), \\
&M=\Quot(\cB),\quad \overline{M}=\Quot(\overline{\cB}).
\end{split}
\end{equation}
In view of \eqref{GALeq2a}
we may regard so as
\begin{equation}\label{GALeq3}
\Quot(\cB)\subset \Quot(\cB)\ot_{\Quot(\overline{\cB})}\Quot(\overline{\cA})\otk \wedge_k(\mathfrak{w}_{H/\a})=\Quot(\cA),
\end{equation}
or in other words, so as
\begin{equation}\label{GALeq3a}
M\subset M\ot_{\overline{M}}\overline{L}\otk \wedge_k(\mathfrak{w}_{H/\a})=L.
\end{equation}
Here one should notice that the equality holds since  
$\overline{\cA}$ and $\Quot(\cB)$ (or $\Quot(\overline{\cA})$ and $\cB$)
are $\overline{\cB}$-flat.
We conclude that $M$ is included in $L$, and $L|M$ is admissible. 

\medskip

\emph{Step 3}.\
To complete the proof we aim to prove: \emph{$M\, (=k(\Y))$ is $D$-stable in $L$, and includes $K$.}

Regard $\cA\otk (H/\a)$ as a superalgebra over $\cA\ot_{\cB}\cA$ through
$\Theta_{\! \cA}$, and localize it by 
the multiplicative subset 
\[
\cS:=\{\, s\ot t \in \cA\ot_{\cB}\cA : s\ \text{and}\ t\ \text{are}\ \cA\text{-regular elements of}\ \cA_0\, \}.
\]
of $\cA\ot_{\cB}\cA$. 
Notice from \eqref{GALeq3} that
\[
\cS^{-1}(\cA\ot_{\cB}\cA)=\Quot(\cA)\ot_{\cB}\Quot(\cA)=L\ot_M L. 
\]
To localize $\cA\otk (H/\a)$, define a multiplicative subset $\cT$ of $L\otk(H/\a)$ by 
\[
\cT:=\{\, \theta_{\! \cA}(t) : t\ \text{is an}\ \cA\text{-regular element of}\ \cA_0\, \}.
\]
By \eqref{GALeq3} the isomorphism $\Theta_{\! \cA}$ turns by the localization into the bottom isomorphism
in the diagram:
\begin{equation}\label{GALeq4}
\begin{xy}
(0,8)   *++{L\ot_KL}  ="1",
(36,8)  *++{T^{-1}(L\otk(H/\mathfrak{a}))} ="2",
(0,-8)  *++{L\ot_ML} ="3",
(36,-8) *++{\mathcal{T}^{-1}(L\otk(H/\mathfrak{a})).}="4",
{"1" \SelectTips{cm}{} \ar @{->} "2"},
{"3" \SelectTips{cm}{} \ar@{->}^{\hspace{-8mm}\simeq} "4"},
{"1" \SelectTips{cm}{} \ar@{->} "3"},
{"2" \SelectTips{cm}{} \ar @{->} "4"}
\end{xy}
\end{equation}
Here and in what follows the unlabeled $L\ot_K L\to L\ot_M L$ indicates the natural 
surjection $a\ot_Kb\mapsto a\ot_Mb$, which is a \tc{($D$-superlinear)} $L$-coring morphism.  
The top morphism
is the composite of the projection $L\ot_KL\to (L\ot_K L)/\cJ$ with the isomorphism \eqref{GALeq0d}. 
Here recall that $\cJ$
is the $D$-coideal of $L\ot_KL$ which corresponds to $\a$. 
The restriction morphism $A =\cO_{\X}(\X) \to \cO_{\X}(\pi^{-1}(U))=\cA$ for the structure sheaf $\cO_{\X}$ of $\X$ is compatible with the coaction by $H/\a$, and
it turns by localization into the identity on $L$.
Therefore, the vertical morphism
on the RHS is induced, and it is indeed an isomorphism. 
The diagram above, which is seen to be commutative, shows 
$\cJ=\op{Ker}(L\ot_K L\twoheadrightarrow L\ot_M L)$. This proves
\eq\label{GALeq4a}
M=\{\, x \in L: x\ot 1\equiv 1\ot x\op{mod} \cJ\ \text{in}\ L\ot_KL\, \},
\eeq
since by Proposition \ref{SUFprop2} (1), $M$ is a direct summand of $L$ in ${}_M\M$.
The result implies that $M$ is $D$-stable in $L$. 

The composite $K\hookrightarrow A \to \cA$ of the inclusion $K=A^{\op{co}H}\hookrightarrow A$ with
the restriction morphism $A \to \cA$ (in $\M^{H/\a}$) maps into $\cB=\cA^{\op{co}(H/\a)}$. 
Therefore, the inclusion
$K\hookrightarrow L=\Quot(\cA)$ maps into $M=\Quot(\cB)$.  
\epf

\begin{theorem}[Galois correspondence, I]\label{GALthm1}\
The assignment\
$\mathsf{F}~\mapsto~k(\X\tilde{/}\mathsf{F})$\ 
which is obtained by the preceding proposition gives a bijection from 
\begin{itemize}
\item[$\bullet$] the set of all closed sub-supergroup schemes of $\G(L|K)$
\end{itemize}
onto
\begin{itemize}
\item[$\bullet$] the set of all intermediate $D$-SUSY fields of $L|K$.
\end{itemize}
\end{theorem}
\pf
Let $M$ be an intermediate $D$-SUSY field $M$ of $L|K$. We aim to prove: 
\emph{there
uniquely exists a Hopf super-ideal $\mathfrak{a}$ such that $M$ is assigned 
to the closed sub-supergroup scheme $\op{Sp}_k(H/\mathfrak{a})$
of $\G(L|K)$.}
To find such a Hopf super-ideal, let $\a$ be the Hopf super-ideal which corresponds, under the correspondence
of Proposition \ref{GALprop1}, to the
$D$-coideal 
\[
\cJ:=\op{Ker}(L\ot_KL\to L\ot_ML)
\]
of $L\ot_KL$. 
For this $\mathfrak{a}$, set $\mathsf{F}:=\op{Sp}_k(H/\mathfrak{a})$,
$Y:=\X\tilde{/}\mathsf{F}$, 
and argue as in the preceding proof. Then the commutative diagram, which modifies \eqref{GALeq4} so that the top arrow 
is replaced by the isomorphism \eqref{GALeq0d}
shows that the natural surjection $L\ot_KL\to L\ot_{k(\Y)}L$ induces an isomorphism 
\[
(L\ot_KL)/\cJ\overset{\simeq}{\longrightarrow} L\ot_{k(\Y)} L.
\]
(Notice that the $L\ot_ML$ in  \eqref{GALeq4} should now read $L\ot_{k(\Y)}L$.)
This proves $M=k(\Y)$. 
The argument shows uniqueness of desired $\a$, as well.
\epf

\begin{rem}\label{GALrem2}
As is seen from the proof above, 
the inverse of the bijection $\mathsf{F}\mapsto k(\X\tilde{/}\mathsf{F})$ is
given by $M \mapsto \op{Sp}_k(H/\mathfrak{a})$, where $\mathfrak{a}$ is the Hopf super-ideal of $H$ given by
\eq\label{GALeq4b}
\mathfrak{a}=H \cap \op{Ker}(L\ot_K L\to L\ot_ML);
\eeq
cf. \cite[Theorem 2.7]{T}.
Here we naturally regard so as $H \subset A\ot_KA\subset L\ot_KL$ by using the $K$-flatness of
$A$ and $L$; as for $L$, see Remark \ref{SUFrem3}
\end{rem}

\begin{corollary}\label{GALcor1}
Let $M$ be an intermediate $D$-SUSY field of $L|K$.
\begin{itemize}
\item[(1)]
$M$ is $D$-simple.
\item[(2)]
The extensions $L|M$, $M|K$ of the SUSY fields
are both admissible and of finite type.
\end{itemize}
\end{corollary}
\pf We prove first Part 2, and then Part 1.

(2)\
\emph{``Admissiblity''.}\
Since by Theorem \ref{GALthm1}, $M$ is of the form $k(\X\tilde{/}\F)$, it follows by
Proposition \ref{GALprop2} (2) that $L|M$ is admissible. 
One then sees that $M|K$ is admissible, or in other words, the associated graded-algebra morphism 
$\op{gr} K \to \op{gr} M$ is injective, since $\op{gr}  K\to \op{gr} L$ is so by Theorem \ref{DEFthm1} (1).

\emph{``Finite type''.}\
\tc{Notice that $\overline{M}$ is an intermediate field of the finitely generated field extension
$\L|\K$. One then sees that $\L|\overline{M}$ is obviously finitely generated, and $\overline{M}|\K$ is as well by Nagata's Theorem.}

(1)\
By the admissibility of $L|M$ just proven, $L$ is $M$-faithfully flat; see Remark \ref{SUFrem3}. 
This,
together with the $D$-simplicity of $L$, proves Part 1.
\epf

\begin{prop}\label{GALprop3}
Let $\mathsf{F}=\op{Sp}_k(H/\mathfrak{a})$ be a closed sub-supergroup scheme of $\mathsf{G}(L|K)$, which is given
by a Hopf super-ideal $\mathfrak{a}$ of $H$, as presented. Let $M=k(\X\tilde{/}\mathsf{F})$ be the
corresponding intermediate SUSY field of $L|K$. Then $(L|M,AM,H/\mathfrak{a})$ is a SUSY PV extension
with the Galois supergroup $\mathsf{F}$. 
\end{prop}
\pf
Recall Definition \ref{DEFdef1}, the definition of SUSY PV extensions. 
First, notice that $L|M$ satisfies (LK4--6). 
Next, we wish to prove that
$L|M$ is a SUSY PV extension with the principal $D$-superalgebra $AM$. 
One sees from Lemma \ref{SUFlem2}
that $\Quot(A)\subset \Quot(AM) \subset L=\Quot(A)$, whence $\Quot(AM)=L$. 
We have the natural commutative diagram 
\[
\begin{xy}
(0,0)   *++{A\ot_k H}  ="1",
(43,0)  *++{A\ot_KA} ="2",
(0,-16)  *++{AM\otk(AM\ot_M AM)^D} ="3",
(43,-16) *++{AM\ot_MAM,}="4",
{"1" \SelectTips{cm}{} \ar @{->}_{\mu}^{\simeq} "2"},
{"3" \SelectTips{cm}{} \ar @{->} "4"},
{"1" \SelectTips{cm}{} \ar@{->} "3"},
{"2" \SelectTips{cm}{} \ar @{->>} "4"}
\end{xy}
\]
where $H$ maps into $(AM\ot_MAM)^D$ through the vertical morphism on the LHS.
The surjectivity of the vertical natural morphism on the RHS ensures the condition (A1) for $AM$, and proves
the desired result. We remark that $AM$ is $M$-flat, in particular. 

Finally, we wish to show that the Hopf superalgebra of $L|M$ is $H/\mathfrak{a}$. 
As is seen from \eqref{GALeq4b}, 
the base extension $L\ot_A$ of $\mu: A\ot_K A \overset{\simeq}{\longrightarrow} A \otk H$
\[
L\otk H \overset{\simeq}{\longrightarrow}L\ot_K A \twoheadrightarrow L\ot_M AM,
\]
thus composed with the last natural surjection,
induces an isomorphism $L\otk (H/\mathfrak{a})\simeq L\ot_M AM$, which restricts to 
\[
H/\mathfrak{a}\simeq (L\ot_M AM)^D \cap AM\ot_M AM=(AM\ot_M AM)^D.
\]
This proves the desired result. 
\epf

\begin{theorem}[Galois correspondence, II]\label{GALthm2}
The bijection obtained by Theorem \ref{GALthm1} restricts to a bijection from 
\begin{itemize}
\item[$\bullet$] the set of all closed normal sub-supergroup schemes of $\G(L|K)$
\end{itemize}
onto
\begin{itemize}
\item[$\bullet$] the set of those intermediate $D$-SUSY fields $M$ of $L|K$ such that $M|K$ is a SUSY PV extension.
\end{itemize}
\end{theorem}
\pf
Suppose $\mathsf{F}=\op{Sp}_k(H/\mathfrak{a})\mapsto M$ under the bijection, or namely, $M=k(\X\tilde{/}\mathsf{F})$.
We wish to prove that $\F$ is normal if and only if $M|K$ is SUSY PV.

\emph{``Only if''}.\quad 
Assume that $\F$ is normal in $\G(L|K)$. Regard the algebras $H$ and $A$ in $\M^H$ as
algebras in $\M^{H/\a}$ with respect to the $(H/\a)$-coaction induced by the projection $H \to H/\a$, 
and set 
\[
J:=H^{\op{co}(H/\a)},\quad B:=A^{\op{co}(H/\a)}.
\]
By \cite[Lemma 5.8]{M}, $J$ is 
a Hopf sub-superalgebra of $H$. Moreover, $M \mapsto M^{\op{co}(H/\a)}$ gives rise to 
a weak monoidal functor
\[
(\ )^{\op{co}(H/\a)}: \M^H\to \M^J.
\]
Therefore, $B$ is a $K$-algebra in $\M^J$, \tc{which is obviously $D$-stable.}
Our aim is to show that 
$(M|K, B, J)$ is SUSY PV. By \cite[Theorem 7.1]{MZ1}, $A$ is injective in $\M^H$. 
Since $\F$ is normal, $H$ is injective in $\M^{H/\a}$ by \cite[Theorem 5.9 (2)]{M}, whence $A$ is injective in $\M^{H/\a}$. 
Again by \cite[Theorem 7.1]{MZ1}, 
$A|B$ is an $H/\a$-Galois extension, with which associated is now an algebra isomorphism in $\M^{H/\a}$,
\eq\label{GALeq5}
A\ot_B A\overset{\simeq}{\longrightarrow} A\otk (H/\a). 
\eeq
Identifying so as
\begin{equation}\label{GALeq6}
A=K\ot_{\K}\overline{A}\, \square_{\overline{H}}H=K\ot_{\K}\overline{A}\ot_k\wedge_k(\mathfrak{w}_H)
\end{equation}
as in \eqref{HGEeq4}--\eqref{HGEeq5}, we have
\begin{equation}\label{GALeq7}
B=K\ot_{\K}\overline{A}\, \square_{\overline{H}}J=K\ot_{\K}\overline{A}^{\op{co}(\overline{H/\a})}\ot_k\wedge_k(\mathfrak{w}_J). 
\end{equation}
Here, the last equality holds by the following two reasons:
(i)~an isomorphism such $H \simeq \overline{H}\otk \wedge_k(\mathfrak{w}_H)$ as in \eqref{HGEeq2} and an analogous one, $J\simeq \overline{J}\ot_k\wedge_k(\mathfrak{w}_J)$,
can be chosen so that they are compatible with the inclusions $J\hookrightarrow H$, $\overline{J}\hookrightarrow \overline{H}$
and the natural morphism $\wedge_k(\mathfrak{w}_J)\to \wedge_k(\mathfrak{w}_H)$ (see \cite[Remark 4.8]{M}), and 
(ii)~one has the isomorphism 
\[
\overline{A}^{\op{co}(\overline{H/\a})}\simeq \overline{A}\, \square_{\overline{H}}\overline{J}
\]
as a restriction of the natural isomorphism $\overline{A}\simeq \overline{A}\, \square_{\overline{H}}\overline{H}$.
We see from \eqref{GALeq6} and \eqref{GALeq7} the following: 
(I)~$\Quot(B)$ is an intermediate SUSY field of $(\Quot(A)=)\, L|K$, \tc{which 
$D$-stable by Proposition \ref{EXDprop1} (1), and is $D$-simple by Corollary \ref{GALcor1} (1)}; 
(II)~the isomorphism \eqref{GALeq5} induces, by localization, an isomorphism 
\eq\label{GALeq7a}
L\ot_{\Quot(B)}L\simeq T^{-1}(L\otk (H/\a)),
\eeq
where $T$ is defined by \eqref{GALeq1}. 
It follows from \eqref{GALeq4a} that 
\eq\label{GALeq8}
M=\Quot(B).
\eeq

The isomorphism $\Theta_{\! A} : A\ot_KA\overset{\simeq}{\longrightarrow} A\otk H$ such as in \eqref{HGEeq3}, 
with $(\ )^{\op{co}(H/\a)}$ applied, restricts to $A\ot_KB\simeq A\otk J$,
which in turn restricts to $B\ot_KB\simeq B\otk J$, as is seen by using the $B$-faithful flatness of $A$.  
It follows that $J=(B\ot_KB)^D$ and $B\cdot J=B\ot_KB$. 

One sees easily $M^D=k$, and $M|K$ is of finite type by Corollary \ref{GALcor1} (2).
Combining all together, we conclude that $(M|K,B,J)$ is SUSY PV.

\emph{``If''}.\quad Suppose that $(M|K,B,J)$ is SUSY PV. 
Since one sees that $AB$ satisfies (A1--2), the uniquess of $A$ shows $B\subset A$; see the poof of 
Proposition \ref{GALprop1} (1).
Since $A$ and $B$ are $K$-flat, we see $B\ot_KB\subset A\ot_KA$, and that $J$ is a Hopf sub-superalgebras of $H$. Let $\a':=J^+H$,
where $J^+=\op{Ker}(\varepsilon :J\to k)$ denotes the augmentation super-ideal of $J$. Set $\F':=\op{Sp}_k(H/\a')$. Then $\F'$
is a normal closed sub-supergroup scheme of $\G(L|K)$ such that $\G(L|K)\tilde{/}\F'=\op{Sp}_k(J)$. 

Our aim is to prove $\F'=\F$, or equivalently, $(\Quot(B)=)\, M=k(\X\tilde{/}\F')$. 
We have $J=H^{\op{co}(H/\a')}$. Compare $\Theta_A: A\ot_K A \overset{\simeq}{\longrightarrow} A \otk H$ with the base extension $A\ot_B$ of
$\Theta_B: B\ot_K B \overset{\simeq}{\longrightarrow} B \otk J$. Then one sees that $A\ot_KB=(A\ot_K A)^{\op{co}(H/\a')}$; this implies
$B=A^{\op{co}(H/\a')}$ since $A$ is $K$-faithfully flat. 
The same argument as in the \tc{first} paragraph of the ``only if" part above shows that $A|B$ is $H/\a'$-Galois. 
Associated is an isomorphism $A\ot_BA\overset{\simeq}{\longrightarrow} A\otk (H/\a')$ analogous to \eqref{GALeq5}, which
induces an analogous one to \eqref{GALeq7a}. We conclude
$M=k(\X\tilde{/}\F')$, just as we did \eqref{GALeq8} . 
\epf


\section{Characterizations of SUSY PV extensions}\label{secCHA} 

The title above refers to the main result of this section, Theorem \ref{CHAthm1}, 
on which and after we assume that $D$ satisfies (D1--2).
But before that, $D$ may be a super-cocommutative Hopf superalgebra in general.

\subsection{The bozonized Hopf algebra $\Db$}\label{BSH}
Present $\mathbb{Z}/(2)$ as a multiplicative group so that $\mathbb{Z}/(2)=\{ \sigma^i : i =0,1\}$ 
with the generator $\sigma$. We let $\Db$ denote the bozonization of $D$; it is alternatively denoted by $\widehat{D}$
in \cite{M}. 
Recall that $\Db$ is an ordinary 
Hopf algebra which is not cocommutative unless $D$ is purely even; it is, as an algebra, 
the semi-direct product $\Db =D\rtimes \mathbb{Z}/(2)$ with respect to the action
\[
\sigma \triangleright d=(-1)^{|d|} d, \quad d \in D,
\]
and contains $\sigma$ as a grouplike element. The counit $\eb: \Db \to R$ extends the one
$\varepsilon$ of $D$ so that $\eb (d)=\varepsilon(d)$ for $d \in D$, while 
the coproduct $\Delb : \Db \to \Db \ot \Db$ is given by
\[
\Delb (d)= d_{(1)} \sigma^{|d_{(2)}|} \ot d_{(2)},\quad d\in D.
\]
Notice $\eb(\sigma)=1$ and $\Delb(\sigma)=\sigma \ot \sigma$, since $\sigma$ is grouplike.
To distinguish the $\Delb$ above from
the coproduct $\Delta(d)=d_{(1)}\ot d_{(2)}$ of $D$, let us write so as
\eq\label{BSHeq1}
\Delb (d)=d_{[1]}\ot d_{[2]},\quad d \in \Db.
\eeq
Notice that $D$ is a left coideal of $\Db$, that is, 
\eq\label{BSHeq2}
\Delb (D) \subset \Db \ot D.
\eeq
The antipode $\Sb : \Db \to \Db$ is given by
\[
\Sb(d)=\sigma^{|d|}\mathcal{S}(d)\, (=(-1)^{|d|}\mathcal{S}(d)\sigma^{|d|}),\quad d \in D,
\]
where $\mathcal{S}$ is the antipode of $D$. Notice $\Sb(\sigma)=\sigma$. 
It is easy to see that $\Sb$ is a bijection, and its inverse, which we denote by
$\Sbi : \Db \to \Db$, is given by
\[
\Sbi (d)=\mathcal{S}(d)\sigma^{|d|},\quad d \in D.
\]
Given a super-vector space $V$, the $D$-supermodule structures on $V$ are in one-to-one
correspondence with 
\tc{the left $\Db$-module structures on  $V$ such that the action by $\sigma$ affords the given party of 
$V$ in the sense that} 
\begin{equation*}
\sigma \triangleright v=(-1)^{|v|} v,\quad v \in V.
\end{equation*}
\tc{To each $D$-supermodule structure, corresponding is the unique extension satisfying the formula above.}
We thus have the identification
\eq\label{BSHeq3}
{}_D\M={}_{\Db}\mathtt{Mod}, 
\eeq
where ${}_{\Db}\mathtt{Mod}$ denotes the category of left $\Db$-modules. In fact, this is an identification 
of the monoidal categories; they both have the tensor product $\ot\, (=\ot_R)$, and the unit object $R$
which is regarded as a trivial (super-) module through the counit. 

Let $K$ be a $D$-superalgebra. Since 
this may be regarded as a \emph{$\Db$-algebra}, that is, an algebra in ${}_{\Db}\mathtt{Mod}$, we can construct
the algebra $K \# \Db$ of smash product \cite[p.155]{Sw}, which is not necessarily commutative. It includes
\[
K\# D:=K\ot D\, (\subset K\ot \Db=K\# \Db)
\]
as a subalgebra. This is naturally as a (not necessarily super-commutative) superalgebra whose bosonization
$(K\# D)^{\mathrm{b}}:=(K\# D)\rtimes \mathbb{Z}/(2)$ coincides with $K\# \Db$. Therefore, the category ${}_{K\# D}\M$ 
of left $K\# D$-supermodules is identified with the category ${}_{K\# \Db}\mathtt{Mod}$ of left $K\# \Db$-modules.

Notice that the identification \eqref{SASeq00} of symmetric monoidal categories is generalized to 
\begin{equation*}\label{BSHeq4}
({}_K({}_D\M), \ot_K, K)= (({}_D\M)_K, \ot_K, K).
\end{equation*}
These are further identified with ${}_{K\# D}\M$ and with ${}_{K\# \Db}\mathtt{Mod}$, as categories. 

Suppose that $V$ is an object in ${}_K({}_D\M)\, (=({}_D\M)_K)$, and 
$L|K$ is an extension of $D$-superalgebras.
Let
\begin{equation}\label{BSHeq5}
\op{Hom}_{K-}(V,L) \quad \tu{(resp.,} \op{Hom}_{-K}(V,L)\tu{)}
\end{equation}
denote the set which consists of all left (resp., right, with $V$ supposed to be in
$({}_D\M)_K$) $K$-linear morphisms $V \to L$ that may not preserve the parity.
In fact, this is a right (resp., left) $L$-supermodule; the parity is inherited from 
$\op{Hom}(V, L)$ (see \eqref{SASeq0}), and $L$ acts so that
\[
(f a)(v):=f(v)a\quad (\tu{resp.,}~~(a f)(v):=af(v)),
\]
where $a\in L$, $f$ is in the set \eqref{BSHeq5}, and $v\in V$. 

\begin{prop}\label{BSHprop1}
We have the following.
\begin{itemize}
\item[(1)] $\op{Hom}_{K-}(V, L)$ turns into an object in $({}_D\M)_L$ with respect to the $D$-action
defined by
\eq\label{BSHeq5a0}
df(v):= d_{[2]}f(\Sbi (d_{[1]})v),
\eeq
where $d \in D$, $f \in \op{Hom}_{K-}(V, L)$ and $v \in V$. 
\item[(2)] $\op{Hom}_{-K}(V,L)$ turns into an object in $({}_D\M)_L$ with respect to the $D$-action
defined by
\eq\label{BSHeq5a}
df(v):= d_{[1]}f(\Sb (d_{[2]})v),
\eeq
where $d \in D$, $f \in \op{Hom}_{-K}(V, L)$ and $v \in V$. 
\item[(3)]
Given a homogeneous, left (resp., right) $K$-linear morphism $f : V \to L$, 
\[
f\circ \sigma^{|f|}: V \to L,\quad v \mapsto f(\sigma^{|f|}v)=(-1)^{|f||v|}f(v)
\]
is a right (resp., left) $K$-linear morphism. Moreover, $f \mapsto f\circ \sigma^{|f|}$ gives 
$R$-superlinear isomorphisms
\[
\op{Hom}_{K-}(V,L)\rightleftarrows\op{Hom}_{-K}(V,L),
\]
which are mutual inverses. 
\end{itemize}
\end{prop}
\pf This is straightforward to see. We only remark: in order to see the $K$-linearity of $df$ in Parts 1-2, 
one should use
\eq\label{BSHeq6}
d_{[2]}\Sbi(d_{[1]})=\varepsilon(d)1=
d_{[1]}\Sb(d_{[2]}),\quad d\in \Db,
\eeq
as well as
\eq\label{BSHeq6a}
\Delb(\Sb(d))=\Sb(d_{[2]})\ot\Sb(d_{[1]}),\quad d \in \Db
\eeq
and the analogous one for $\Sbi$.
\epf

We let 
\[
\op{Hom}_{K\# D}(V, L)
\]
denote the set of all left $K\# D$-linear (or equivalently, left $K$-linear and $D$-linear) 
morphisms $V\to L$ that may not preserve the parity. 
Notice that this is $L^D$-stable in the right $L$-module $\op{Hom}_{K-}(V,L)$ in ${}_D\M$, whence it is
a right $L^D$-supermodule.

\begin{prop}\label{BSHprop2}
We have
\eq\label{BSHeq7}
\op{Hom}_{K-}(V,L)^D=\op{Hom}_{K\# D}(V, L).
\eeq
\end{prop}
\pf
This follows easily by using the first equality in \eqref{BSHeq6} and the analogous
\[
\Sbi(d_{[2]})d_{[1]}=\varepsilon(d)1,\quad d\in \Db.
\]
\epf

\subsection{Splitting $D$-SUSY fields}\label{SPD}

Let $K$ be a $D$-SUSY field. 
Choose arbitrarily a $K$-finite free $K$-module $V$ in ${}_D\M$, and suppose that its rank is $m|n$; see Definition \tc{\ref{SASdef1}}. 

In what follows we let $L|K$ be an extension of $D$-SUSY fields, and assume that $L$ is $D$-simple, whence $L^D$ is
necessarily a field. 

\begin{lemma}\label{SPDlem1}
We have the following.
\begin{itemize}
\item[(1)]
The super-vector space $\op{Hom}_{K\# D}(V,L)$ over the field $L^D$ is of dimension
at most $m|n$, or in notation,
\eq\label{SPDeq0}
\dim_{L^D}\op{Hom}_{K\# D}(V,L)\le m|n.
\eeq
\item[(2)]
The right $L$-multiplication
\eq\label{SPDeq1}
\mu' : \op{Hom}_{K\# D}(V,L)\ot_{L^D} L \to \op{Hom}_{K-}(V, L),\quad \mu'(f\ot a)=fa
\eeq
defines a morphism in $({}_D\M)_L$, which is injective and splits $L$-superlinearly. 
\end{itemize}
\end{lemma}
\pf We prove first Part 2 and then Part 1.

(2)\ Consider the morphism $\mu$ given in \eqref{DSAeq1}, replacing the $V$ in ${}_A({}_D\M)$ there with
the present $\op{Hom}_{K-}(V,L)$ in $({}_D\M)_L$. 
In view of \eqref{BSHeq7} one then has the $\mu'$ above, which is a morphism in $({}_D\M)_L$.
It is injective by Proposition \ref{DSAprop1} (2). 
Moreover, it splits $L$-superlinearly by Remark \ref{SUFrem2} (1). 

(1)\ By the result just proven, the morphism $\mu'$ with $\ot_L\L$ applied
\[
\mu'\ot_L\L : \op{Hom}_{K\# D}(V,L)\ot_{L^D}\L \to \op{Hom}_{K-}(V,\L)
\]
is a (split) $\L$-superlinear injection. This proves Part 1 since one sees
\[\dim_{\L}\op{Hom}_{K-}(V,\L)=m|n. \]
\epf

\begin{prop}\label{SPDprop1}
For $L$ such as above the following are equivalent:
\begin{itemize}
\item[(a)] The base extension $L\ot_K V$, regarded as an $L$-module in ${}_D\M$ with respect to the $D$-action
\[ d(a\ot v)=d_{[1]}a\ot d_{[2]}v,\quad d\in D,\ a\ot v\in L\ot_K V,\]
is isomorphic to $L^{m|n}\, (=L[0]^m\oplus L[1]^n)$;
\item[(b)] $\dim_{L^D}\op{Hom}_{K\# D}(V,L)= m|n$;
\item[(c)] The morphism $\mu'$ in \eqref{SPDeq1} is an isomorphism.
\end{itemize}
\end{prop}
\pf
From the preceding lemma one sees (b)$\Leftrightarrow$(c). Let us set $k:=L^D$. 

(a)$\Rightarrow$(b).\ This follows since we see, assuming (a), that
\[
\begin{split}
\op{Hom}_{K\# D}(V,L)&=\op{Hom}_{L\# D}(L\ot_KV,L)\\
&\simeq \op{Hom}_{L\# D}(L^{m|n},L)=k^{m|n}.
\end{split}
\]

(c)$\Rightarrow$(a).\ Assume (c), or equivalently, (b). Then the $\mu'$ gives an isomorphism
\[
L^{m|n}\overset{\simeq}{\longrightarrow} \op{Hom}_{K-}(V,L)=\op{Hom}_{L-}(L\ot_KV, L)
\]
in $({}_D\M)_L$. 
Applying $\op{Hom}_{-L}(\ , L)$, we have an isomorphism ${}_L({}_D\M)$,
\[
\op{Hom}_{-L}(\op{Hom}_{L-}(L\ot_K V,L),L) \overset{\simeq}{\longrightarrow}\op{Hom}_{-L}(L^{m|n},L).
\]
Here the source and the target are supposed to have the $D$-actions given by the formula \eqref{BSHeq5a} in which
$f$ is now supposed to be $L$-linear.
Moreover, we can verify the fact that they are canonically isomorphic to $L\ot_KV$ and $L^{m|n}$, respectively, in ${}_L({}_D\M)$;
this shows (a). A point of verifying the fact especially for the target is to see, setting $M:=L\ot_KV$, that the canonical
bijection
\[
M\overset{\simeq}{\longrightarrow}\op{Hom}_{-L}(\op{Hom}_{L-}(M,L),L),\quad m\mapsto(f \mapsto f(m))
\]
is $D$-linear. This is verified so as
\[
\begin{split}
d_{[1]}\{(\Sb(d_{[2]})f)(m)\}&=d_{[1]}\{ \Sb(d_{[2]}) f (\Sbi(\Sb(d_{[3]}))m)\}\\
&=d_{[1]} \Sb(d_{[2]}) f (\Sbi(\Sb(d_{[3]})) m)\\
&=f(dm),
\end{split}
\]
where $d \in D,\ f\in \op{Hom}_{L-}(M,L)$ and $m\in M$. Here, for the first equality we have used
\eqref{BSHeq5a0} and \eqref{BSHeq6a}. 
\epf

\begin{definition}\label{SPDdef1}
We say that $L$ is a \emph{splitting $D$-SUSY field} for $V$, if the equivalent conditions (a)--(c) above
are satisfied. Such an $L$ is said to be \emph{minimal} if in addition, we have
\[
L=\Quot(K[V;L]),
\]
where we let 
\eq\label{SPDeq2}
K[V; L]=K[f(V): f\in \op{Hom}_{K\# D}(V,L)]
\eeq
denote the $K$-subalgebra of $L$ generated by all $f(V)$, where $f$ runs through
$\op{Hom}_{K\# D}(V,L)$. Notice that each $f(V)$ is a $K\# D$-sub-supermodule of $L$, whence
$K[V;L]$ is a $D$-stable $K$-sub-superalgebra of $L$. 
\end{definition}


\subsection{Interlude---the associated $\K\# \D$-module}\label{ASS}
As in the preceding subsection, let $K$ be a $D$-SUSY field, and let $V$ be a $K$-finite free 
$K$-module in ${}_D\M$. Just as in \eqref{INDeq8}, we define
\eq\label{ASSeq1}
\widetilde{V}:=V/K_1V. 
\eeq

The following is easy to see. 

\begin{lemma}\label{ASSlem1}
$\widetilde{V}$ naturally turns into a $\K$-module in ${}_{\D}\M$, whence
it is a $\K\# \D$-module with the parity forgotten. 
\end{lemma}

\begin{definition}\label{ASSdef1}
We call $\widetilde{V}$ the $\K \# \D$-module \emph{associated with} $V$.
\end{definition}

\begin{prop}\label{ASSprop1}
Let $L|K$ be an extension of $D$-SUSY field such that $L$ is $D$-simple. 
If $L$ is a splitting $D$-SUSY field for $V$, then $\L$ is a 
splitting $\D$-field for $\widetilde{V}$ (see \cite[Definition 3.0]{T}), or in other words, a purely even 
\tc{splitting}
$\D$-SUSY field for it. Moreover, if this is the case and $L$ is minimal, then $\L$ is minimal.
\end{prop}
\pf
Suppose that $V$ is of rank $m|n$. The condition (a) in Proposition \ref{SPDprop1}
\[
L\ot_K V\simeq L^{m|n}\ \, \text{in}\ \, {}_L({}_D\M)
\]
implies, with $\L\ot_L$ applied,
\[
\L\ot_{\K}\widetilde{V}\simeq \L^{m|n}\ \, \text{in}\ \, {}_{\L}({}_{\D}\M).
\]
The cited proposition applied to $\widetilde{V}$ proves the first half of the proposition.

Given an $f$ in $\op{Hom}_{K\# D}(V,L)$, let $\overline{f}$ be the composite 
\[
\widetilde{V}\to \widetilde{L}=L/K_1L\to \L
\]
of $\K\ot_{K} f$ (see \eqref{INDeq9}) with the natural projection $\widetilde{L}\to \L$. Then one sees
$\overline{f} \in \op{Hom}_{\K\# \D}(\widetilde{V},\L)$, and that for every $v \in V$, the element
$f(v)\, \op{mod}I_L$ in $\L$ coincides with $\overline{f}(v\, \op{mod}(K_1V))$. 
Therefore, 
$K[V;L]$ maps into $\K[\widetilde{V};\L]$ under the projection $L\to \L$. 
It follows that if $L=\Quot(K[V;L])$, then $\L=\Quot(\K[\widetilde{V};L])$. 
This proves the second half. 
\epf

In the continuation \cite{MasuokaII} of the present paper, the
solvability of $V$ is discussed in relation with the solvability
of $\widetilde{V}$, depending on the proposition above. 


\subsection{The characterizations}\label{CHA}

By convention, when we say that 
\[
a_1,\dots,a_m\mid a_{m+1},\dots,a_{m+n}
\]
are homogeneous elements in some object constructed on an underlying super-vector space,
$a_1,\dots,a_m$ are supposed to be even, and $a_m,\dots,a_{m+n}$ are odd. 

\begin{prop}\label{CHAprop1}
Suppose that $L$ is a $D$-SUSY field which is $D$-simple, and let $k:=L^D$. Then
homogeneous elements $a_1,\dots,a_m~|~a_{m+1},\dots,a_{m+n}$ of $L$ are $k$-linearly independent 
if and only if there exist homogeneous elements $d_1,\dots,d_m~|~d_{m+1},\dots,d_{m+n}$ of $D$ such that
$\big(d_ia_j\big)_{i,j} \in \mathsf{GL}_{m|n}(L)$. 
\end{prop}
\pf[Proof\, \tu{(cf. the proof of \cite[Proposition 1.5]{T})}]
``\emph{If}''.\ This is easy to see. In fact, if we have $\sum_{i=1}^{m+n} c_ia_i=0$, where $c_i\in k$, 
then the equation
with $d_i$, $1\le i \le m+n$, applied turns into
$\big(d_ia_j\big)_{i,j}{}^t(c_1,\dots,c_{m+n})=0$. If the matrix is invertible, then the vector is zero, so that all $c_i$ are zero. 

``\emph{Only if}''.\
Suppose that $k$-linearly independent homogeneous elements 
$a_1,\dots,a_m~|~a_{m+1},\dots,a_{m+n}$ of $L$ are given.

\tc{Recall from the paragraph following Proposition \ref{DSAprop3}
that $\op{Hom}(D,L)$ is an $L$-algebra in ${}_D\M$, with respect
to the convolution product; see \eqref{DSAeq00}. 
Through the morphism induced by the projection $L\to \L$, 
$\op{Hom}(D,\L)$ is a quotient $L$-algebra in the category,
where the morphism from $L$ is given by}
\[
L \to \op{Hom}(D, \L),\quad a \mapsto (d \mapsto \overline{da});
\]
cf. \eqref{DSAeq1a}. 
Here and in what follows, we let $\overline{b}:=b\, \op{mod} I_L$ for $b\in L$. 
Since we see that $\op{Hom}(D, \L)^D$ is naturally isomorphic to $\op{Hom}(D/D^+, \L)=\L$, 
the morphism $\mu$ for
$\op{Hom}(D,\L)$ (see \eqref{DSAeq1}) may be supposed to be
\[
\mu : L\otk \L \to \op{Hom}(D, \L),\quad \mu(a \ot x)(d) =\overline{da}\cdot x.
\]
Since $L$ is $D$-simple, it follows by Proposition \ref{DSAprop1} (2) that this $\mu$ 
is an injection, which is naturally regarded to be $\L$-superlinear; see the proof of Proposition \ref{DSAprop1a}.

Let $U:=\bigoplus_{i=1}^{m+n}ka_i$. By restriction we have an injection $U\otk \L \to \op{Hom}(D, \L)$.
As its $\L$-predual we have the $\L$-superlinear surjection
\[
D \ot_k \L \to \op{Hom}_k(U, \L),\quad d \ot x \mapsto (u\mapsto \overline{du}\cdot x). 
\]
We can choose as an $R$-basis of $D$, homogeneous elements
$d_1,\dots,d_m~|~d_{m+1}$, $\dots,d_{m+n}$
such that the last surjection, restricted to $C\ot L$, turns into an isomorohism, where we have set $C:=\bigoplus_{i=1}^{m+n}Rd_i$. 
The condition means that the morphism in $\M_L$
\[
C\ot_k L\to \op{Hom}_k(U,L),\quad d_{j} \ot a\mapsto (a_i \mapsto (d_ja_i)a)
\]
is isomorphic modulo the nilpotent $I_L$, whence it is isomorphic, indeed. 
The source and the target are identified with $L^{m|n}$ through
the basis $d_1\ot 1,\dots, d_{m+n}\ot 1$, and the dual basis of $a_1,\dots a_{m+n}$, respectively. 
Supposing that elements of $L^{m|n}$ are presented as column vectors, we see that the isomorphism 
is identified with the left multiplication by the matrix $\big( d_ja_i \big)_{i,j}$,
which is, therefore, in $\mathsf{GL}_{m|n}(L)$; hence its transpose $\big( d_ia_j \big)_{i,j}$ is as well. 
\epf

In the remaining of this section we assume that $D$ satisfies (D1--2), 
as in the preceding two sections.

\begin{theorem}\label{CHAthm1}
Suppose that $L|K$ is an extension of $D$-SUSY fields which satisfies the conditions 
\begin{itemize}
\item[(LK5)] $L$ and $K$ are both $D$-simple, and
\item[(LK6)] $L^D=K^D$.
\end{itemize}
Then the following are equivalent: 
\begin{itemize}
\item[(a)]
$L|K$ is a SUSY PV extension in the sense of Definition \ref{DEFdef1} (2);
\item[(b)]
There exists a $K$-finite free $K$-module $V$ in ${}_D\M$, which is cyclic as a $K\# D$-supermodule 
(or namely, generated by a single homogeneous 
element), 
such that $L$ is a minimal splitting $D$-SUSY field for $V$;
\item[(c)]
There exists a $K$-finite free $K$-module $V$ in ${}_D\M$, 
such that $L$ is a minimal splitting $D$-SUSY field for $V$;
\item[(d)]
There exists an invertible matrix with entries in $L$, say, $X=(x_{ij}) \in \mathsf{GL}_{m|n}(L)$, such that
\begin{itemize}
\item[(i)] $(dX)X^{-1}\in \mathsf{M}_{m|n}(K)$ for every $d \in D$, where $dX:=\big(dx_{ij}\big)_{i,j}$, and
\item[(ii)] $L=\Quot(K[x_{ij}])$, where $K[x_{ij}]$ denotes the $K$-sub-superalgebra of $L$ generated by the entries $x_{ij}$
\tc{of $X$.} 
\end{itemize}
\end{itemize}
\end{theorem}

A matrix in $X$ in $\mathsf{GL}_{m|n}(L)$ is said to be $\mathsf{GL}_{m|n}$-\emph{primitive} if it satisfies Condition (i) of (d). 

We let $k:=L^D\, (=K^D)$, as before. We will use the presentation, such as $dX$ as above, which
indicates the relevant operation applied to each entry of matrices.

\pf[Proof\, \tu{(cf. the proof of \cite[Theorem 3.3]{T})}]\ 
(a)$\Rightarrow$(b).\ 
Suppose that $(L|K, A, H)$ is a SUSY PV extension. 
We can choose $k$-linearly independent homogeneous elements
$u_1,\dots,u_m~|~u_{m+1},\dots,u_{m+n}$ of $A$, 
which span
an $H$-sub-supercomodule $A$, and which generate the $K$-superalgebra $A$. 
Let
\[
\u:=(u_1,\dots,u_m,u_{m+1},\dots,u_{m+n}).
\]
Suppose that elements of $A^{m|n}$ (or $L^{m|n}$) are
presented as row vectors, whence $\u\in A^{m|n}$, in particular. 
There uniquely exist homogeneous elements $c_{ij}$, $1\le i,j\le m+n$, of $H\, (=A\ot_KA)^D$ such that
\[
\theta_{\! A}(u_j)=\sum_{j}u_i\ot_k c_{ij}.
\]
This equation is rewritten so as
\[
1\ot_K u_j=\sum_i(u_i\ot_K1)c_{ij}\ \, \text{in}\ \, A\ot_KA,
\]
or by matrix presentation, so as
\eq\label{CHAeq1}
1\ot_K \u =(\u\ot_K 1)C,
\eeq
where
\[ 
C:=\big(c_{ij}\big)_{i,j}.
\]
In addition, we have let $1\ot_K\u:=(1\ot_K u_1,\dots,1\ot_Ku_{m+n})$, meaning to insert $1\ot_K$ to each entry;
see the sentence preceding the proof. 
One sees that $C$ is in $\mathsf{GL}_{m|n}(H)$ with inverse $\cS(C)$, 
where $\cS$ denotes the antipode of $H$. 
By Proposition \ref{CHAprop1} we have homogeneous elements
$d_1,\dots,d_m~|~d_{m+1},\dots,d_{m+n}$ of $D$
such that the matrix $P:=\big(d_iu_j\big)_{i,j}$ satisfies
\eq\label{CHAeq2}
P \in \mathsf{GL}_{m|n}(L). 
\eeq

We claim that for every $d\in D$, $(d\u)P^{-1}$ has entries in $K$, or equivalently, 
\[
1\ot_K(d\u)P^{-1}=(d\u)P^{-1}\ot_K 1.
\]
Indeed, we see from the equation \eqref{CHAeq1} with $d$ and $d_i$ applied that 
\[
1\ot_K d\u =(d\u\ot_K 1)C,\quad 1\ot_K P=(P\ot_K 1)C. 
\]
Therefore, 
\begin{align*}
1\ot_K(d\u)P^{-1}&=(d\u\ot_K 1)C(1\ot_K P)^{-1}\\
&=(d\u\ot_K 1)CC^{-1}(P\ot_K 1)^{-1}\\
&=(d\u)P^{-1}\ot_K 1. 
\end{align*}

By \eqref{CHAeq2} we can define a $K$-sub-supermodule of $L^{m|n}$ by
\[
V:=\bigoplus_{i=1}^{m+n}K(d_i\u),
\]
which is $K$-finite free of rank $m|n$ with basis $d_i\u$, $1\le i\le m+n$. 
By the claim above, $V$ is a $K\# D$-supermodule generated by $\u$. 
Again by \eqref{CHAeq2} one has $L\ot_KV=L^{m|n}$, whence $L$ is a splitting $D$-SUSY field for $V$.
It is minimal since the $K$-sub-superalgebra generated by the images $K(d_iu_j)$ of the $m+n$ projections
$V \subset L^{m|n}\to L$ includes 
\tc{all $u_i$, and so $A$, as well. }  

(b)$\Rightarrow$(c).\ Obvious.

(c)$\Rightarrow$(d).\
Suppose that $L$ is a minimal splitting $D$-SUSY field for a $K$-module in ${}_D\M$,
\eq\label{CHAeq2a1}
V=Kv_1\oplus \dots \oplus Kv_m\oplus Kv_{m+1}\oplus \dots \oplus Kv_{m+n},
\eeq
with homogeneous $K$-free basis elements
$v_1,\dots,v_m~|~v_{m+1},\dots,v_{m+n}$.  
By Proposition \ref{SPDprop1}, $\op{Hom}_{K\# D}(V,L)$ has $k$-dimension $m|n$. Choose as 
homogeneous $k$-basis elements 
\eq\label{CHAeq2a2}
f_1,\dots,f_m~|~f_{m+1},\dots,f_{m+n}. 
\eeq
The canonical morphism
\eq\label{CHAeq2a3}
\mu':\op{Hom}_{K\# D}(V,L)\ot_k L \to 
\op{Hom}_{K-}\! \Big(\bigoplus_{i=1}^{m+n}Kv_i,\ L\Big)=L^{m|n}
\eeq
is now an isomorphism. We regard this $\mu'$ as 
\eq\label{CHAeq2a4}
L^{m|n}\overset
{\simeq}{\longrightarrow}L^{m|n},
\eeq
identifying the source with $L^{m|n}$ through the free basis $f_1\ot 1,\dots, f_{m+n}\ot 1$.  
It is seen to be the left multiplication by the matrix
\eq\label{CHAeq2a5}
X:=\big(f_j(v_i)\big)_{i,j},
\eeq
where we now suppose that elements of $L^{m|n}$ are presented as column vectors. 
Hence, $X \in \mathsf{GL}_{m|n}(L)$. Let
\[
\tc{\v:={}^t(v_1,\dots,v_m,v_{m+1},\dots,v_{m+n}).}
\]
For every $d \in D$, there exists a matrix $E(d)\in \mathsf{M}_{m|n}(K)$ such that
$d\v=E(d)\v$. Since we have
\[
d(f_j(\v))= f_j(d\v)=f_j(E(d)\v)=E(d)f_j(\v),\ \, \text{whence}\ \, dX=E(d)X,
\]
it follows that $X$ is $\mathsf{GL}_{m|n}(L)$-primitive, or namely,
$X$ satisfies Condition (i) of (d). It satisfies Condition (ii) as well, since 
$L=\Quot(K[f_j(v_i)])$ by the minimality of $L$, 
and the generators $f_j(v_i)$ of $K[f_j(v_i)]$ are precisely the entries of $X$. 

(d)$\Rightarrow$(a).\
Assume (d). Suppose that $X=\big(x_{ij}\big)_{i,j}$ is a $\mathsf{GL}_{m|n}$-primitive matrix such that $L=\Quot(K[x_{ij}])$,
and set $Y:=X^{-1}$. Let $d \in D$. Define
\[
F(d):=(dX)Y\, (\in \mathsf{M}_{m|n}(K)). 
\]
Then we have
\eq \label{CHAeq3}
dX=F(d)X. 
\eeq
Applying $d$ to the both sides of $YX=I$, we have 
$(d_{[1]}Y)F(d_{[2]})X=\varepsilon(d)I$, whence
\eq\label{CHAeq4}
(d_{[1]}Y)F(d_{[2]})=\varepsilon(d)Y. 
\eeq

We have to prove that there exists a desired $A$, and $L|K$ is of finite type (i.e., $\L|\K$ is a finitely generated
field extension, see Definition \ref{DEFdef1} (1)). 

Define $A$ to be the $K$-sub-superalgebra of $L$ generated by all entries of $X$ and of $Y$;
it is seen, in view of \eqref{ASGeq2}, to be a localization of
the $K$-sub-superalgebra generated by all entries of $X$. 
Since that $K$-sub-superalgebra 
is $D$-stable in $L$ by \eqref{CHAeq3}, 
\tc{so is $A$ by Proposition \ref{EXDprop1} (1).} 
We wish to show that $A$ satisfies (A1--2).
As for (A2), i.e., $L=\Quot(A)$, we are done since $L=\Quot(K[x_{ij}])$ and $A \supset K[x_{ij}]$; 
the result implies that $L|K$ is of finite type. Set
\[
H:=(A\ot_KA)^D. 
\]
\tc{As for (A1), }
we wish to prove $A\ot_K A\subset A\cdot H$; the result will complete the proof. 
Define two matrices in $\mathsf{M}_{m|n}(A\ot_KA)$ by
\eq\label{CHAeq5}
Z:=(Y\ot_K1)(1\ot_KX),\quad W:=(1\ot_KY)(X\ot_K1). 
\eeq
These are mutual inverses.
We claim that $Z$ and $W$ both have entries in $H$. This will imply the desired inclusion,
since we then have
that 
\eq\label{CHAeq6}
1 \ot_K X=(X\ot_K1)Z,\quad 1\ot_KY=W(Y\ot_K1).
\eeq
To prove the claim we let $d \in D$, and compute, using \eqref{CHAeq3} and \eqref{CHAeq4}, so as
\begin{align*}
dZ&=(d_{[1]}Y \ot_K 1)(1 \ot_K d_{[2]}X)\\
&=(d_{[1]}Y \ot_K 1)(1\ot_K F(d_{[2]}))(1 \ot_K X)\\
&=(d_{[1]}Y \ot_K 1)(F(d_{[2]})\ot_K 1)(1 \ot_K X)\\
&=\varepsilon(d)(Y\ot_K1)(1\ot_K X)=\varepsilon(d)Z.
\end{align*}
Similarly, we see that $dW=\varepsilon(d)W$, which concludes the claim. 
\epf

\begin{rem}\label{CHArem1}
For later use we record here some facts that are seen from the last proof.  
Suppose that $(L|K,A,H)$ is a SUSY PV extension with $k=L^D\, (=K^D)$, and $L$ is a minimal
splitting $D$-SUSY field of a $K$-finite free $K\# D$-suermodule $V$.
\begin{itemize}
\item[(1)]
Suppose $V=\bigoplus_{i=1}^{m+n}Kv_i$ as in \eqref{CHAeq2a1}, and $\op{Hom}_{K\# D}(V,L)$ has the
$k$-basis $f_1,\dots,f_{m+n}$ as in \eqref{CHAeq2a2}. 
Set 
\[
\boldsymbol{x}_j:={}^t(f_j(v_1),\dots,f_j(v_{m+n})),\quad 1\le j \le m+n, 
\]
and let
\begin{equation*}
X:=\big( f_j (v_i) \big)_{i,j}\, (=\big(\boldsymbol{x}_1,\dots, \boldsymbol{x}_{m+n}\big))
\end{equation*}
denote the matrix with the $j$-th column $\boldsymbol{x}_j$, 
as in \eqref{CHAeq2a5}.
We may and we do suppose that this $X$ is the $\mathsf{GL}_{m|n}$-primitive which appeared in 
the last part, (d) $\Rightarrow$ (a), of the proof, 
whence $X\in \mathsf{GL}_{m|n}(A)$.
It follows that the canonical isomorphism $\mu'$ in \eqref{CHAeq2a3} restricts to an isomorphism
\eq\label{CHAeq7}
\op{Hom}_{K\# D}(V,A)\ot_k A \overset{\simeq}{\longrightarrow} \op{Hom}_{K-}(V,A)=A^{m|n},
\eeq
under which 
\eq\label{CHAeq8}
f_i \mapsto \boldsymbol{x}_i,\quad 1\le i\le m+n.
\eeq
\item[(2)]
Let us be in situation of the last part of the proof. 
Suppose that $\boldsymbol{x}_j$ and $X$ are as in Part 1 above, and define
$Z$ and $W$ as in \eqref{CHAeq5}. 
Then we have the following.
\begin{itemize}
\item[(i)] $H$ is generated by all entries of $Z$ and of $W$. 
\item[(ii)] We have
\[
\cS(Z)=W,\quad Z=\cS(W),
\]
where $\cS$ denotes the antipode of $H$. 
\item[(iii)] The structure morphism $\theta_{\! A} : A \to A\ot_k H$ of $A$ satisfies
\[
\tc{\theta_{\! A}(\boldsymbol{x}_j)=\sum_{i=1}^{m+n}\boldsymbol{x}_i\ot_kz_{ij},\quad 1\le j\le m+n,}
\]
\tc{where $Z=\big( z_{ij} \big)_{i,j}$.}
\end{itemize}
In fact, the argument of 
\tc{verifying Condition (A1)} 
proves (i);
indeed, since the $k$-subalgebra of $H$ generated by all entries of $W$ and of $Z$
is shown to generate the left $A$-module $A\ot_KA$, 
\tc{it must coincide with $H$, as is seen in view of 
the isomorphism $\mu : A\ot_k H\overset{\simeq}{\longrightarrow} A\ot_KA$.}
The first equation of \eqref{CHAeq6} shows (iii).
This (iii), combined with the fact that
$Z$ and $W$ are mutual inverses, implies (ii). 
\end{itemize}
\end{rem}


\section{Tannaka-type theorem}\label{secTTT}

This section consists of two subsection.
In the first subsection we do not pose any additional assumption to $D$, while in
the second, we assume that $D$ satisfies (D1--2). 
 

\subsection{The dual of a $K\# D$-supermodule}\label{DUA}
Let us be in the situation of Section \ref{SPD}, and suppose thus that
$K$ is a $D$-SUSY field. 
Recall the identification
\begin{equation}\label{DUAeq1}
({}_K({}_D\M),\ot_K,K)=(({}_D\M)_K,\ot_K,K)
\end{equation}
of symmetric monoidal categories.
One may choose an appropriate one from the these two categories, according to the situation. 

Let $V$ be a $K$-finite free object in the category, and let
\[
{}^*V:=\op{Hom}_{K-}(V,K),\quad V^*:=\op{Hom}_{-K}(V,K)
\]
denote the mutually isomorphic objects in the category which are
defined in Proposition \ref{BSHprop1}, where the $L$ there
is now specified to $K$. 

\begin{lemma}\label{DUAlem1}
We have the following.
\begin{itemize}
\item[(1)]\
${}^*V$, together with the obvious morphisms
\[
\mathrm{eval} : V\ot_K {}^*V\to K,\quad \mathrm{coeval} : K\to {}^*V \ot_K V,
\]
form a left dual object of $V$, or explicitly, we have
\[
(\mathrm{eval}\ot \mathrm{id}_V)\circ(\mathrm{id}_V\ot\mathrm{coeval})=\mathrm{id}_V. 
\]
\item[(2)]\
$V^*$, together with the obvious morphisms
\[
\mathrm{eval}' : {}^*V\ot_K V\to K,\quad \mathrm{coeval}' : K\to V \ot_K V^*,
\]
form a right dual object of $V$, or explicitly, we have
\[
( \mathrm{id}_V \ot \mathrm{eval}')\circ(\mathrm{coeval}' \ot\mathrm{id}_V)=\mathrm{id}_V. 
\]
\end{itemize}
\end{lemma}

This is directly verified. As for Part 1, for example,
essential is to see that $\mathrm{eval}$ and $\mathrm{coeval}$ are $D$-linear;
they are defined explicitly by
\[
\mathrm{eval}(v\ot f)=f(v),\quad \mathrm{coeval}(1)=\sum_i {}^*\! e_i\ot e_i,
\]
where $(e_i)$ in $V$ and $({}^*\! e_i)$ in ${}^*V$ are mutually dual $K$-free bases.

Since the monoidal category in question is symmetric, the left and the right duals above are necessarily 
isomorphic with each other. 
We emphasize that an explicit isomorphism is already given by Proposition \ref{BSHprop1} (3). 

\subsection{The theorem; Hopf-algebraic formulation and familiar version}\label{TNK}
In this subsection we assume that $D$ satisfies (D1--2). 

Suppose that $(L|K, A, H)$ is a SUSY PV extension, and set $k:=L^D\, (=K^D)$. 
Recall that $A$ is $D$-simple, and $A^D=k$. 

We will say that supermodules are $K$-\emph{finite} or $A$-\emph{finite}, if they are
finitely generated over $K$ or over $A$. 

\begin{definition}\label{TNKdef1}
Here are some definitions.
\begin{itemize}
\item[(1)]
An object $M$ in ${}_A({}_D\M)$ is said to be \emph{splitting} if the canonical morphism
$\mu_M:A\otk M^D \to M$, $\mu_M(a\ot m)=am$ (see \eqref{DSAeq1}), which is injective by 
\tc{Proposition \ref{DSAprop1} (2),}
is isomorphic.
\item[(2)]
An object $V$ in ${}_K({}_D\M)$ is said to be \emph{split by} $A$, if the base extension 
$A\ot_K V$, which is naturally an object in ${}_A({}_D\M)$, is splitting. We let 
\[
\MAK
\]
denote the full subcategory of ${}_K({}_D\M)$ consisting of those objects which are $K$-finite, and 
split by $A$.
\end{itemize}
\end{definition}

By using the $K$-faithful flatness of $A$, one sees that
if $V$ in ${}_K({}_D\M)$ is split by $A$, then it is $K$-flat, or equivalently, $K$-free; 
\tc{see Proposition \ref{SUFprop2} (1).}
Therefore, every object $V$ of $\MAK$ is $K$-finite free, and hence has a dual
$V^*\, (\simeq {}^*V)$ in ${}_K({}_D\M)$. 

Notice that the symmetric monoidal category ${}_K({}_D\M)=({}_K({}_D\M),$ $\ot_K,K)$
is $k$-linear abelian. 

\begin{prop}\label{TNKprop1}
$\MAK$ is closed under constructing tensor products, direct sums, sub-quotients and duals
in ${}_K({}_D\M)$.\ Therefore, 
$\MAK$ forms a $k$-linear abelian, rigid symmetric monoidal
category with respect to the tensor product $\ot_K$, the unit object $K$ and the supersymmetry. 
\end{prop}
\pf
We have only to prove the claim: 
\emph{all $A$-finite splitting objects in ${}_A({}_D\M)$ 
are closed under the prescribed constructions.} 
Indeed, the claim implies the first half of the
proposition, from which the second half will easily follow. 
To see this implication for constructing (i)~tensor products and (ii)~duals, one should notice
that (i)~the base-extension functor
\eq\label{TNKeq1}
A\ot_K : ({}_K({}_D\M), \ot_K, K)\to ({}_A({}_D\M),\ot_A, A)
\eeq
is monoidal, and (ii)~it preserves the dual, or explicitly, we have a canonical isomorphism
\[
\op{Hom}_{K-}(V, K)\ot_K A \overset{\simeq}{\longrightarrow} \op{Hom}_{K-}(V, A)=\op{Hom}_{A-}(A\ot_K V, A)
\] 
in $({}_D\M)_A$ for every $K$-finite free $V$ in ${}_K({}_D\M)$. 

The claim for tensor products and direct sums are easy to prove. 
To prove the remaining, let $M$ be an $A$-finite splitting object in ${}_A({}_D\M)$. 
Obviously, one has $\dim_kM^D<\infty$. 
If $N$ is a sub-object of $M$, we have the commutative diagram with exact rows
\begin{equation*}
\begin{xy}
(0,7)   *++{0}  ="1",
(20,7) *++{A\ot_k N^D}  ="2",
(46,7) *++{A\ot_k M^D}  ="3",
(76,7) *++{A\ot_k (M/N)^D} = "4",
(0,-8)  *++{0} ="5",
(20,-8) *++{N} ="6",
(46,-8)  *++{M} ="7",
(76,-8)  *++{M/N} ="8",
(96,-8)  *++{0.} ="9",
{"1" \SelectTips{cm}{} \ar @{->} "2"},
{"2" \SelectTips{cm}{} \ar @{->} "3"},
{"3" \SelectTips{cm}{} \ar @{->} "4"},
{"5" \SelectTips{cm}{} \ar@{->} "6"},
{"6" \SelectTips{cm}{} \ar@{->} "7"},
{"7" \SelectTips{cm}{} \ar @{->} "8"},
{"8" \SelectTips{cm}{} \ar @{->} "9"},
{"2" \SelectTips{cm}{} \ar @{>->}^{\mu_N} "6"},
{"3" \SelectTips{cm}{} \ar @{->}^{\mu_M}_{\simeq} "7"},
{"4" \SelectTips{cm}{} \ar @{>->}^{\mu_{M/N}} "8"},
\end{xy}
\end{equation*}
Since $\mu_M$ is isomorphic, 
we see first, $\mu_{M/N}$, and then $\mu_N$ are isomorphic. This proves the claim for sub-quotients. 

Identify so as $M=A\ot_k M^D$ through $\mu_M$. Then we have a natural isomorphism
\[
{}^*\! M=\op{Hom}_{A-}(A\ot_k M^D, A)\overset{\simeq}{\longrightarrow} 
\op{Hom}_{k}(M^D, A)=\op{Hom}_k(M^D,k)\ot_k A
\]
in $({}_D\M)_A$. It is then obvious that $\mu_{{}^*\! M}$ is isomorphic. This proves the claim for duals.
\epf 

Let 
\[
\mathtt{Smod}^H=(\mathtt{Smod}^H, \ot_k, k)
\]
denote the monoidal full subcategory of $\M^H$ which consists of all $k$-finite-dimensional $H$-supercomodules.
This is a $k$-linear abelian, rigid symmetric monoidal category, and it may be identified with
\[
{}_{\G(L|K)}\mathtt{Smod},
\]
which presents the category of finite-dimensional supermodules over the Galois supergroup $\G(L|K)$ of $L|K$; see \eqref{ASGeq1}. 

\begin{theorem}\label{TNKthm1}
For every object $V$ in ${}_D\mathtt{Smod}_{(L|K)}$, $(V\ot_KA)^D$ is an object in ${}_{\G(L|K)}\mathtt{Smod}$,
where $H$ is supposed to coact through the tensor factor $A$ of $V\ot_KA$. 
Moreover, the assignment $V \mapsto (V\ot_KA)^D$
gives rise to a $k$-linear, symmetric monoidal equivalence
\eq\label{TNKeq2}
{}_D\mathtt{Smod}_{(A|K)}\approx {}_{\G(L|K)}\mathtt{Smod}
\eeq
between the $k$-linear abelian, rigid symmetric monoidal categories. 
\end{theorem}
\pf
Let ${}_D\M^H=({}_D\M^H,\ot_k, k)$ denote the symmetric monoidal category which consists of all $D$-supermodule 
$M$ equipped with $D$-superlinear $H$-supercomodule structure $M \to M \ot_k H$, where
$D$ is supposed to act on  $M\ot_k H$ through the tensor factor $M$. Since $A$ is an algebra in this category,
we have the category ${}_A({}_D\M^H)$ of $A$-modules in the category, which, in fact, forms a symmetric monoidal 
category, $({}_A({}_D\M^H),\ot_A, A)$. As an obvious variant of the Hopf-module theorem (see Theorem \ref{HGEthm0}), 
a $k$-linear, symmetric monoidal equivalence
\[
{}_K({}_D\M)\approx {}_A({}_D\M^H)
\]
is given by 
$V \mapsto V\ot_K A$  (cf. \eqref{TNKeq1}) and
$M^{\op{co} H}\mapsfrom M$.
We see that this restricts to the equivalence of the same kind
\eq\label{TNKeq3}
{}_D\mathtt{Smod}_{(A|K)}\approx {}_A({}_D\mathtt{Smod}^H)_{\op{split}},
\eeq
where the latter denotes the monoidal full subcategory of  ${}_A({}_D\M^H)$ which consists of all 
$A$-finite splitting objects. 

Notice that for $M$ in ${}_A({}_D\mathtt{Smod}^H)_{\op{split}}$, the isomorphism 
$\mu_M : A\ot_k M^D\overset{\simeq}{\longrightarrow}M$ is
$H$-super-colinear, as well, where $A\ot_k M^D$ is supposed to be the tensor product
of two $H$-supercomodules. Then one sees that a $k$-linear, symmetric monoidal equivalence
\eq\label{TNKeq4}
{}_A({}_D\mathtt{Smod}^H)_{\op{split}}\approx {}_{\G(L|K)}\mathtt{Smod}
\eeq
is given by the symmetric monoidal functors 
\[
M\mapsto M^D \quad\text{and}\quad  A \ot_k U \mapsfrom U
\]
equipped with the obvious monoidal structures, which are mutually quasi-inverses. 

The composite of the two equivalences above proves the theorem.
\epf

\begin{rem}\label{TNKrem1}
\tc{As is seen from the proof above,}
the category equivalences \eqref{TNKeq3} and \eqref{TNKeq4} extend, in an obvious manner, to those among the 
larger categories that are enlarged by removing the finiteness restrictions. Consequently, \eqref{TNKeq2} 
extends to a $k$-linear, symmetric monoidal equivalence
\[
{}_D\M_{(A|K)}\approx {}_{\G(L|K)}\M\, (=\M^H) 
\]
between the $k$-linear abelian, symmetric monoidal categories, 
where ${}_D\M_{(A|K)}$ denotes the monoidal full subcategory 
of ${}_K({}_D\M)$ consisting of all objects that are split by $A$.
\end{rem}

Notice from Theorem \ref{TNKthm1} that $L$ is a minimal splitting $D$-SUSY field of some
$K$-finite free object in ${}_K({}_D\M)$. 

\begin{corollary}\label{TNKcor1}
If $L$ is a minimal splitting $D$-SUSY field of a $K$-finite free object $V$ in ${}_K({}_D\M)$, then $V$ is contained
in ${}_D\mathtt{Smod}_{(A|K)}$, 
and ${}_D\mathtt{Smod}_{(A|K)}$ is, as an abelian, rigid symmetric monoidal category, generated by $V$. 
This means that the category consists of all sub-quotients 
of the finite direct sums
\[
W_1\oplus W_2\oplus \dots \oplus W_r,\quad r>0
\]
where each $W_i$ is isomorphic to the tensor product of some finitely many objects which are
isomorphic to either $V$ or its dual $V^*\, (\simeq {}^*V)$; we remark that the relevant constructions 
are the same as the ones done in ${}_K({}_D\M)$, as is seen from Proposition \ref{TNKprop1}. 
\end{corollary}
\pf
We may suppose that we are in the situation of Remark \ref{CHArem1}. In what follows,
(i), (ii) and (iii) refer to those of Part 2 of the Remark. 

To prove $V \in {}_D\mathtt{Smod}_{(A|K)}$, we may prove that ${}^*V=\op{Hom}_{K-}(V,K)$ instead is in the category.
This fact for ${}^*V$ 
follows, since the isomorphism \eqref{CHAeq7} shows that ${}^*V\ot_KA\, (=\op{Hom}_{A-}(V,A))$
is splitting. To prove that $V$ generates $\mathtt{Smod}_{(A|K)}$, we may replace $V$ by ${}^*V$, again. To prove the desired result for ${}^*V$, 
it suffices by Theorem \ref{TNKthm1} that the corresponding 
$H$-supercomodule
\[
U:=({}^*V\ot_K A)^D\, (=\op{Hom}_{K\# D}(V,A))
\]
generates $\mathtt{Smod}^H$, or equivalently, the coefficient super-vector space $\op{cf}(U)$ of $U$, together with the image $\cS(\op{cf}(U))$ under the antipode $\cS$, generate $H$. Here, 
$\op{cf}(U)$ is the smallest super-vector subspace $C$ of $H$
such that $\theta(U) \subset U\ot_k C$, where $\theta : U \to U \ot_k H$ denotes the structure morphism of $U$. We see from \eqref{CHAeq8} and  (iii) that $\op{cf}(U)$ is spanned by the
entries 
\tc{$z_{ij}$}
of the matrix $Z$, the desired result follows from (i)--(ii). 
\epf

\begin{rem}\label{TNKrem2}
Let us be the situation of Corollary \ref{TNKcor1} above. 
In view of the result there proved, we may denote ${}_D\mathtt{Smod}_{(A|K)}$
alternatively by $\langle\langle V\rangle\rangle_{\ot}$, using the generator $V$, and reformulate the $k$-linear, 
symmetric monoidal equivalence \eqref{TNKeq2}
as
\[
\langle\langle V\rangle\rangle_{\ot}\approx {}_{\G(L|K)}\mathtt{Smod}, 
\]
\tc{which is consistent with}
the standard formulation of the Tannaka-type Theorem found in \cite{Del} and other literature. 
\end{rem}

\section{Unique existence of minimal splitting $D$-SUSY fields}\label{secUEX}

In this section we assume
\begin{itemize}
\item[(D1)] $\dim (\g_D)_1<\infty$, and
\item[(D3)] $D$ is irreducible, or explicitly, $D=U(\g_D)$.
\end{itemize}

\begin{theorem}\label{UEXthm1}
Let $K$ be a $D$-simple $D$-SUSY field such that the field $K^D$ of $D$-invariants in $K$ is algebraically
closed. Then, given a $K$-finite free $K$-module $V$ in ${}_D\M$, there exists a $D$-simple minimal-splitting
$D$-SUSY field $L$ for $V$ such that $L^D=K^D$. Such an $L$ is unique up to $K$-algebra isomorphism in ${}_D\M$. 
\end{theorem}

To prove the theorem
let $k:=K^D$, as before. Suppose that $V=\bigoplus_{i=1}^{m+n}Kv_i$ is of rank $m|n$, 
having a $K$-free basis consisting of homogeneous elements
$v_1,\dots,v_m~|~v_{m+1},$ $\dots,v_{m+n}$. Excluding the trivial case we may suppose $m>0$ or $n>0$. We divide
the proof into two parts, for the existence and for the uniqueness. 

\pf[Proof of the existence]\ 
We will use matrix presentation such as used in Theorem \ref{CHAthm1} and its proof. 

\medskip

\emph{Step 1}.\
Set $\v:={}^t(v_1,\dots,v_m,v_{m+1},\dots,v_{m+n})$.  
There uniquely exists 
\tc{an $R$-superlinear morphism $F:D\to\mathsf{M}_{m|n}(K)$}
such that
\eq\label{UEXeq1}
d\v =F(d)\v,\quad d \in D.
\eeq
Since $V$ is a $D$-supermodule, we have
\eq\label{UEXeq2}
F(1)=I; \quad F(dd')=(d_{[1]}F(d'))F(d_{[2]}),\ d, d'\in D. 
\eeq

Let $\cO_k:=\cO_k(\mathsf{GL}_{m|n})$; recall from Example \ref{ASGex1} that
this is the polynomial $k$-superalgebra in the indeterminates $t_{ij}$, $1\le i,j \le m+n$, 
localized at $\op{det}_0$; see \eqref{ASGeq3}. 
Set $T:=\big( t_{ij}\big)_{i,j}$. 
This step aims to prove: 

\medskip

\noindent
{\bf Claim.}\ \emph{The base extension $\cO_K =K\otk \cO_k$ of $\cO_k$ is made 
into a $K$-algebra in ${}_D\M$, so that}
\eq\label{UEXeq3}
d(1\otk T)=(F(d)\otk 1)(1\otk T),\quad d \in D. 
\eeq

\medskip

Recall from the paragraph preceding Remark \ref{DSArem1} that given an algebra  
in ${}_D\M$, there is naturally associated a superalgebra morphism such as $\rho$ in \eqref{DSAeq1a}.
Notice that the algebra structure recovers from the $\rho$ by the formula 
\begin{equation}\label{UEXeq3a}
da=\rho(a)(d)
\end{equation}
found in \eqref{DSAeq1a}.
To prove the claim, we wish to construct a superalgebra morphism, $\cO_K\to\op{Hom}(D,\cO_K)$,
which is to be associated with the desired $K$-algebra structure on $\cO_K$ in ${}_D\M$. 

Notice that we have a $k$-superalgebra morphism,
\[
\beta : \cO_k \to \op{Hom}(D, K),\ \beta(T)=F.
\]
This is seen to be 
well-defined at $\op{det}_0^{-1}$, in view of the assumption (D3) and
the fact that $\beta(\op{det}_0)(1)\, (=1)$ is invertible in $K$; see the proof of Proposition \ref{EXDprop1} (1). 
Recall that $\op{Hom}(D, K)$ is regarded as a $K$-superalgebra (in fact, as a $K$-algebra in ${}_D\M$)
through the associated $\rho_K$;
\tc{see the paragraph following Proposition \ref{DSAprop1}. Let}
\[
\tc{\widehat{\beta} : \cO_K=K\otk \cO_k \to \op{Hom}(D,K)} 
\]
\tc{denote the $K$-superalgebra morphism which associates
to an element $x\ot_k u\in K\otk \cO_k$, the convolution
product $\rho_K(x)*\beta(u)$.
This is explicitly presented so as}
\[
\tc{\widehat{\beta}(x\ot_k T)(d)=(d_{[1]}x)F(d_{[2]}),\quad x \in K.}
\]

We regard $\cO_K$ naturally as an algebra in $\M^{\cO_k}$ by the coproduct 
\[
\op{id}_K\ot \Delta : 
\cO_K=K\otk \cO_k\to K\otk \cO_k\otk \cO_k=\cO_K\otk \cO_k 
\]
of $\cO_k$. 
Make the composite $(\widehat{\beta} \otk \op{id}_{\cO_k})\circ(\op{id}_K\ot \Delta):\cO_K\to
\op{Hom}(D,K)\otk \cO_k$, and further compose this
with the natural superalgebra morphism
\[
\op{Hom}(D,K)\otk \cO_k\to \op{Hom}(D,K\otk \cO_k),\quad f \ot a\mapsto (d\mapsto f(d)\ot a). 
\]
Let $\rho : \cO_K\to\op{Hom}(D,\cO_K)$ denote the result. 
This is what we wished to construct.
Indeed, in view of the formula $\Delta(T)=(T\otk 1)(1\otk T)$ of the coproduct, 
we see that the $D$-action which arises by the formula \eqref{UEXeq3a} satisfies \eqref{UEXeq3}. 

The $\rho$ above is seen to be a $K$-superalgebra morphism, where $\op{Hom}(D,\cO_K)$ is regarded
as a $K$-superalgebra through the composite $K\overset{\rho_K}{\longrightarrow} \op{Hom}(D,K)\to \op{Hom}(D,\cO_K)$,
where the second arrow indicates the superalgebra morphism arising from the inclusion 
$K\hookrightarrow\cO_K$.
In order to see that the $D$-action defines a $K$-algebra structure in ${}_D\M$, 
it remains to show that the following holds:
given $v\in \cO_K$, we have $1v=v$, and $(dd')v = d(d'v)$ for all $d,d'\in D$; see \cite[Section 7.0]{Sw}, modifying 
into the super situation. 
Since the $D$-action arises from $\rho$,
we may suppose that $v$ is one of  
the entries $t_{ij}$ of $T$. The desired result is then easily verified 
by using \eqref{UEXeq2}. This completes the proof of the claim. 

\medskip

\emph{Step 2}.\
This step aims to construct an extension $L|K$ of $D$-SUSY fields such that $L^D=k$, and $L$ is $D$-simple.

Let us change the notations $\cO_K$, $\cO_k$ and $\op{id}_K\ot \Delta$ above into 
$B$, $J$ and $\theta_B$, respectively, to make their roles clearer. Confirm that 
$B=(B, \theta_B)$
is an algebra in $\M^J$. 
Moreover, $B|K$ is a $J$-Galois extension, whence we have the
isomorphism $\Theta_B :B\ot_KB \overset{\simeq}{\longrightarrow}B\otk J$. Suppose that $D$ acts trivially on $J$. 
One then sees that $\theta_B: B\to B\ot_k J$
is $D$-superlinear, whence $\Theta_B$ is an isomorphism
in ${}_D\M$. Therefore, identifying so as 
$B\ot_KB =B\otk J$ through $\Theta_B$, we have
\eq\label{UEXeq8}
J \subset (B\ot_KB)^D.
\eeq

In ${}_D\M$, choose arbitrarily a $D$-simple quotient algebra $A$ of $B$. 
Notice that $A$ includes $K$, since $K$ is $D$-simple. In addition, $A^D$ is a field
\tc{by Proposition \ref{DSAprop1} (1),} 
since $A$ is $D$-simple. 
We claim 
\eq\label{UEXeq9}
A^D=k. 
\eeq
To see this, notice from Proposition \ref{DSAprop2} that $(k=)\, K^D=\K^{\D}$.  
We have an embedding $k=K^D \subset A^D$ of fields, which then induces $k=\K^{\D}\hookrightarrow \overline{A}^{\D}$.
This prolongs to $k=\K^{\D}\to Q^{\D}$, where $Q$ is an arbitrarily chosen, $\D$-simple 
quotient $\D$-algebra of $\overline{A}$.
Since $Q$ is finitely generated over $\K$, and $\D$ is irreducible, one sees from \cite[Theorem 4.4]{T} that the
field extension $Q^{\D}|\K^{\D}\, (=k)$ is algebraic, whence $Q^{\D}=k$ by the assumption on $k$. Thus the first embedding
prolongs so that
\eq\label{UEXeq10}
k=K^D\subset A^D \hookrightarrow \overline{A}^{\D} \to Q^{\D}=k.
\eeq
This implies \eqref{UEXeq9}.

Set $H:=(A\ot_KA)^D$. 
Since $A$ is $D$-simple and satisfies \eqref{UEXeq9}, we have 
the injection $\mu :A \otk H \to A \ot_KA$ by Proposition \ref{DSAprop1} (2). 
As is seen from \eqref{UEXeq8}, the composite 
\[
B\otk J \overset{\Theta_B^{-1}}{\longrightarrow}B\ot_KB\to A\ot_K A
\]
of $\Theta_B^{-1}$ with the natural projection is seen to factor through the $A\otk H$ in $A\ot_K A$, so that 
$J$ maps into $H$, whence we have the commutative diagram
\begin{equation*}
\begin{xy}
(0,7)   *++{B\otk J}  ="1",
(26,7)  *++{B\ot_K B} ="2",
(0,-8)  *++{A\otk H} ="3",
(26,-8) *++{A \ot_KA.}="4",
{"1" \SelectTips{cm}{} \ar @{->}_{\Theta_B^{-1}}^{\simeq} "2"},
{"3" \SelectTips{cm}{} \ar@{>->}_{\mu} "4"},
{"1" \SelectTips{cm}{} \ar@{->} "3"},
{"2" \SelectTips{cm}{} \ar @{->>} "4"}
\end{xy}
\end{equation*}
It follows that $\mu$ is an isomorphism. 
Just as when we proved Lemma \ref{INDlem3},
we see that $H$ and $A$ uniquely turn into Hopf superalgebra over $k$ and an algebra in $\M^H$, respectively, 
so that Conditions
(HG1--2) in Definition \ref{HGEdef1} are satisfied. 
We see also that $H$ is a quotient Hopf superalgebra of $J$, and is, therefore,
finitely generated over $k$.

Just as when we proved Proposition \ref{INDprop3},
discussing $\widetilde{A}=A/K_1A$, 
we see that $A|K$ is an $H$-Galois extension, and obtain from Theorem \ref{HGEthm1} (2) 
a $K$-superalgebra isomorphism 
$A \simeq K\ot_{\K}\overline{A}\otk\wedge_k(\mathfrak{w}_H)$ such as in \eqref{HGEeq5}. 
Notice from Theorem \ref{HGEthm1} (1) that $\overline{A}|\K$ is $\overline{H}$-Galois, 
whence $\overline{A}$
is smooth over $\K$ by Proposition \ref{HGEprop1}. It follows that $A_{\op{red}}=\overline{A}$, 
which is an integral domain by Proposition \ref{DSAprop3} (2), so that $\Quot(\overline{A})$ is a field. 
Let $L:=\Quot(A)$; it uniquely turns into an $A$-algebra in ${}_D\M$ by Proposition \ref{EXDprop1} (1).
We see from Example \ref{SUFex1} that the last isomorphism extends to 
\[
L\simeq K\ot_{\K} \Quot(\overline{A})\otk \wedge_k(\mathfrak{w}_{H}),
\]
whence $L|K$ is an
extension of $D$-SUSY fields; it is of finite type, since one sees $\L=\Quot(\overline{A})$, and
$\overline{A}$ is finitely generated over $\K$.  
Notice that $L$ is $D$-simple since $A$ is. It remains to prove $L^D=k$. But this follows from
\[
L^D=A^D\, (=k).
\]
Indeed, $L^D=\L^{\D}$ by Proposition \ref{DSAprop2}, $\L^{\D}=\overline{A}^{\D}$ 
by \cite[Theorem 4.4]{T}, again, 
and $A^D=\overline{A}^{\D}$, 
as is then seen from \eqref{UEXeq10}. 

\medskip

\emph{Step 3.}\ 
We wish to conclude that the constructed $L$ is a minimal splitting $D$-SUSY field for $V$.

Recall that $A$ is a quotient algebra of $\cO_K$ in ${}_D\M$. 
Suppose that elements of $A^{m|n}$ are presented as column vectors, and let 
$\boldsymbol{x}_i \in A^{m|n}$ 
denote the natural image of the $i$-th column vector $\boldsymbol{t}_i={}^t(t_{1i},\dots, t_{m+n,i})$ in $T$.
Let $f : V \to A^{m|n}$ be the 
$K$-supermodule morphism defined by $f(v_i)=\boldsymbol{x}_i$, where $1\le i\le m+n$. This is seen from \eqref{UEXeq3} to
be a $K$-module morphism in ${}_D\M$. 
By base extension, $f$ extends to an $A$-module morphism $A\ot_K V \to A^{m|n}$ in ${}_D\M$,
which is, in fact, an isomorphism since the matrix $X:=\big(\boldsymbol{x}_1,\dots, \boldsymbol{x}_{m+n}\big)$ is in 
$\mathsf{GL}_{m|n}(A)$. Therefore, $L$ is a splitting $D$-SUSY field for $V$. It is minimal, 
since $A$ is, as a $K$-superalgebra, generated by the entries of $X$ together with $(\det_0X)^{-1}$.
\epf

\pf[Proof of the uniqueness]
Suppose that $L^{(1)}$ and $L^{(2)}$ are $D$-simple minimal-splitting $D$-SUSY fields such that $(L^{(i)})^D=k$, $i=1,2$. 
By Theorem \ref{CHAthm1} we have SUSY PV extensions
$(L^{(i)}|K, A^{(i)}, H^{(i)})$, $i=1,2$. Now, set 
\[
B:=A^{(1)}\ot_KA^{(2)},\quad J:=H^{(1)}\otk H^{(2)}.
\] 
Then $J$ is a finitely generated Hopf superalgebra over $k$, and $B|K$ is naturally a $J$-Galois extension such that
we have $J \subset (B\ot_K B)^D$, identifying so as $B\ot_KB=B\otk J$ through the isomorphism $\Theta_B$. 
Choose a $D$-simple quotient $D$-superalgebra $A$ of $B$, and let $L:=\Quot(A)$. 
Then we see just as in Step 2 of the preceding proof 
that $H$ is a quotient Hopf superalgebra of $J$, 
$A|K$ an $H$-Galois extension, and 
$L$ is a $D$-simple $D$-SUSY field such that $L^D=k$. 
In particular, we have the isomorphism $\Theta_A : A \ot_K A\overset{\simeq}{\longrightarrow}A\otk H$. 
For $i=1,2$, the embeddings $A^{(i)}\to A^{(1)}\ot_K A^{(2)}$, $H^{(i)}\to H^{(1)}\otk H^{(2)}$, composed with
the natural projections, give algebra morphisms $A^{(i)}\to A$ in ${}_D\M$, and Hopf superalgebra morphisms $H^{(i)}\to H$. Notice that the former $A^{(i)}\to A$ are injective 
by the $D$-simplicity of $A^{(i)}$. The morphisms, together with 
$\Theta_A$ and the base extension $A\ot_{A^{(i)}}$ of $\Theta_{A^{(i)}}$,  
constitute the commutative diagram
\begin{equation*}
\begin{xy}
(0,7)   *++{A\ot_K A^{(i)}}  ="1",
(26,7)  *++{A \otk H^{(i)}} ="2",
(0,-8)  *++{A\ot_K A} ="3",
(26,-8) *++{A \otk H.}="4",
{"1" \SelectTips{cm}{} \ar @{->}^{\simeq} "2"},
{"3" \SelectTips{cm}{} \ar@{->}^{\simeq} "4"},
{"1" \SelectTips{cm}{} \ar@{->} "3"},
{"2" \SelectTips{cm}{} \ar @{->} "4"}
\end{xy}
\end{equation*}
Recall that $A$ is $K$-faithfully flat.
One then sees that $H^{(i)}\to H$ is injective, whence it is faithfully flat by \cite[Theorem 5.9 (1)]{M}. 
It follows that the injective morphism $A^{(i)}\to A$ is (faithfully) flat. 
Hence every $A^{(i)}$-regular element of $A^{(i)}_0$, regarded as an element of $A_0$, is $A$-regular.
Therefore, $A^{(i)}\to A$ uniquely extends to a superalgebra morphism $L^{(i)}\to L$;
this is necessarily an injective algebra-morphism in ${}_D\M$, by Proposition \ref{EXDprop1} (2)
and the $D$-simplicity of $L^{(i)}$; cf. the proof of Theorem \ref{INDthm1}.
Through this we regard so as $L^{(i)}\subset L$. 

We claim that $L^{(1)}=L^{(2)}$ in $L$; this will complete the proof. For $i=1,2$ we have
\[
\op{Hom}_{K\# D}(V, L^{(i)})=\op{Hom}_{K\# D}(V, L),
\]
since the both sides have the same $k$-dimension $m|n$ by Proposition \ref{SPDprop1}. 
Therefore, using the notation \eqref{SPDeq2}, we have
\[
K[V;L^{(i)}]=K[V;L].
\]
This implies
\[
L^{(i)}=\Quot(K[V;L^{(i)}])=\Quot(K[V;L]),
\]
proving the claim.
\epf

\section*{Acknowledgments}
The author was supported by JSPS~KAKENHI, Grant Numbers 20K03552 and 23K03027.


\begin{thebibliography}{99}

\bibitem{A}
K.~Amano,\ \emph{Liouville extensions of artinian simple module algebras},\
Comm.\ Algebra\ {\bf 34} (2006),\
1811--1823.


\bibitem{AM}
A.~Amano,\ A.~Masuoka,\ \emph{Picard-Vessiot extensions of artinian simple module algebras},\
J.~Algebra {\bf 285} (2005), 284--312.

\bibitem{AMT}
A.~Amano,\ A.~Masuoka,\ M.~Takeuchi,\ \emph{Hopf algebraic
approach to Picard-Vessiot theory},\ 
in: M.~Hazewinkel (ed.), \emph{Handbook of Algebra}, vol. 6,\
Elsevier/North-Holland, Amsterdam, 2009, pp. 125--172.

\bibitem{Beukers}
F.~Beukers,\
{\em Differential Galois theory},
in: M. Waldschmidt et al. (eds.), 
{\em From Number Theory to Physics (Les Houches, 1989)},\
Springer, Berlin, 1992. 

\bibitem{CCF}
C.~Carmeli,\ L.~Caston,\ R.~Fioresi,\
\emph{Mathematical foundations of supersymmetry},\
EMS Series of Lectures in Mathematics. European Math. Soc., 
Z\"{u}rich, 2011.


\bibitem{Del}
P.~Deligne,\ 
\emph{Cat\'{e}gories tannakiennes}, in: P. Cartier, et al., (eds.), \emph{Grothendieck Festschrift},\ vol.~2,\
Progress in Math. 87, Birkh\''{a}user,\ Boston,\ 1990,\ pp. 111--195.


\bibitem{GW}
S.~Garnier,\ T.~Wurzbacher,
\emph{Integration of vector fields on smooth
and holomorphic supermanifolds}, 
Doc. Math. {\bf 18} (2013), 519–-545.

\bibitem{HS}
I.~Heckenberger,\ H.-J.~Schneider,\
{\em Hopf Algebras and Root Systems},\ 
Mathematical Surveys and Monographs vol.~247,\ Amer. Math. Soc. Providence, Rhode Island, 2020. 


\bibitem{Kaplansky}
I.~Kaplansky,\ \emph{An Introduction to Differential Algebra},\ 2nd ed.,\ Hermann, Paris,
1976.

\bibitem{Kolchin}
E.~R.~Kolchin,\ \emph{Differential Algebra and Algebraic Groups},\ Academic Press,\ New York,\
1973.

\bibitem{Manin}
Y.~Manin,\
\emph{Gauge field theory and complex geometry},
Translated from the 1984 Russian original by 
N. Koblitz and J. R. King, Second edition, 
Grundlehren der mathematischen Wissenschaften 
vol.~289,\
Springer-Verlag, Berlin, 1997. 

\bibitem{M} 
A.~Masuoka,\
\emph{The fundamental correspondences in super affine groups and
super formal groups},\
J.\ Pure\ Appl.\ Algebra \textbf{202} (2005), 284--312.

\bibitem{M1} 
A.~Masuoka,\ 
{\em Harish-Chandra pairs for algebraic affine supergroup schemes over an arbitrary field},\ 
Transform. Groups \textbf{17} (2012), no. 4, 1085--1121.

\bibitem{MasuokaII} 
A.~Masuoka,\ 
\emph{Supersymmetric Picard-Vessiot theory, II: solvability and examples},\ 
in preparation.

\bibitem{MOT} 
A.~Masuoka,\ T.~Oe,\ Y.~Takahashi,\
\emph{Torsors in super-symmetry},\ preprint $\mathsf{arXiv}$: 2101.03461v3.

\bibitem{MT}
A. Masuoka,\ Y.~Takahashi,\ \emph{Geometric construction of quotients $G/H$ in
supersymmetry},\ Transform. Groups {\bf 26}(1) (2021), 347--375.

\bibitem{MZ1} A.~Masuoka,\ A.~N.~Zubkov,\ 
\emph{Quotient sheaves of algebraic supergroups are superschemes},\ 
J.\ Algebra \textbf{348} (2011), 135--170.

\bibitem{MZ2} A.~Masuoka,\ A.~N.~Zubkov,\ 
\emph{Solvability and nilpotency for algebraic supergroups},\ 
J. Pure Appl. Algebra {\bf 221} (2017),\ 339--365.

\bibitem{Milne} J.~S.~Milne,\
\emph{Algebraic Groups},\
Cambridge Studies in Advanced Mathematics, vol.~170,\ Cambridge Univ. Press, Cambridge, 2017.

\bibitem{MS}
J.~Monterde,\ O.~A.~S\'{a}ntez-Valenzuela,\
{\em Existence and uniqueness of solutions to superdifferential equations},\
J. Geom. Phys. {\bf 10} (1993), 315--343. 

\bibitem{Mon}
S.~Montgomery,\  \emph{Hopf Algebras and Their Actions on Rings},\
CBMS Regional Conference Series in
Mathematics, vol.~82, Amer. Math. Soc., Providence, RI, 1993.

\bibitem{MoralesRuiz}
J.~Morales Ruiz,\
{\em Differential Galois Theory and Non-integrability of Hamiltonian
Systems},\
Progress in Mathematics vol.~\tc{179},\ \tc{Birkh\"{a}user} Verlag, Basel. 1999.

\bibitem{N}, W.~D.~Nichols,\ 
\emph{Quotients of Hopf algebras},\ Comm. Algebra {\bf 6} (17) (1978),\ 1789--1800.

\bibitem{PutSinger}
M.~van~der~Put,\ M.~F.~Singer,\ \emph{Galois Theory of Linear Differential Equations},\
Grundlehren Math. Wiss., vol. 328, Springer, Berlin, 2003.

\bibitem{Sw} 
M.~E.~Sweedler,\ 
\emph{Hopf Algebras},\
W. A. Benjamin, Inc., New York,~1969.

\bibitem{T}
M.~Takeuchi,\
\emph{A Hopf algebraic approach to Picard-Vessiot theory},\
J.~Algebra {\bf 122} (1989),\ 481--509. 


\bibitem{Wei}
C. A.~Weibel,\ \emph{An introduction to homological algebra}, paperback edition, Cambridge
Studies in Advanced Mathematics vol.~38, Cambridge Univ. Press, 1995.

\bibitem{Z}
A.~N.~Zubkov,\
\emph{On some properties of Noether superschemes}, (Russian) Algebra Logika {\bf 57} (2018), no. 2, 197--213; translation in Algebra Logic {\bf 57} (2018), no. 2, 130--140.

\end{thebibliography}
\end{document}